\documentclass[11pt,oneside,reqno]{amsart}
\usepackage[latin9]{inputenc}
\usepackage{geometry}
\usepackage{color}
\usepackage{mathrsfs}
\usepackage{amstext}
\usepackage{amsthm}
\usepackage{amssymb}
\usepackage{graphicx}
\usepackage{setspace}
\usepackage[unicode=true,
 bookmarks=false,
 breaklinks=false,pdfborder={0 0 1},backref=section,colorlinks=false]
 {hyperref}

\makeatletter

\newcommand{\lyxdot}{.}

\numberwithin{equation}{section}
\numberwithin{figure}{section}
\theoremstyle{plain}
\newtheorem{thm}{\protect\theoremname}[section]
\theoremstyle{remark}
\newtheorem{rem}[thm]{\protect\remarkname}
\theoremstyle{plain}
\newtheorem{prop}[thm]{\protect\propositionname}
\theoremstyle{definition}
\newtheorem{defn}[thm]{\protect\definitionname}
\theoremstyle{plain}
\newtheorem{lem}[thm]{\protect\lemmaname}
\theoremstyle{remark}
\newtheorem*{notation*}{\protect\notationname}
\theoremstyle{remark}
\newtheorem{notation}[thm]{\protect\notationname}
\theoremstyle{plain}
\newtheorem{cor}[thm]{\protect\corollaryname}
\theoremstyle{remark}
\newtheorem*{acknowledgement*}{\protect\acknowledgementname}

\usepackage{amsthm}\usepackage{mathrsfs}\usepackage{nameref}\usepackage{upgreek}
\usepackage[noadjust]{cite}
\usepackage{stmaryrd}
\usepackage{etoolbox}
\usepackage{enumitem}
\setlist[itemize]{leftmargin=*}
\setlist[enumerate]{leftmargin=*}

\DeclareFontFamily{U}{matha}{\hyphenchar\font45}
\DeclareFontShape{U}{matha}{m}{n}{
      <5> <6> <7> <8> <9> <10> gen * matha
      <10.95> matha10 <12> <14.4> <17.28> <20.74> <24.88> matha12
      }{}
\DeclareSymbolFont{matha}{U}{matha}{m}{n}
\DeclareFontSubstitution{U}{matha}{m}{n}

\DeclareFontFamily{U}{mathx}{\hyphenchar\font45}
\DeclareFontShape{U}{mathx}{m}{n}{
      <5> <6> <7> <8> <9> <10>
      <10.95> <12> <14.4> <17.28> <20.74> <24.88>
      mathx10
      }{}
\DeclareSymbolFont{mathx}{U}{mathx}{m}{n}
\DeclareFontSubstitution{U}{mathx}{m}{n}

\DeclareMathDelimiter{\vvvert}{0}{matha}{"7E}{mathx}{"17}

\DeclareMathAlphabet{\scal}{U}{dutchcal}{m}{n}

\numberwithin{equation}{section}

\def\th@plain{\thm@notefont{}\itshape}
\def\th@definition{\thm@notefont{}\normalfont}

\makeatother

\providecommand{\acknowledgementname}{Acknowledgement}
\providecommand{\corollaryname}{Corollary}
\providecommand{\definitionname}{Definition}
\providecommand{\lemmaname}{Lemma}
\providecommand{\notationname}{Notation}
\providecommand{\propositionname}{Proposition}
\providecommand{\remarkname}{Remark}
\providecommand{\theoremname}{Theorem}

\begin{document}
\title[Energy Landscape and Metastability of 3D Ising/Potts Models]{Energy Landscape and Metastability of Stochastic Ising and Potts Models on Three-dimensional
Lattices Without External Fields}
\author{Seonwoo Kim and Insuk Seo}
\address{S. Kim. Department of Mathematical Sciences, Seoul National University,
Republic of Korea.}
\email{ksw6leta@snu.ac.kr}
\address{I. Seo. Department of Mathematical Sciences and R.I.M., Seoul National
University, Republic of Korea.}
\email{insuk.seo@snu.ac.kr}
\begin{abstract}
In this study, we investigate the energy landscape of the Ising and
Potts models on fixed and finite but large three-dimensional (3D)
lattices where no external field exists and quantitatively characterize
the metastable behavior of the associated Glauber dynamics in the
very low temperature regime. Such analyses for the models with non-zero
external magnetic fields have been extensively performed over the
past two decades; however, models without external fields remained
uninvestigated. Recently, the corresponding investigation has been
conducted for the two-dimensional (2D) model without an external field,
and in this study, we further extend these successes to the 3D model,
which has a far more complicated energy landscape than the 2D one.
In particular, we provide a detailed description of the highly complex
plateau structure of saddle configurations between ground states and
then analyze the typical behavior of the Glauber dynamics thereon.
Thus, we acheive a quantitatively precise analysis of metastability,
including the Eyring--Kramers law, the Markov chain model reduction,
and a full characterization of metastable transition paths. 
\end{abstract}

\maketitle
\begin{figure}
\begin{centering}
\includegraphics[width=11cm]{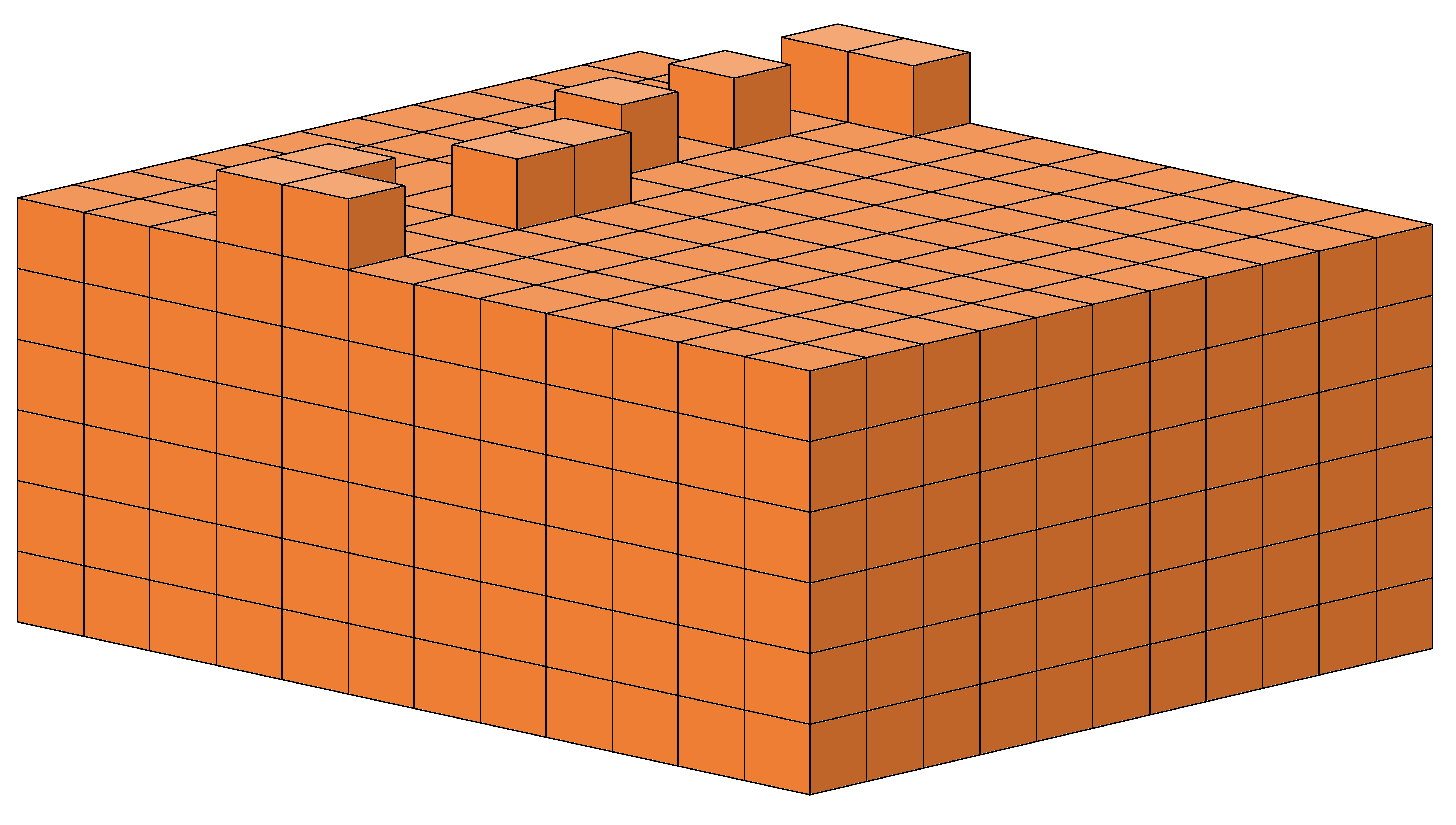}
\par\end{centering}
\medskip
\centering{}\textit{Example of a three-dimensional saddle configuration}
\end{figure}

\thispagestyle{empty}

\newpage

\section{\label{sec1}Introduction}

Metastability is a ubiquitous phenomenon that arises when a stochastic
system has several locally stable sets; it is observed in a wide class
of models, e.g., in the small random perturbations of dynamical systems
(e.g., \cite{BEGK,BGK,FW RPDS,LMS DT,LS NR2,LeeS NR1,LeeS NR2,LPM,RezSeo}),
interacting particle systems consisting of sticky particles (e.g.,
\cite{BL ZRP,BDG,Kim,KS1,LMS cri,LMS res,OhRez,Seo NRZRP}), and spin
systems in the low temperature regime (e.g., \cite{AC,BAC,BBI,BdenH meta,BdenHN,BdenHS,BM,CGOV,CO,Kim2,LL,LL 19,LS Potts,Lee,NS Ising1,NS Ising2,NZ}).
Numerous important works are not listed here; we direct the references
of the monographs \cite{BdenH meta,OV meta}, which provide a comprehensive
introduction to this broad topic.

\subsubsection*{Metastable behaviors of stochastic Ising and Potts models}

In this study, we consider the metastability of the stochastic Ising
and Potts models evolving according to Metropolis--Hastings-type
Glauber dynamics on a large, but fixed three-dimensional (3D) lattice.
For such models, the Gibbs invariant measure is exponentially concentrated
on monochromatic configurations (i.e., the configurations consisting
of a single spin, which are the ground states of the Ising and Potts
Hamiltonians) in the very low temperature regime. Hence, in such regimes,
the dynamics exhibits metastable behavior between the monochromatic
configurations: It starts from a monochromatic configuration, remains
in a certain neighborhood of the starting configuration for an exponentially
long time, and finally overcomes the energy barrier between monochromatic
configurations to reach another monochromatic one.

Several mathematical questions persist regarding the metastable behavior
explained above. For instance, in the transition from one monochromatic
configuration to another, the mean transition time, the asymptotic
law of the rescaled transition time, and the typical transition paths
are all points of interest. We are also interested in the characterization
of the energy barrier and the saddle configurations that realize this
energy barrier via optimal paths between monochromatic configurations.
The final issue is particularly important and challenging for the
model considered in the present article and has remained open for
a long time. It is also important to estimate the mixing time or spectral
gap of the associated dynamics; this allows us to measure the effects
of metastable behavior on the global mixing properties of the associated
Markovian dynamics. \emph{In this article, we answer all these questions
for the stochastic Ising and Potts models on finite three-dimensional
lattices in the absence of external fields.}

\subsubsection*{Model with non-zero external field}

The first rigorous mathematical treatment of the metastable behavior
of the Ising model was performed in \cite{NS Ising1,NS Ising2}, where
the authors considered the Ising model on a two-dimensional (2D) lattice
in the presence of a \emph{non-zero} external field. These studies
verified that the transition from a metastable monochromatic configuration
to a stable one is essentially equivalent to the formation of a certain
type of critical droplet. From this observation, precise information
regarding the transition path was obtained, as well as large deviation-type
estimates for the transition time and mixing behavior associated with
the Metropolis--Hastings dynamics. This result was extended to the
3D Ising model presented in \cite{AC,BAC}. Similar results for four-
or higher-dimensional models remain to be found, because the variational
problems related to the analysis of the energy landscape and critical
droplet are highly complicated. 

In \cite{BM}, the aforementioned analyses were further refined via
the \textit{potential-theoretic approach} developed in \cite{BEGK}.
In \cite{BM}, the authors obtained the Eyring--Kramers law for the
transition time between monochromatic configurations, as well as the
spectral gap of the associated dynamics. This new technology does
not provide information on the transition path; however, it provides
precise asymptotics for the mean metastable transition time and spectral
gap. The same model on growing lattice boxes, rather than fixed ones,
was investigated in \cite{BdenHS}, and the Kawasaki-type (instead
of Glauber-type) dynamics for the same model were studied in \cite{BdenHN}.

\subsubsection*{Model without external field}

When studying the metastability of the stochastic Ising model with
a non-zero external field (as described above), the crucial object
is the critical droplet, which provides a sharp saddle structure for
the energy landscape. However, in the zero external field case, the
critical droplet does not exist. Instead, the saddle structure is
flat, structurally complex, and composed of a large set of saddle
configurations. This is the crucial challenge in the zero external
field case, which has left the problem unsolved for a long time. \textit{In
the present study, we solve this problem by comprehensively analyzing
the energy landscape.}

Recently, \cite{NZ} analyzed for the first time the 2D Ising and
Potts models in the absence of external fields. More precisely, they
characterized (1) the energy barrier between ground states and (2)
the deepest metastable valleys in the landscape. Using the energy
landscape results and a general tool referred to as the \textit{pathwise
approach} to metastability (developed in \cite{CGOV,CNSoh,MNOS,NZB}),
they obtained large deviation-type results for the metastable behaviors
of the 2D models in the absence of external fields. 

In \cite{KS 2D}, which is a companion article of the present one,
we improved on the refinement of results in the previous studies for
the 2D model using the potential-theoretic approach, thereby making
the following contributions:
\begin{itemize}
\item the Eyring--Kramers law for metastable transitions between monochromatic
configurations,
\item the Markov chain model reduction of metastable behavior (cf. \cite{Landim MMC}
for a comprehensive review on this method), and
\item the full characterization of typical transition paths.
\end{itemize}
To this end, we derive a highly detailed analysis of the energy landscape
and characterize \textbf{\emph{all}} saddle configurations. In particular,
we comprehensively and precisely describe the large and complicated
saddle structure of the model. Our analysis is sufficiently accurate
to allow the transition paths between ground states to be characterized
explicitly.

\subsubsection*{Main achievement}

In the current article, \textbf{\emph{we extend all these analyses
to the 3D Ising and Potts models by combining the pathwise approach
and the potential-theoretic approach.}} Indeed, the energy landscape
of the 3D model is significantly more complicated than that of the
2D model. For both the 2D and 3D models, there are numerous saddle
configurations between ground states, and they form a plateau structure.
For the 2D model, at least the bulk part of this plateau structure
is relatively simple, because each saddle configuration can only move
forward or backward to reach another saddle configuration. In contrast,
for the 3D model, we cannot expect such a simplification, because
there exist certain configurations for which the legitimate movements
between saddle configurations can occur in a substantially more complex
manner. We refer to the figure at the front page for an example of
a highly complicated saddle configuration in the 3D case (which should
be characterized in some way to answer all the questions above). Readers
who are familiar with the results on the non-zero external field model
can notice from this figure that the saddle configurations for the
zero external model may not have a clear structure as in the non-zero
external field case.

\subsubsection*{Approximation method to metastability}

In our companion paper \cite{KS 2D}, we introduced a new approximation
method to prove the Eyring--Kramers law and Markov chain model reduction.
This method relies on the approximation of the equilibrium potential
function (refer to Section \ref{sec3.1} for the precise definition)
in a Sobolev space defined via the Dirichlet norm associated with
the Markov chain. It is robust and particularly suitable if the energy
landscape is too complex to apply the potential-theoretic approach
\cite{BEGK} via variational principles (the Dirichlet and Thomson
principles), because it effectively avoids these variational principles
via an approximation in the Sobolev space. We apply this method to
the 3D model to achieve our main result.

The main mathematical difficulty of applying this method lies in the
fact that we must construct a test function that accurately approximates
the equilibrium potential function so that we can obtain the precise
Sobolev norm. For this procedure, we need a comprehensive understanding
of the whole energy landscape regarding the metastable transitions.
Thus, compared with the 2D model, the corresponding construction for
the 3D model is far more complicated. Overcoming this difficulty is
the main contribution of the present study.

\section{\label{sec2}Main Results}

\subsection{\label{sec2.1}Models}

In this subsection, we introduce the stochastic Ising and Potts models
on a fixed 3D lattice and review their basic features.

\subsubsection*{Ising and Potts models}

We fix three positive integers $K\le L\le M$. Then, we denote by
\[
\Lambda=\llbracket1,\,K\rrbracket\times\llbracket1,\,L\rrbracket\times\llbracket1,\,M\rrbracket
\]
the 3D lattice box. We use the notation $\llbracket a,\,b\rrbracket=[a,\,b]\cap\mathbb{Z}$
throughout this article. We impose either open or periodic boundary
conditions upon the lattice box $\Lambda$. For the latter boundary
condition, we can write 
\begin{equation}
\Lambda=\mathbb{T}_{K}\times\mathbb{T}_{L}\times\mathbb{T}_{M}\;,\label{e_torus}
\end{equation}
where $\mathbb{T}_{k}=\mathbb{Z}/(k\mathbb{Z})$ represents the discrete
one-dimensional torus.

For an integer $q\ge2$, we use $S=\{1,\,\dots,\,q\}$ to represent
the set of spins and $\mathcal{X}=S^{\Lambda}$ to represent the space
of spin configurations in the 3D box $\Lambda$. We express a configuration
$\sigma\in\mathcal{X}$ as $\sigma=(\sigma(x))_{x\in\Lambda}$, where
$\sigma(x)\in S$ represents the spin of $\sigma$ at site $x\in\Lambda$.

For $x,\,y\in\Lambda$, we write $x\sim y$ if they are neighboring
sites; that is, $\Vert x-y\Vert=1$ where $\Vert\cdot\Vert$ denotes
the Euclidean distance in $\Lambda$. With this notation, we define
the Hamiltonian $H:\mathcal{X}\rightarrow\mathbb{R}$ as
\begin{equation}
H(\sigma)=\sum_{\{x,\,y\}\subseteq\Lambda:\,x\sim y}\mathbf{1}\{\sigma(x)\ne\sigma(y)\}-h\sum_{x\in\Lambda}\sigma(x)\;\;\;\;;\;\sigma\in\mathcal{X}\;,\label{e_Ham}
\end{equation}
where $h\in\mathbb{R}$ denotes the magnitude of the external magnetic
field. Thus, the first summation on the right-hand side represents
the spin--spin interactions, and the second one corresponds to the
effect of the external magnetic field. We use $\mu_{\beta}(\cdot)$
to denote the Gibbs measure on $\mathcal{X}$ associated with the
Hamiltonian $H$ at inverse temperature $\beta>0$; that is, 
\begin{equation}
\mu_{\beta}(\sigma)=\frac{1}{Z_{\beta}}e^{-\beta H(\sigma)}\;\;\;\;;\;\sigma\in\mathcal{X}\;,\label{e_mu}
\end{equation}
where $Z_{\beta}=\sum_{\zeta\in\mathcal{X}}e^{-\beta H(\zeta)}$ is
the partition function. The random spin configuration on box $\Lambda$
corresponds to the probability measure $\mu_{\beta}(\cdot)$ on $\mathcal{X}$;
it is referred to as the \textit{Ising model} if $q=2$ and the \textit{Potts
model} if $q\ge3$. Henceforth, we treat $q$ as a fixed parameter.
Our primary concern is the metastability analyses of these models
as $\beta\rightarrow\infty$ under Metropolis--Hastings dynamics,
which will be defined precisely below.
\begin{rem}[Results for non-zero external field]
 Comprehensive analyses of the energy landscape and the metastability
of the Ising model with a non-zero external field i.e., $h\neq0$,
were performed in \cite{BM,NS Ising1,NS Ising2} for the 2D case,
and in \cite{AC,BAC} for the 3D one. For these models, the characterization
of the \textit{critical droplet} comprehensively explains the metastable
behavior. We remark that analysis for cases of more than three dimensions
has yet to be undertaken, because the energy landscape is too complex
to allow critical droplets to be characterized. Recently, the 2D Potts
model with an external field toward one specific spin has been studied
\cite{BGN,BGN Neg,BGN Pos}.
\end{rem}

In this study, we consider the zero external field case (i.e., $h=0$);
thus, \textbf{\textit{we henceforth assume that}} \textbf{\textit{$h=0$.}}
This case differs from those involving non-zero external fields, in
the sense that the energy landscape is not characterized by critical
droplets. Instead, we must tackle a large and complex landscape of
saddle configurations via complicated combinatorial and probabilistic
arguments.

\subsubsection*{Ground states}

For each $a\in S$, denote by $\mathbf{s}_{a}\in\mathcal{X}$ the
monochromatic configuration in which all spins are $a$, i.e., $\mathbf{s}_{a}(x)=a$
for all $x\in\Lambda$. We write 
\begin{equation}
\mathcal{S}=\{\mathbf{s}_{1},\,\dots,\,\mathbf{s}_{q}\}\;.\label{e_S}
\end{equation}
It is precisely upon $\mathcal{S}$ that the Hamiltonian $H(\cdot)$
attains its minimum $0$; hence, $\mathcal{S}$ represents the set
of ground states of the model. Accordingly, we obtain the following
characterization of the partition function $Z_{\beta}$ that appears
in \eqref{e_mu}, as well as the Gibbs measure $\mu_{\beta}$ as $\beta\rightarrow\infty$. 
\begin{thm}
\label{t_Zbest}We have\footnote{For two collections $(a_{\beta})_{\beta>0}=(a_{\beta}(K,\,L,\,M))_{\beta>0}$
and $(b_{\beta})_{\beta>0}=(b_{\beta}(K,\,L,\,M))_{\beta>0}$ of real
numbers, we write $a_{\beta}=O_{\beta}(b_{\beta})$ if there exists
some $C=C(K,\,L,\,M)>0$ such that
\[
|a_{\beta}|\le Cb_{\beta}\text{ for all }\beta>0\text{ and }K,\,L,\,M\;.
\]
}
\begin{equation}
Z_{\beta}=q+O_{\beta}(e^{-3\beta})\;.\label{e_Zbest}
\end{equation}
Thus, we obtain
\[
\lim_{\beta\rightarrow\infty}\mu_{\beta}(\mathbf{s})=\frac{1}{q}\text{ for all }\mathbf{s}\in\mathcal{S}\;\;\;\;\text{and}\;\;\;\;\lim_{\beta\rightarrow\infty}\mu_{\beta}(\mathcal{S})=1\;.
\]
\end{thm}

\begin{proof}
The estimate \eqref{e_Zbest} of the partition function comes directly
from the expression of the partition function given right after \eqref{e_mu}
and the fact that $H(\sigma)\ge3$ for $\sigma\notin\mathcal{S}$.
The second assertion of the theorem is directly derived from the first
one and the expression \eqref{e_mu} of $\mu_{\beta}$.
\end{proof}

\subsubsection*{Metropolis--Hastings dynamics and metastability}

We give a continuous version of the Metropolis--Hastings dynamics,
which is the standard heat-bath Glauber dynamics used for studying
the metastability of the Ising model \cite{NS Ising1}. For $x\in\Lambda$
and $a\in S$, we use $\sigma^{x,\,a}\in\mathcal{X}$ to denote the
configuration obtained from $\sigma$ by updating the spin at site
$x$ to $a$. Then, the continuous version of the Metropolis--Hastings
dynamics is defined as a continuous-time Markov chain $\{\sigma_{\beta}(t)\}_{t\ge0}$
on $\mathcal{X}$, whose transition rates are given by
\[
r_{\beta}(\sigma,\,\zeta)=\begin{cases}
e^{-\beta[H(\zeta)-H(\sigma)]_{+}} & \text{if }\zeta=\sigma^{x,\,a}\ne\sigma\text{ for some }x\in\Lambda\text{ and }a\in S\;,\\
0 & \text{otherwise}\;,
\end{cases}
\]
where $[\alpha]_{+}=\max\,\{\alpha,\,0\}$. We notice from this definition
of the rate $r_{\beta}(\cdot,\,\cdot)$ that the Metropolis--Hastings
dynamics tends to lower the energy, particularly when $\beta$ is
large, because the jump rate from one configuration to another one
with higher energy is exponentially small, whereas the jump rate to
another one with lower or equal energy is $1$. We let $\mathbb{P}_{\sigma}^{\beta}$
and $\mathbb{E}_{\sigma}^{\beta}$ represent the law and expectation,
respectively, of the process $\sigma_{\beta}(\cdot)$ starting from
$\sigma$. 

For $\sigma,\,\zeta\in\mathcal{X}$, we write $\sigma\sim\zeta$ if
$r_{\beta}(\sigma,\,\zeta)>0$. Note that $\sigma\sim\zeta$ if and
only if $\zeta\sim\sigma$, and that the relation $\sigma\sim\zeta$
does not depend on $\beta$. A crucial observation regarding the rate
$r_{\beta}(\cdot,\,\cdot)$ defined above is that 
\begin{equation}
\mu_{\beta}(\sigma)\,r_{\beta}(\sigma,\,\zeta)=\mu_{\beta}(\zeta)\,r_{\beta}(\zeta,\,\sigma)=\begin{cases}
\min\,\{\mu_{\beta}(\sigma),\,\mu_{\beta}(\zeta)\} & \text{if }\sigma\sim\zeta\;,\\
0 & \text{otherwise\;.}
\end{cases}\label{e_muprop}
\end{equation}
From this detailed balance condition, we observe that the invariant
measure for the Metropolis--Hastings dynamics $\sigma_{\beta}(\cdot)$
is $\mu_{\beta}(\cdot)$ and that $\{\sigma_{\beta}(t)\}_{t\ge0}$
is reversible with respect to $\mu_{\beta}(\cdot)$. We also note
that the Markov chain $\sigma_{\beta}(\cdot)$ is irreducible.

In view of Theorem \ref{t_Zbest}, we anticipate that the process
$\sigma_{\beta}(\cdot)$ will exhibit \textit{metastable behavior
between ground states, provided that $\beta$ is sufficiently large}.
More precisely, the process $\sigma_{\beta}(\cdot)$ starting from
configuration $\mathbf{s}\in\mathcal{S}$ remains in a certain neighborhood
of $\mathbf{s}$ for a sufficiently long time, and then undergoes
a rare but rapid transition to another ground state. Our main concern
is to precisely analyze such metastability of the stochastic Ising
and Potts models under the Metropolis--Hastings dynamics (defined
above) in the very low temperature regime; that is, when $\beta\rightarrow\infty$.
We explain these results in the following subsection.
\begin{rem}
We employ the continuous-time dynamics (as applied in numerous previous
studies) because it offers a simpler presentation than the corresponding
discrete dynamics (as demonstrated in \cite{BAC,BM,NZ}), for which
the jump probability is given by
\begin{equation}
p_{\beta}(\sigma,\,\zeta)=\begin{cases}
\frac{1}{q|\Lambda|}\,e^{-\beta[H(\zeta)-H(\sigma)]_{+}} & \text{if }\zeta=\sigma^{x,\,a}\ne\sigma\text{ for some }x\in\Lambda\;,\;a\in S\;,\\
1-\sum_{x\in\Lambda,\,a\in S:\,\sigma^{x,a}\ne\sigma}p_{\beta}(\sigma,\,\sigma^{x,\,a}) & \text{if }\zeta=\sigma\;,\\
0 & \text{otherwise}\;.
\end{cases}\label{e_discrete}
\end{equation}
However, our computations can be applied to this model as well. See
also Remark \ref{r_discrete}.
\end{rem}

\subsection{\label{sec2.2}Main results: large deviation-type results}

Hereafter, we explain our results regarding the metastability of the
stochastic Ising and Potts models. In the current subsection, we explain
the large deviation-type results obtained for the metastable behavior.

\subsubsection*{Energy barrier between ground states}

First, we introduce the energy barrier associated with the Ising and
Potts models considered in this study. This is important for the analysis
of metastable behaviors, in that the Metropolis--Hastings dynamics
must overcome this energy barrier to make a transition from one ground
state to another.

A sequence of configurations $(\omega_{t})_{t=0}^{T}=(\omega_{0},\,\omega_{1},\,\dots,\,\omega_{T})\subseteq\mathcal{X}$
for some integer $T\ge0$ is called a \textit{path} if $\omega_{t}\sim\omega_{t+1}$
(i.e., $r_{\beta}(\omega_{t},\,\omega_{t+1})>0$) for all $t\in\llbracket0,\,T-1\rrbracket$.
We say that this path connects two configurations $\sigma$ and $\zeta$
if $\omega_{0}=\sigma$ and $\omega_{T}=\zeta$ or vice versa. The
\textsl{communication height} between two configurations $\sigma,\,\zeta\in\mathcal{X}$\footnote{By writing $a,\,b\in A$, we implicitly state that $a$ and $b$ are
different.} is defined as
\[
\Phi(\sigma,\,\zeta)=\min_{(\omega_{t})_{t=0}^{T}}\,\max_{t\in\llbracket0,\,T\rrbracket}\,H(\omega_{t})\;,
\]
where the minimum is taken over all paths connecting $\sigma$ and
$\zeta$. Moreover, for two disjoint subsets $\mathcal{P}$ and $\mathcal{Q}$
of $\mathcal{X}$, we define
\[
\Phi(\mathcal{P},\,\mathcal{Q})=\min_{\sigma\in\mathcal{P},\,\zeta\in\mathcal{Q}}\,\Phi(\sigma,\,\zeta)\;.
\]
Then, we define
\[
\Gamma=\Gamma(K,\,L,\,M)=\Phi(\mathbf{s},\,\mathbf{s}')\;\;\;\;;\;\mathbf{s},\,\mathbf{s}'\in\mathcal{S}\;.
\]
Note that $\Phi(\mathbf{s},\,\mathbf{s}')$ does not depend on the
selections of $\mathbf{s},\,\mathbf{s}'\in\mathcal{S}$, owing to
the model symmetry. Additionally, note that $\Gamma$ represents the
\textit{energy barrier} between ground states, because the dynamics
must overcome this energy level to make a transition from one ground
state to another. 

To characterize the energy barrier, we must check the maximum energy
of all paths connecting the ground states. Thus, the energy barrier
is a global feature of the energy landscape, and characterizing it
is a non-trivial task. For the current model, we can identify the
exact value of the energy barrier. Recall that we assumed $K\le L\le M$.
\begin{thm}
\label{t_energy barrier}For all sufficiently large $K$, it holds
that
\begin{equation}
\Gamma=\begin{cases}
2KL+2K+2 & \text{under periodic boundary conditions}\;,\\
KL+K+1 & \text{under open boundary conditions}\;.
\end{cases}\label{e_Gamma}
\end{equation}
\end{thm}

\begin{rem}
\label{r_K}Our arguments state that this theorem holds for $K\ge2829$,
where the threshold $2829$ may be sub-optimal (cf. Remark \ref{r_Kcond}).
However, the optimality of this threshold is a minor issue, because
our main concern is the spin system on large boxes.\textit{ Henceforth,
we assume that $K$ satisfies this condition, i.e., $K\ge2829$.}
\end{rem}

Theorem \ref{t_energy barrier} is proved in Section \ref{sec8}.
\begin{rem}
Several remarks regarding the previous theorem are in order.
\begin{enumerate}
\item Note that Theorem \ref{t_energy barrier} does not depend on the value
of $q$, because in the transition from $\mathbf{s}_{a}$ to $\mathbf{s}_{b}$
for $a,\,b\in S$, no spins besides $a$ and $b$ play a significant
role. 
\item Suppose temporarily that $\Gamma_{d}$ is the energy barrier, defined
in the same way as above, subjected to Ising/Potts models defined
on a $d$-dimensional lattice box of size $K_{1}\times\cdots\times K_{d}$
with $K_{1}\le\cdots\le K_{d}$. Then, we expect that $\Gamma_{d}=2+2\sum_{n=1}^{d-1}\prod_{i=1}^{n}K_{i}$
under periodic boundary conditions and $\Gamma_{d}=1+\sum_{n=1}^{d-1}\prod_{i=1}^{n}K_{i}$
under open boundary conditions for all $d\ge2$. Notice that the case
of $d=2$ is handled in \cite[Theorem 1.1]{NZ} and the case of $d=3$
is handled in Theorem \ref{t_energy barrier}. We leave the verification
of this conjecture for the case of $d\ge4$ as a future research problem.
\end{enumerate}
\end{rem}

\subsubsection*{Comparison with non-zero external field case }

We conclude this energy barrier discussion by comparing our results
for the zero external field case with those for the non-zero external
field case obtained in \cite{NS Ising1} and \cite{BAC} for the Ising
model (i.e., $q=2$) in two or three dimensions, respectively. More
precisely, they showed that the energy barrier is given by (under
some technical assumptions regarding $h$) 
\begin{align*}
\Gamma_{2}(h)= & \,4\ell_{h}-h[\ell_{h}(\ell_{h}-1)+1]\;,\\
\Gamma_{3}(h)= & \,2m_{h}(2m_{h}-\delta_{h}-1)+2(m_{h}-\delta_{h})(m_{h}-1)+4\ell_{h}\\
 & \,-h[m_{h}(m_{h}-\delta_{h})(m_{h}-1)+\ell_{h}(\ell_{h}-1)+1]\;,
\end{align*}
where $\Gamma_{d}$ represents the $d$-dimensional energy barrier,
$\ell_{h}=\lceil2/|h|\rceil$, $m_{h}=\lceil4/|h|\rceil$, and $\delta_{h}\in\{0,\,1\}$
is a constant depending only on $h$ (provided that the lattice is
sufficiently large). We refer to \cite[Chapter 17]{BdenH meta} for
details. These energy barriers are characterized by the energy of
the critical droplet, and \textit{their values do not depend on the
size of the box but are determined solely by the magnitude $h$ of
the external field.} This is primarily because the size of the critical
droplet is determined solely by $|h|$, and the size of the box plays
no role provided that the box is sufficiently large to contain a single
droplet. In contrast, the zero external field case does not feature
such a critical droplet; hence, the magnitude of the energy barrier
depends crucially on the box size. This is the key difference between
the zero external field and non-zero external field cases.

\subsubsection*{Large deviation-type results based on pathwise approach}

Here, we explain the large deviation-type analysis of the metastable
behavior of the Metropolis--Hastings dynamics. These results can
be obtained via the pathwise approach developed in \cite{CGOV}, provided
that we can analyze the model energy landscape to a certain degree
of precision. We refer to the monograph \cite{OV meta} for an extensive
summary of the pathwise approach. This approach allows us to analyze
the metastability from three different perspectives: transition time,
spectral gap, and mixing time. All these quantities are crucial for
quantifying the metastable behavior. First, we explicitly define them
as follows:
\begin{itemize}
\item For $\mathcal{A}\subseteq\mathcal{X}$, we denote by $\tau_{\mathcal{A}}=\inf\,\{t\ge0:\sigma_{\beta}(t)\in\mathcal{A}\}$
the hitting time of the set $\mathcal{A}$. If $\mathcal{A}=\{\sigma\}$
is a singleton, we write $\tau_{\{\sigma\}}=\tau_{\sigma}$. 
\item For $\mathbf{s}\in\mathcal{S}$, we write $\breve{\mathbf{s}}=\mathcal{S}\setminus\{\mathbf{s}\}$.
Then, our primary concern is the hitting time $\tau_{\breve{\mathbf{s}}}$
or $\tau_{\mathbf{s}'}$ for $\mathbf{s}'\in\breve{\mathbf{s}}$ when
the dynamics starts from $\mathbf{s}\in\mathcal{S}$. We refer to
this as the \textit{(metastable) transition time}, because it expresses
the time required for a transition to proceed from the ground state
to another one.
\item The \textit{mixing time} corresponding to the level $\epsilon\in(0,\,1)$
is defined as
\[
t_{\beta}^{\mathrm{mix}}(\epsilon)=\min\,\big\{\,t\ge0:\max_{\sigma\in\mathcal{X}}\,\Vert\mathbb{P}_{\sigma}^{\beta}\,[\sigma_{\beta}(t)\in\cdot]-\mu_{\beta}(\cdot)\Vert_{\textrm{TV}}\le\epsilon\,\big\}\;,
\]
where $\Vert\cdot\Vert_{\textrm{TV}}$ represents the total variation
distance between measures (cf. \cite[Chapter 4]{LPW MCMT}).
\item We denote by $\lambda_{\beta}$ the \textit{spectral gap} of the Metropolis--Hastings
dynamics defined in Section \ref{sec2.1}.
\end{itemize}
The 2D version of the following theorem was established in \cite{NZ}
using the refined pathwise approach developed in \cite{CNSoh,MNOS,NZB}.
We extend their results to the 3D model.
\begin{thm}
\label{t_LDT results}The following statements hold.
\begin{enumerate}
\item \textbf{(Transition time)} For all $\mathbf{s},\,\mathbf{s}'\in\mathcal{S}$
and $\epsilon>0$, we have
\begin{align}
 & \lim_{\beta\rightarrow\infty}\mathbb{P}_{\mathbf{s}}^{\beta}\,[e^{\beta(\Gamma-\epsilon)}<\tau_{\breve{\mathbf{s}}}\le\tau_{\mathbf{s}'}<e^{\beta(\Gamma+\epsilon)}]=1\;,\label{e_nz1}\\
 & \lim_{\beta\rightarrow\infty}\frac{1}{\beta}\log\mathbb{E}_{\mathbf{s}}^{\beta}\,[\tau_{\breve{\mathbf{s}}}]=\lim_{\beta\rightarrow\infty}\frac{1}{\beta}\log\mathbb{E}_{\mathbf{s}}^{\beta}\,[\tau_{\mathbf{s}'}]=\Gamma\;.\label{e_nz2}
\end{align}
Moreover, under $\mathbb{P}_{\mathbf{s}}^{\beta}$, as $\beta\rightarrow\infty$,
\begin{equation}
\frac{\tau_{\breve{\mathbf{s}}}}{\mathbb{E}_{\mathbf{s}}^{\beta}\,[\tau_{\breve{\mathbf{s}}}]}\;,\;\frac{\tau_{\mathbf{s}'}}{\mathbb{E}_{\mathbf{s}}^{\beta}\,[\tau_{\mathbf{s}'}]}\rightharpoonup\mathrm{Exp}(1)\;,\label{e_nz3}
\end{equation}
where $\mathrm{Exp}(1)$ is the exponential random variable with a
mean value of $1$.
\item \textbf{(Mixing time)} For all $\epsilon\in(0,\,1/2)$, the mixing
time satisfies 
\[
\lim_{\beta\rightarrow\infty}\frac{1}{\beta}\log t_{\beta}^{\mathrm{mix}}(\epsilon)=\Gamma\;.
\]
\item \textbf{(Spectral gap)} There exist two constants $0<c_{1}\le c_{2}$
such that 
\[
c_{1}\,e^{-\beta\Gamma}\le\lambda_{\beta}\le c_{2}\,e^{-\beta\Gamma}\;.
\]
\end{enumerate}
\end{thm}

\begin{rem}
\label{r_LDT results}The above theorem holds under both open and
periodic boundary conditions. 
\end{rem}

Theorem \ref{t_LDT results} states that the metastable transition
time, mixing time, and inverse spectral gap become exponentially large
as $\beta\rightarrow\infty$, and their exponential growth rates are
determined by the energy barrier $\Gamma$.

The robust methodology developed in \cite{CNSoh,MNOS,NZB} implies
that characterizing the energy barrier between ground states and identifying
all the deepest valleys suffice (up to several technical issues) to
confirm the results presented in Theorem \ref{t_LDT results}. In
\cite{NZ}, the authors performed corresponding analyses of the energy
landscape; then, they used this robust methodology to prove Theorem
\ref{t_LDT results} for two dimensions. We perform the corresponding
analysis of the energy landscape for the 3D model as well in Sections
\ref{sec6}, \ref{sec7}, and \ref{sec8}. The proof of Theorem \ref{t_LDT results}
is given in Section \ref{sec8.3}. Analysis of the energy landscape
is far more difficult than that of the 2D one considered in \cite{KS 2D}
for several reasons. Details are presented at the beginning of Section
\ref{sec6}.

\subsubsection*{Characterization of transition path}

Our analysis of the energy landscape is sufficiently precise to characterize
all the possible transition paths between ground states in a high
level of detail. The \textit{transition paths} are rigorously defined
in Definition \ref{d_transpath}; we do not present explicit definitions
here, because we would have to define a large amount of notation.
The following theorem asserts that, with dominating probability, the
Metropolis--Hastings dynamics evolves along one of the transition
paths when a transition occurs from one ground state to another.
\begin{thm}
\label{t_transpath}For all $\mathbf{s}\in\mathcal{S}$, we have\footnote{A collection $(a_{\beta})_{\beta>0}=(a_{\beta}(K,\,L,\,M))_{\beta>0}$
of real numbers is written as $a_{\beta}=o_{\beta}(1)$ if 
\[
\lim_{\beta\rightarrow\infty}a_{\beta}=0\text{ for all }K,\,L,\,M\;.
\]
} 
\begin{align*}
 & \mathbb{P}_{\mathbf{s}}^{\beta}\,\big[\,\exists0<t_{1}<\cdots<t_{N}<\tau_{\breve{\mathbf{s}}}\text{ such that }(\sigma_{\beta}(t_{n}))_{n=1}^{N}\text{ is a transition path between }\mathbf{s}\text{ and }\breve{\mathbf{s}}\,\big]\\
 & =1-o_{\beta}(1)\;.
\end{align*}
\end{thm}

The characterization of the transition paths and the proof of this
theorem are given in Section \ref{sec9.4}. 

\subsection{\label{sec2.3}Main results: Eyring--Kramers law and Markov chain
model reduction}

The following results constitute more quantitative analyses of the
metastable behavior obtained using potential-theoretic methods. In
particular, we obtain the Eyring--Kramers law (which is a considerable
refinement of \eqref{e_nz2}) and the Markov chain model reduction
of metastable behavior in the sense of \cite{BL TM,BL TM2}.

For these results, we require an accurate understanding of the energy
landscape and the behavior of the Metropolis--Hastings dynamics on
a large set of saddle configurations between ground states. We conduct
these analyses in Sections \ref{sec9} and \ref{sec10}.

We further remark that the quantitative results given below depend
on the selection of boundary condition, in contrast to Theorems \ref{t_LDT results}
and \ref{t_transpath} (cf. Remark \ref{r_LDT results}). For brevity,
we assume periodic boundary conditions throughout this subsection.
We can treat the open boundary case in a similar manner; the results
and a sketch of the proof are presented in Section \ref{sec11}.

\subsubsection*{Eyring--Kramers law}

The following result constitutes a refinement of \eqref{e_nz2} (and
hence of \eqref{e_nz3}) that allows us to pin down the sub-exponential
prefactor associated with the large deviation-type exponential estimates
of the mean transition time between ground states.
\begin{thm}
\label{t_EK}There exists a constant $\kappa=\kappa(K,\,L,\,M)>0$
such that for all $\mathbf{s},\,\mathbf{s}'\in\mathcal{S}$,
\begin{equation}
\mathbb{E}_{\mathbf{s}}^{\beta}\,[\tau_{\breve{\mathbf{s}}}]=(1+o_{\beta}(1))\,\frac{\kappa}{q-1}\,e^{\Gamma\beta}\;\;\;\;\text{and}\;\;\;\;\mathbb{E}_{\mathbf{s}}^{\beta}\,[\tau_{\mathbf{s}'}]=(1+o_{\beta}(1))\,\kappa\,e^{\Gamma\beta}\;.\label{e_EK}
\end{equation}
Moreover, the constant $\kappa$ satisfies 
\begin{align}
\lim_{K\rightarrow\infty}\kappa(K,\,L,\,M) & =\begin{cases}
1/8 & \text{if }K<L<M\;,\\
1/16 & \text{if }K=L<M\text{ or }K<L=M\;,\\
1/48 & \text{if }K=L=M\;.
\end{cases}\label{e_kappaest}
\end{align}
\end{thm}

In particular, the quantity $\mathbb{E}_{\mathbf{s}}^{\beta}\,[\tau_{\breve{\mathbf{s}}}]$
represents the mean time required to jump from $\mathbf{s}$ to another
ground state; hence, the first formula of \eqref{e_EK} corresponds
to the so-called \textit{Eyring--Kramers law} for the Metropolis--Hastings
dynamics.
\begin{rem}
Here, we make several comments regarding Theorem \ref{t_EK}.
\begin{enumerate}
\item Although we do not present the exact formula for the constant $\kappa$
in the theorem, they can be explicitly expressed in terms of potential-theoretic
notions relevant to a random walk defined in a complicated space (cf.
\eqref{e_c} and \eqref{e_kappa} for the formulas). This random walk
is vague (cf. Proposition \ref{p_eest}) compared with the corresponding
random walk identified in \cite[Proposition 6.22]{KS 2D} for the
2D model, which reflects the complexity of the energy landscape of
the 3D model compared with that of the 2D one.
\item The constant $\kappa$ is model-dependent. For different Glauber dynamics
(even with identical boundary conditions), this constant may differ.
\item If $K<L<M$, the transition between ground states must occur in a
specific direction; meanwhile, if $K=L<M$ or $K<L=M$, there are
two possible directions for the transition. If $K=L=M$, there are
six possible directions. This explains the dependence of the asymptotics
of $\kappa$ on the relationships among $K$, $L$, and $M$.
\end{enumerate}
\end{rem}

The proof of Theorem \ref{t_EK} is conducted via the potential-theoretic
approach, which originates from \cite{BEGK}. Using this approach,
we can estimate the mean transition time $\mathbb{E}_{\mathbf{s}}^{\beta}\,[\tau_{\breve{\mathbf{s}}}]$
by obtaining a precise estimate of the \textit{capacity} between ground
states (cf. \cite[Proposition 6.10]{BL TM}). This estimate is typically
obtained from variational principles for capacities, such as the Dirichlet
and Thomson principles. In contrast, we use the $H^{1}$-approximation
technique developed in our companion article \cite{KS 2D}, which
considerably simplifies the proof but still points out the gist of
the logical structure needed to estimate the capacity.

To this end, we require precise analyses of the energy landscape and
the behavior of the underlying metastable processes on a certain neighborhood
of saddle configurations between metastable sets. In most other models
for which the Eyring--Kramers law can be obtained via such robust
strategies, the energy landscape is relatively simple; hence, the
landscape only marginally presents serious mathematical issues. However,
in the current model, the saddle consists of a very large collection
of saddle configurations, which form a complex structure. Analyzing
this structure is a highly complicated task; moreover, it is difficult
to assess the behavior of the dynamics in the neighborhood of this
large set with adequate precision. The achievement of these tasks
is one of the main contributions of this study. We emphasize here
that the $H^{1}$-approximation technique, which is used in the proof
of the main results in a critical manner, is particularly handy for
models with complicated landscapes, such as the one considered in
this study.

\subsubsection*{Markov chain model reduction of metastable behavior}

Because the transitions between ground states occur successively,
analyzing all these transitions together is also an important problem
in the study of metastability. The general method used is Markov chain
model reduction \cite{BL TM,BL TM2,BL MG}. In this methodology, one
proves that the metastable process (accelerated by a certain scale)
converges, in a suitable sense, to a Markov chain on the set of metastable
sets. For our model, the target Markov chain must be a Markov chain
on the collection of ground states, because each ground state corresponds
to a metastable set. 

To explain this result in the context of our model, we introduce trace
process on ground states. In view of Theorem \ref{t_EK}, we must
accelerate the process by a factor $e^{\Gamma\beta}$ to observe transitions
between ground states in the ordinary time scale; hence, let us denote
by $\widehat{\sigma}_{\beta}(t)=\sigma_{\beta}(e^{\Gamma\beta}t)$,
$t\ge0$ the accelerated process. Then, we define a random time $T(t)$,
$t\ge0$ as 
\[
T(t)=\int_{0}^{t}\mathbf{1}\{\widehat{\sigma}_{\beta}(u)\in\mathcal{S}\}du\;\;\;\;;\;t\ge0\;,
\]
which measures the amount of time (up to $t$) the accelerated process
spends on the ground states. Let $S(\cdot)$ be the generalized inverse
of $T(\cdot)$; that is, 
\[
S(t)=\sup\,\{u\ge0:T(u)\le t\}\;\;\;\;;\;t\ge0\;.
\]
Then, the \textit{(accelerated) trace process} $\{X_{\beta}(t)\}_{t\ge0}$
on the set $\mathcal{S}$ of ground states is defined by
\begin{equation}
X_{\beta}(t)=\widehat{\sigma}_{\beta}(S(t))\text{ for }t\ge0\;.\label{e_trace}
\end{equation}
We observe that the trace process $X_{\beta}(\cdot)$ is obtained
from the accelerated process $\widehat{\sigma}_{\beta}(\cdot)$ by
turning off the clock whenever it is not on a ground state; thus,
the process $X_{\beta}(\cdot)$ extracts information regarding the
hopping dynamics on ground states. It is well known that the trace
process $X_{\beta}(\cdot)$ is a continuous-time, irreducible Markov
chain on $\mathcal{S}$; see \cite[Proposition 6.1]{BL TM} for a
rigorous proof.

Here, in view of the second estimate of \eqref{e_EK}, we define the
limiting Markov chain $\{X(t)\}_{t\ge0}$ on $\mathcal{S}$, which
expresses the asymptotic behavior of the accelerated process $\widehat{\sigma}_{\beta}(\cdot)$
between the ground states as a continuous-time Markov chain with jump
rate

\begin{equation}
r_{X}(\mathbf{s},\,\mathbf{s}')=\kappa^{-1}\text{ for all }\mathbf{s},\,\mathbf{s}'\in\mathcal{S}\;.\label{e_LMC}
\end{equation}

\begin{thm}
\label{t_MC}The following statements hold.
\begin{enumerate}
\item The law of the Markov chain $X_{\beta}(\cdot)$ converges to that
of the limiting Markov chain $X(\cdot)$ as $\beta\rightarrow\infty$,
in the usual Skorokhod topology.
\item It holds that 
\[
\lim_{\beta\rightarrow\infty}\max_{\mathbf{s}\in\mathcal{S}}\,\mathbb{E}_{\mathbf{s}}^{\beta}\,\Big[\,\int_{0}^{t}\mathbf{1}\{\widehat{\sigma}_{\beta}(u)\notin\mathcal{S}\}du\,\Big]=0\;.
\]
\end{enumerate}
\end{thm}

The second part of this theorem implies that the accelerated process
spends a negligible amount of time in the set $\mathcal{X}\setminus\mathcal{S}$.
Therefore, the trace process $X_{\beta}(\cdot)$ of $\widehat{\sigma}_{\beta}(\cdot)$
on the set $\mathcal{S}$, which is essentially obtained by neglecting
the excursion of $\widehat{\sigma}_{\beta}(\cdot)$ on the set $\mathcal{X}\setminus\mathcal{S}$,
is indeed a reasonable object for approximating the process $\widehat{\sigma}_{\beta}(\cdot)$.
Combining this observation with the first part of the theorem implies
that the limiting Markov chain $X(\cdot)$ describes the successive
metastable transitions of the Metropolis--Hastings dynamics.
\begin{rem}
The proofs of Theorems \ref{t_EK} and \ref{t_MC} are based on the
potential-theoretic argument, and we present the arguments in Section
\ref{sec3}. We conjecture that these results also hold for the cases
of more than three dimensions.
\end{rem}

\begin{rem}
Temporarily, we denote by $E_{\mathbf{s}}$ the law of the limiting
Markov chain $X(\cdot)$ starting at $\mathbf{s}\in\mathcal{S}$.
Theorem \ref{t_MC} is consistent with Theorem \ref{t_EK}, in that
for any $\mathbf{s}'\in\breve{\mathbf{s}}$, we have $E_{\mathbf{s}}\,[\tau_{\mathbf{s}'}]=\kappa$. 
\end{rem}

\begin{rem}[Discrete Metropolis--Hastings dynamics]
\label{r_discrete} The only difference in the discrete dynamics
defined by \eqref{e_discrete} is that it is $q|\Lambda|$ times slower
than the continuous dynamics (in the average sense). Therefore, Theorems
\ref{t_energy barrier}, \ref{t_LDT results}, and \ref{t_transpath}\textbf{
}are valid for this dynamics without any modification. Theorems \ref{t_EK}
and \ref{t_MC} hold provided that we replace the constant $\kappa$
with $\kappa'=q|\Lambda|\,\kappa$. The rigorous verification of the
result proceeds in a similar way; thus, we do not repeat it here.
\end{rem}

\subsubsection*{Outlook of proofs of main results}

To prove Theorems \ref{t_energy barrier} and \ref{t_LDT results},
which fall into the category of pathwise-type metastability results,
we investigate the energy landscape of the Ising/Potts models on the
3D lattice $\Lambda$, as described in Sections \ref{sec6}, \ref{sec7},
and \ref{sec8}. Along the investigation, we present proofs of Theorems
\ref{t_energy barrier} and \ref{t_LDT results} in Section \ref{sec8}.
Then, we proceed to the proofs of Theorems \ref{t_EK} and \ref{t_MC},
which require more accurate analyses of the energy landscape than
the previous theorems. These detailed analyses are presented in Section
\ref{sec9}, and as a byproduct we present the proof of Theorem \ref{t_transpath}
in Section \ref{sec9.4}. Then, we present the proofs of Theorems
\ref{t_EK} and \ref{t_MC} in Section \ref{sec10}.

\subsubsection*{Non-reversible models}

The stochastic system considered in this study is the continuous-time
Metropolis--Hastings spin-updating dynamics, which is reversible
with respect to the Gibbs measure $\mu_{\beta}(\cdot)$. In fact,
as in our companion paper \cite{KS 2D}, we can consider various dynamics
with invariant measure $\mu_{\beta}(\cdot)$ but are non-reversible
with respect to this measure. Since the approximation method and the
pathwise approach used in the proof of the main results presented
above are robust and can be used in the non-reversible setting as
well, we can analyze the 3D version of the non-reversible models introduced
in \cite{KS 2D} for the 2D model and obtain similar results. However,
for simplicity (as analysis of the energy landscape of the 3D model
is very complicated itself), we decided not to include the non-reversible
content in the current article. Readers who are interested in non-reversible
generalizations can refer to \cite[Sections 2.2 and 5]{KS 2D} for
details.

\section{\label{sec3}Outline of the Proof}

In this section, we provide a brief summary of proof of the main results.
We emphasize again that in the remainder of this article (except in
Section \ref{sec11}), we assume periodic boundary conditions; that
is, $\Lambda=\mathbb{T}_{K}\times\mathbb{T}_{L}\times\mathbb{T}_{M}$.
In addition, we always assume that $K$ satisfies the condition given
in Remark \ref{r_K}. 

We reduce the proofs of Theorems \ref{t_EK} and \ref{t_MC} (which
are the final destinations of the current article) to an estimate
of the capacity between ground states (cf. Theorem \ref{t_Cap}),
and then we reduce the proof of this capacity estimate to the construction
of a certain test function (cf. Proposition \ref{p_H1approx}) which
is a proper approximation of the equilibrium potential function defined
in \eqref{e_eqp}. The construction and verification of Proposition
\ref{p_H1approx} are done in Section \ref{sec10}. This procedure
takes into advantage all the information on the energy landscape,
analyzed in Sections \ref{sec6}-\ref{sec9}. 

General strategy to prove such results, which works also in non-reversible
cases, was developed in our companion article \cite[Section 4]{KS 2D}.
Thus, we state here only the essential ingredients in a self-contained
manner and refer the interested readers to \cite[Section 4]{KS 2D}
for more detail.

\subsection{\label{sec3.1}Capacity estimate and proof of Theorems \ref{t_EK}
and \ref{t_MC}}

The \textit{Dirichlet form} $D_{\beta}(\cdot)$ associated with the
(reversible) Metropolis--Hastings dynamics $\sigma_{\beta}(\cdot)$
is given by, for $f:\mathcal{X}\rightarrow\mathbb{R}$,
\begin{equation}
D_{\beta}(f)=\frac{1}{2}\sum_{\sigma,\,\zeta\in\mathcal{X}}\mu_{\beta}(\sigma)\,r_{\beta}(\sigma,\,\zeta)\,\{f(\zeta)-f(\sigma)\}^{2}\;.\label{e_Diri}
\end{equation}
An alternative expression for the Dirichlet form is given as
\begin{equation}
D_{\beta}(f)=\langle f,\,-\mathcal{L}_{\beta}f\rangle_{\mu_{\beta}}\;,\label{e_Diriinnprod}
\end{equation}
where $\langle\cdot,\,\cdot\rangle_{\mu_{\beta}}$ is the inner product
on $L^{2}(\mu_{\beta})$ and $\mathcal{L}_{\beta}$ is the generator
of the original process, that is,
\begin{equation}
(\mathcal{L}_{\beta}f)(\sigma)=\sum_{\zeta\in\mathcal{X}}r_{\beta}(\sigma,\,\zeta)\,[f(\zeta)-f(\sigma)]\;.\label{e_gen}
\end{equation}
For two disjoint and non-empty subsets $\mathcal{P}$ and $\mathcal{Q}$
of $\mathcal{X}$, the \textit{equilibrium potential} between $\mathcal{P}$
and $\mathcal{Q}$ is the function $h_{\mathcal{P},\,\mathcal{Q}}^{\beta}:\mathcal{X}\rightarrow\mathbb{R}$
defined by 
\begin{equation}
h_{\mathcal{P},\,\mathcal{Q}}^{\beta}(\sigma)=\mathbb{P}_{\sigma}^{\beta}\,[\tau_{\mathcal{P}}<\tau_{\mathcal{Q}}]\;.\label{e_eqp}
\end{equation}
By definition, it readily follows that $h_{\mathcal{P},\,\mathcal{Q}}^{\beta}\equiv1\text{ on }\mathcal{P}$
and $h_{\mathcal{P},\,\mathcal{Q}}^{\beta}\equiv0\text{ on }\mathcal{Q}$.
Then, we define the \textit{capacity} between $\mathcal{P}$ and $\mathcal{Q}$
as
\begin{equation}
\mathrm{Cap}_{\beta}(\mathcal{P},\,\mathcal{Q})=D_{\beta}(h_{\mathcal{P},\,\mathcal{Q}}^{\beta})\;.\label{e_Capdef}
\end{equation}
It is well known that the equilibrium potential is the unique solution
to the following equation:
\begin{equation}
\begin{cases}
f\equiv1 & \text{on }\mathcal{P}\;,\\
f\equiv0 & \text{on }\mathcal{Q}\;,\\
\mathcal{L}_{\beta}f\equiv0 & \text{on }\mathcal{X}\setminus(\mathcal{P}\cup\mathcal{Q})\;.
\end{cases}\label{e_eqpotsol}
\end{equation}

Next, we define the constant $\kappa=\kappa(K,\,L,\,M)$ that appears
in Theorems \ref{t_EK} and \ref{t_MC}.
\begin{itemize}
\item Let $\mathfrak{m}_{K}=\lfloor K^{2/3}\rfloor$, where $\lfloor\alpha\rfloor$
is the biggest integer not bigger than $\alpha$, and let $\kappa^{\mathrm{2D}}=\kappa^{\mathrm{2D}}(K,\,L)$
be the constant that appeared in \cite[(4.13)]{KS 2D}, which is defined
later in \eqref{e_kappa2Ddef} explicitly and satisfies
\begin{equation}
\lim_{K\to\infty}\kappa^{\mathrm{2D}}(K,\,L)=\begin{cases}
1/4 & \text{if }K<L\;,\\
1/8 & \text{if }K=L\;.
\end{cases}\label{e_kappa2Dprop-1}
\end{equation}
Then, for $n\in\llbracket1,\,q-1\rrbracket$, the bulk constant $\mathfrak{b}(n)$
is defined explicitly as 
\begin{equation}
\mathfrak{b}(n)=\begin{cases}
\frac{1}{n(q-n)}\cdot\frac{M-2\mathfrak{m}_{K}}{2M}\cdot\kappa^{\mathrm{2D}}(K,\,L) & \text{if }K<L<M\;,\\
\frac{1}{n(q-n)}\cdot\frac{M-2\mathfrak{m}_{K}}{2M}\cdot\kappa^{\mathrm{2D}}(K,\,L) & \text{if }K=L<M\;,\\
\frac{1}{n(q-n)}\cdot\frac{M-2\mathfrak{m}_{K}}{4M}\cdot\kappa^{\mathrm{2D}}(K,\,L) & \text{if }K<L=M\;,\\
\frac{1}{n(q-n)}\cdot\frac{M-2\mathfrak{m}_{K}}{6M}\cdot\kappa^{\mathrm{2D}}(K,\,L) & \text{if }K=L=M\;.
\end{cases}\label{e_b}
\end{equation}
\item The edge constant $\mathfrak{e}(n)$, $n\in\llbracket1,\,q-1\rrbracket$,
is defined in \eqref{e_e3def}. Furthermore, it is verified in Proposition
\ref{p_eest} that 
\begin{equation}
0<\mathfrak{e}(n)\le\frac{1}{K^{1/3}}\text{ for all }n\in\llbracket1,\,q-1\rrbracket\;.\label{e_e}
\end{equation}
\item Then, for $n\in\llbracket1,\,q-1\rrbracket$, we define the constant
\begin{equation}
\mathfrak{c}(n)=\mathfrak{b}(n)+\mathfrak{e}(n)+\mathfrak{e}(q-n)\;.\label{e_c}
\end{equation}
We remark that by definition, $\mathfrak{b}(n)=\mathfrak{b}(q-n)$
for $n\in\llbracket1,\,q-1\rrbracket$; therefore, we have $\mathfrak{c}(n)=\mathfrak{c}(q-n)$.
Finally, we define the constant $\kappa$ that appears in Theorem
\ref{t_EK} as 
\begin{equation}
\kappa=(q-1)\mathfrak{c}(1)\;.\label{e_kappa}
\end{equation}
\end{itemize}
For $A\subseteq S$, we define (cf. \eqref{e_S})
\[
\mathcal{S}(A)=\{\mathbf{s}_{a}:a\in A\}\;.
\]
A pair $(A,\,B)$ of two subsets $A$ and $B$ of $S$ is referred
to as a \textit{proper partition} of $S$ if $A$ and $B$ are non-empty
subsets of $S$ satisfying $A\cup B=S$ and $A\cap B=\emptyset$.
Our aim is to estimate the capacity between $\mathcal{S}(A)$ and
$\mathcal{S}(B)$ for proper partitions $(A,\,B)$ of $S$. The following
theorem expresses the key capacity estimate:
\begin{thm}
\label{t_Cap}It holds for any proper partition $(A,\,B)$ of $S$
that 
\begin{equation}
\mathrm{Cap}_{\beta}(\mathcal{S}(A),\,\mathcal{S}(B))=\frac{1+o_{\beta}(1)}{q\mathfrak{c}(|A|)}\,e^{-\Gamma\beta}\;,\label{e_Cap}
\end{equation}
where $\mathfrak{c}(|A|)$ is the constant defined in \eqref{e_c}.
\end{thm}

We explain the strategy used to prove this theorem in Section \ref{sec3.2}.
Here, we conclude the proofs of Theorems \ref{t_EK} and \ref{t_MC}
by assuming Theorem \ref{t_Cap}.
\begin{proof}[Proof of Theorem \ref{t_EK}]
 By \cite[Proposition 6.10]{BL TM}, we have the following formula
for the mean transition time:
\[
\mathbb{E}_{\mathbf{s}}^{\beta}\,[\tau_{\breve{\mathbf{s}}}]=\frac{1}{\mathrm{Cap}_{\beta}(\mathbf{s},\,\breve{\mathbf{s}})}\sum_{\sigma\in\mathcal{X}}\mu_{\beta}(\sigma)\,h_{\mathbf{s},\,\breve{\mathbf{s}}}^{\beta}(\sigma)\;.
\]
Using Theorem \ref{t_Zbest} and the fact that $h_{\mathbf{s},\,\breve{\mathbf{s}}}^{\beta}(\mathbf{s})=1$
and $h_{\mathbf{s},\,\breve{\mathbf{s}}}^{\beta}\equiv0$ on $\breve{\mathbf{s}}$,
we can rewrite the last summation as 
\[
\frac{1}{q}+o_{\beta}(1)+\sum_{\sigma\in\mathcal{X}\setminus\mathcal{S}}\mu_{\beta}(\sigma)\,h_{\mathbf{s},\,\breve{\mathbf{s}}}^{\beta}(\sigma)=\frac{1}{q}+o_{\beta}(1)\;,
\]
where the identity follows from the trivial bound $|h_{\mathbf{s},\,\breve{\mathbf{s}}}^{\beta}|\le1$
(cf. \eqref{e_eqp}). Summing up the computations above and applying
Theorem \ref{t_Cap}, we obtain 
\begin{equation}
\mathbb{E}_{\mathbf{s}}^{\beta}\,[\tau_{\breve{\mathbf{s}}}]=\frac{1}{\mathrm{Cap}_{\beta}(\mathbf{s},\,\breve{\mathbf{s}})}\,\Big[\,\frac{1}{q}+o_{\beta}(1)\,\Big]=(1+o_{\beta}(1))\,\frac{\kappa}{q-1}\,e^{\Gamma\beta}\;.\label{e_EKpf1}
\end{equation}

We next address the second estimate of \eqref{e_EK}. Assume that
the process $\sigma_{\beta}(\cdot)$ starts at $\mathbf{s}$ and that
$\mathbf{s}\neq\mathbf{s}'$. We define a sequence of stopping times
$(J_{n})_{n=0}^{\infty}$ by $J_{0}=0$ and 
\[
J_{n+1}=\inf\,\{t\ge J_{n}:\sigma_{\beta}(t)\in\mathcal{S}\setminus\sigma_{\beta}(J_{n})\}\;\;\;\;;\;n\ge0\;.
\]
In other words, $(J_{n})_{n=0}^{\infty}$ is the sequence of random
times at which the process $\sigma_{\beta}(\cdot)$ visits a new ground
state. By \eqref{e_EKpf1} and the strong Markov property, we have
for all $n\ge0$ that
\begin{equation}
\mathbb{E}_{\mathbf{s}}^{\beta}\,[J_{n+1}-J_{n}]=(1+o_{\beta}(1))\,\frac{\kappa}{q-1}\,e^{\Gamma\beta}\;.\label{e_EKpf2}
\end{equation}
Then, we define 
\[
n(\mathbf{s}')=\inf\,\{n\ge0:\sigma_{\beta}(J_{n})=\mathbf{s}'\}
\]
such that $\tau_{\mathbf{s}'}=J_{n(\mathbf{s}')}$; thus, we can write
\begin{equation}
\tau_{\mathbf{s}'}=\sum_{i=0}^{n(\mathbf{s}')-1}(J_{i+1}-J_{i})\;.\label{e_EKpf3}
\end{equation}
Note that because we have assumed $\mathbf{s}\neq\mathbf{s}'$, it
holds that $n(\mathbf{s}')\ge1$. By symmetry, we observe that $n(\mathbf{s}')$
is a geometric random variable with success probability $\frac{1}{q-1}$
that is independent of the sequence $(J_{n})_{n=0}^{\infty}$. Thus,
we get from \eqref{e_EKpf2} and \eqref{e_EKpf3} that 
\[
\mathbb{E}_{\mathbf{s}}^{\beta}\,[\tau_{\mathbf{s}'}]=(1+o_{\beta}(1))\,\frac{\kappa}{q-1}\,e^{\Gamma\beta}\times(q-1)=(1+o_{\beta}(1))\,\kappa\,e^{\Gamma\beta}\;.
\]
Finally, from \eqref{e_kappa2Dprop-1}, \eqref{e_b}, \eqref{e_e},
and \eqref{e_c}, we can easily see that $\kappa$ satisfies the asymptotics
\eqref{e_kappaest}. This completes the proof.
\end{proof}
Next, we consider Theorem \ref{t_MC}. The general methodology used
to prove this type of Markov chain model reduction, based on potential-theoretic
computations, was developed in \cite{BL TM,BL TM2}. Our proof also
uses the potential-theoretic approach; however, the computation is
slightly simpler because the metastable sets are singletons. Before
stating the proof, we remark that two alternative approaches are available
for the Markov chain model reduction in the context of metastability:
an approach based on the Poisson equation \cite{LMS DT,LS NR2,OhRez,RezSeo},
and one based on the resolvent equation \cite{LMS res,LeeS NR2}.
\begin{proof}[Proof of Theorem \ref{t_MC}]
 We first consider part (1). We denote by $r_{\beta}^{\mathrm{tr}}:\mathcal{S}\times\mathcal{S}\rightarrow[0,\,\infty)$
the transition rate of the trace process $X_{\beta}(\cdot)$. In view
of the rate \eqref{e_LMC} of the limiting Markov chain, it suffices
to prove that $r_{\beta}^{\mathrm{tr}}(\mathbf{s},\,\mathbf{s}')=(1+o_{\beta}(1))\,\frac{1}{\kappa}$
for all $\mathbf{s},\,\mathbf{s}'\in\mathcal{S}$. Since $r_{\beta}(\mathbf{s},\,\mathbf{s}')$
does not depend on the selections of $\mathbf{s},\,\mathbf{s}'\in\mathcal{S}$
by the symmetry of the model, it remains to prove that 
\begin{equation}
r_{\beta}^{\mathrm{tr}}(\mathbf{s},\,\breve{\mathbf{s}})=(1+o_{\beta}(1))\,\frac{q-1}{\kappa}\;\;\;\;\text{for all }\mathbf{s}\in\mathcal{S}\;.\label{e_MCpf1}
\end{equation}
We denote by $\mathbf{E}_{\mathbf{s}}^{\beta}$ the law of the trace
process $X_{\beta}(\cdot)$ starting at $\mathbf{s}$. Then, 
\begin{equation}
\frac{1}{r_{\beta}^{\mathrm{tr}}(\mathbf{s},\,\breve{\mathbf{s}})}=\mathbf{E}_{\mathbf{s}}^{\beta}\,[\tau_{\breve{\mathbf{s}}}]=e^{-\Gamma\beta}\,\mathbb{E}_{\mathbf{s}}^{\beta}\,\Big[\,\int_{0}^{\tau_{\breve{\mathbf{s}}}}\mathbf{1}\{\sigma_{\beta}(t)\in\mathcal{S}\}dt\,\Big]\;,\label{e_MCpf2}
\end{equation}
where the factor $e^{-\Gamma\beta}$ is included because we accelerated
the process by the factor $e^{\Gamma\beta}$ when defining the trace
process; the integrand $\mathbf{1}\{\sigma_{\beta}(t)\in\mathcal{S}\}$
arises because the trace process is obtained from the accelerated
process by turning off the clock when the process resides outside
$\mathcal{S}$. Then, by \cite[Proposition 6.10]{BL TM}, we can write
\[
\mathbb{E}_{\mathbf{s}}^{\beta}\,\Big[\,\int_{0}^{\tau_{\breve{\mathbf{s}}}}\mathbf{1}\{\sigma_{\beta}(t)\in\mathcal{S}\}dt\,\Big]=\frac{1}{\mathrm{Cap}_{\beta}(\mathbf{s},\,\breve{\mathbf{s}})}\sum_{\sigma\in\mathcal{X}}\mu_{\beta}(\sigma)\,\mathbf{1}\{\sigma\in\mathcal{S}\}\,h_{\mathbf{s},\,\breve{\mathbf{s}}}^{\beta}(\sigma)=\frac{\mu_{\beta}(\mathbf{s})}{\mathrm{Cap}_{\beta}(\mathbf{s},\,\breve{\mathbf{s}})}\;,
\]
where the second identity follows from the fact that $h_{\mathbf{s},\,\breve{\mathbf{s}}}^{\beta}(\mathbf{s})=1$
and $h_{\mathbf{s},\,\breve{\mathbf{s}}}^{\beta}\equiv0$ on $\breve{\mathbf{s}}$.
Therefore, by Theorems \ref{t_Zbest} and \ref{t_Cap}, we obtain
\[
\mathbb{E}_{\mathbf{s}}^{\beta}\,\Big[\,\int_{0}^{\tau_{\breve{\mathbf{s}}}}\mathbf{1}\{\sigma_{\beta}(t)\in\mathcal{S}\}dt\,\Big]=(1+o_{\beta}(1))\,\frac{\kappa}{q-1}\,e^{\Gamma\beta}\;.
\]
Inserting this into \eqref{e_MCpf2} yields \eqref{e_MCpf1}.

Here, we address part (2). Denote by $\mathbb{P}_{\mu_{\beta}}^{\beta}$
the law of the Metropolis--Hastings dynamics $\sigma_{\beta}(\cdot)$
for which the initial distribution is $\mu_{\beta}$. Then, for any
$u>0$, we obtain
\begin{equation}
\mathbb{P}_{\mathbf{s}}^{\beta}\,[\sigma_{\beta}(u)\notin\mathcal{S}]\le\frac{1}{\mu_{\beta}(\mathbf{s})}\,\mathbb{P}_{\mu_{\beta}}^{\beta}\,[\sigma_{\beta}(u)\notin\mathcal{S}]=\frac{\mu_{\beta}(\mathcal{X}\setminus\mathcal{S})}{\mu_{\beta}(\mathbf{s})}\;,\label{e_MCpf3}
\end{equation}
where the final identity holds because $\mu_{\beta}$ is the invariant
distribution. Therefore by the Fubini theorem, 
\[
\mathbb{E}_{\mathbf{s}}^{\beta}\,\Big[\,\int_{0}^{t}\mathbf{1}\{\widehat{\sigma}_{\beta}(u)\notin\mathcal{S}\}du\,\Big]=\int_{0}^{t}\mathbb{P}_{\mathbf{s}}^{\beta}\,[\sigma_{\beta}(e^{\Gamma\beta}u)\notin\mathcal{S}]du\le t\cdot\frac{\mu_{\beta}(\mathcal{X}\setminus\mathcal{S})}{\mu_{\beta}(\mathbf{s})}\;,
\]
which vanishes as $\beta\rightarrow\infty$ by Theorem \ref{t_Zbest}.
\end{proof}

\subsection{\label{sec3.2}$H^{1}$-approximation of equilibrium potential and
proof of Theorem \ref{t_Cap}}

We fix a proper partition $(A,\,B)$ of $S$, and explain the general
strategy to prove Theorem \ref{t_Cap}, that is, to estimate the capacity
$\mathrm{Cap}_{\beta}(\mathcal{S}(A),\,\mathcal{S}(B))$. 

The methodology explained here is based on \cite[Section 4.5]{KS 2D},
in which it is demonstrated that finding a suitable $H^{1}$-approximation
of the equilibrium potential $h_{\mathcal{S}(A),\,\mathcal{S}(B)}^{\beta}$
between $\mathcal{S}(A)$ and $\mathcal{S}(B)$ is sufficient to establish
the capacity estimate. The following proposition states this result.
\begin{prop}[$H^{1}$-approximation of the equilibrium potential]
\label{p_H1approx} For any proper partition $(A,\,B)$ of $S$,
there exists a function $\widetilde{h}=\widetilde{h}_{\mathcal{S}(A),\,\mathcal{S}(B)}^{\beta}:\mathcal{X}\rightarrow\mathbb{R}$
such that the following properties hold.
\begin{enumerate}
\item The function $\widetilde{h}$ approximates $h_{\mathcal{S}(A),\,\mathcal{S}(B)}^{\beta}$
in the sense that
\begin{equation}
D_{\beta}(h_{\mathcal{S}(A),\,\mathcal{S}(B)}^{\beta}-\widetilde{h})=o_{\beta}(e^{-\beta\Gamma})\;.\label{e_H1approx1}
\end{equation}
\item It holds that
\begin{equation}
D_{\beta}(\widetilde{h})=\frac{1+o_{\beta}(1)}{q\mathfrak{c}(|A|)}\,e^{-\Gamma\beta}\;.\label{e_H1approx2}
\end{equation}
\end{enumerate}
\end{prop}

\begin{rem}
The following statements are remarks on the previous proposition. 
\begin{enumerate}
\item Since the (square root of the) Dirichlet form can be regarded as an
$H^{1}$-seminorm, by \eqref{e_H1approx1}, the test function $\widetilde{h}$
approximates $h_{\mathcal{S}(A),\,\mathcal{S}(B)}^{\beta}$ in the
$H^{1}$-sense.
\item Property \eqref{e_H1approx2} is the one that should be satisfied
by the equilibrium potential, provided that Theorem \ref{t_Cap} holds
in view of \eqref{e_Capdef}. 
\item Proposition \ref{p_H1approx} has a simpler form compared to the original
one \cite[Proposition 4.4]{KS 2D}, because here we only need to consider
the case when $(A,\,B)$ is a proper \emph{partition} of $S$. 
\item Finding the test function $\widetilde{h}$ requires precise information
on the energy landscape and a deep insight into typical patterns of
the Metropolis--Hastings dynamics in a suitable neighborhood of saddle
configurations. We derive this in Sections \ref{sec6}-\ref{sec9}.
Then, the construction of the test function $\widetilde{h}$ and the
proof of Proposition \ref{p_H1approx} are given in Section \ref{sec10}.
\end{enumerate}
\end{rem}

Finally, provided that Proposition \ref{p_H1approx} holds, we prove
Theorem \ref{t_Cap}.
\begin{proof}[Proof of Theorem \ref{t_Cap}]
 By the triangle inequality for the seminorm $D_{\beta}(\cdot)^{1/2}$,
it holds that
\[
D_{\beta}(\widetilde{h})^{1/2}-D_{\beta}(h_{\mathcal{S}(A),\,\mathcal{S}(B)}^{\beta}-\widetilde{h})^{1/2}\le D_{\beta}(h_{\mathcal{S}(A),\,\mathcal{S}(B)}^{\beta})^{1/2}\le D_{\beta}(\widetilde{h})^{1/2}+D_{\beta}(h_{\mathcal{S}(A),\,\mathcal{S}(B)}^{\beta}-\widetilde{h})^{1/2}\;.
\]
Hence, by \eqref{e_H1approx1} and \eqref{e_H1approx2}, we obtain
\[
D_{\beta}(h_{\mathcal{S}(A),\,\mathcal{S}(B)}^{\beta})=\frac{1+o_{\beta}(1)}{q\mathfrak{c}(|A|)}\,e^{-\Gamma\beta}\;.
\]
By \eqref{e_Capdef}, the last display completes the proof.
\end{proof}
Hence, to prove the main results given in Theorems \ref{t_EK} and
\ref{t_MC}, it remains to prove Proposition \ref{p_H1approx}. The
proof is given in Section \ref{sec10}.

\section{\label{sec4}Neighborhood of Configurations}

In this section, we introduce several notions of neighborhoods of
configurations, which are analogues of the same concepts defined in
\cite[Section 6.1]{KS 2D}. These notions will be crucially used in
the characterization of energy landscape and in the construction of
test objects.

For $c\in\mathbb{R}$, a path $(\omega_{t})_{t=0}^{T}$ in $\mathcal{X}$
is called a \textit{$c$-path} if it is a path in the sense of Section
\ref{sec2.2}, and moreover satisfies $H(\omega_{t})\le c$ for all
$t\in\llbracket0,\,T\rrbracket$. Moreover, we say that this path
is in $\mathcal{P}\subseteq\mathcal{X}$ if $\omega_{t}\in\mathcal{P}$
for all $t\in\llbracket0,\,T\rrbracket$.
\begin{defn}[Neighborhood of configurations]
\label{d_nbd}
\begin{enumerate}
\item For $\sigma\in\mathcal{X}$, the \textit{neighborhood} $\mathcal{N}(\sigma)$
and the \textit{extended neighborhood} $\widehat{\mathcal{N}}(\sigma)$
are defined as 
\begin{align*}
\mathcal{N}(\sigma) & =\{\zeta\in\mathcal{X}:\text{There exists a }(\Gamma-1)\text{-path }(\omega_{t})_{t=0}^{T}\text{ connecting }\sigma\text{ and }\zeta\}\;\text{and}\\
\widehat{\mathcal{N}}(\sigma) & =\{\zeta\in\mathcal{X}:\text{There exists a }\Gamma\text{-path }(\omega_{t})_{t=0}^{T}\text{ connecting }\sigma\text{ and }\zeta\}\;.
\end{align*}
We set $\mathcal{N}(\sigma)=\emptyset$ (resp. $\widehat{\mathcal{N}}(\sigma)=\emptyset$)
if $H(\sigma)>\Gamma-1$ (resp. $H(\sigma)>\Gamma$). Then for $\mathcal{P}\subseteq\mathcal{X}$,
we define 
\[
\mathcal{N}(\mathcal{P})=\bigcup_{\sigma\in\mathcal{P}}\mathcal{N}(\sigma)\;\;\;\;\text{and}\;\;\;\;\widehat{\mathcal{N}}(\mathcal{P})=\bigcup_{\sigma\in\mathcal{P}}\mathcal{\widehat{\mathcal{N}}}(\sigma)\;.
\]
\item Let $\mathcal{Q}\subseteq\mathcal{X}$. For $\sigma\in\mathcal{X}$
such that $\sigma\notin\mathcal{Q}$, we define
\[
\widehat{\mathcal{N}}(\sigma\,;\,\mathcal{Q})=\{\zeta\in\mathcal{X}:\text{There exists a }\Gamma\text{-path in }\mathcal{X}\setminus\mathcal{Q}\text{ connecting }\sigma\text{ and }\zeta\}\;.
\]
As before, we set $\widehat{\mathcal{N}}(\sigma\,;\,\mathcal{Q})=\emptyset$
if $H(\sigma)>\Gamma$. Then for $\mathcal{P}\subseteq\mathcal{X}$
disjoint with $\mathcal{Q}$, define 
\[
\widehat{\mathcal{N}}(\mathcal{P}\,;\,\mathcal{Q})=\bigcup_{\sigma\in\mathcal{P}}\widehat{\mathcal{N}}(\sigma\,;\,\mathcal{Q})\;.
\]
\end{enumerate}
\end{defn}

With this notation, by the definition of $\Gamma$, it holds that
$\mathcal{N}(\mathbf{s})\cap\mathcal{N}(\mathbf{s}')=\emptyset$ and
$\widehat{\mathcal{N}}(\mathbf{s})=\widehat{\mathcal{N}}(\mathbf{s}')$
for any $\mathbf{s},\,\mathbf{s}'\in\mathcal{S}$. Moreover, in the
spirit of the large deviation principle, the only configurations relevant
to the study of metastability are the ones in $\widehat{\mathcal{N}}(\mathcal{S})$.
Hence, it is crucial to understand the structure of the set $\widehat{\mathcal{N}}(\mathcal{S})$.
That is the content of Proposition \ref{p_typ}.

We conclude this section with an elementary lemma which will be used
in several instances of our discussion. The proof is well explained
in \cite[Lemma A.1]{KS 2D}, and thus we omit the detail.
\begin{lem}
\label{l_set}Suppose that $\mathcal{P}$ and $\mathcal{Q}$ are disjoint
subsets of $\mathcal{X}$. Then, it holds that 
\[
\widehat{\mathcal{N}}(\mathcal{P}\cup\mathcal{Q})=\widehat{\mathcal{N}}(\mathcal{Q}\,;\,\mathcal{P})\cup\widehat{\mathcal{N}}(\mathcal{P}\,;\,\mathcal{Q})\;.
\]
\end{lem}

\section{\label{sec5}Review of Two-dimensional Model}

In this section, we recall some crucial 2D results on the energy landscape
from \cite[Sections 6, 7 and Appendices B, C]{KS 2D}, which are needed
in our investigation of the 3D model. Since all the results that appear
in the current section are proved in \cite{KS 2D}, we refer to the
proofs therein.
\begin{notation*}
Greek letters $\eta$ and $\xi$ are used to denote the spin configurations
of the 2D model, while letters $\sigma$ and $\zeta$ are used to
denote the 3D configurations. We use the superscript $\mathrm{2D}$
to stress the notation for the 2D model; for example, we shall denote
by $H^{\mathrm{2D}}(\cdot)$ the Hamiltonian of the 2D model to distinguish
with $H(\cdot)$ which denotes the Hamiltonian of the 3D model. 
\end{notation*}

\subsection{2D stochastic Ising and Potts models with periodic boundary conditions}

We denote by $\Lambda^{\mathrm{2D}}=\mathbb{T}_{K}\times\mathbb{T}_{L}$
the 2D lattice with periodic boundary conditions. Recall that $S=\{1,\,2,\,\dots,\,q\}$
denotes the set of spins, and denote by $\mathcal{X}^{\mathrm{2D}}=S^{\Lambda^{\mathrm{2D}}}$
the space of spin configurations on the 2D lattice. Then, the 2D Ising/Potts
Hamiltonian function $H^{\mathrm{2D}}:\mathcal{X}^{\mathrm{2D}}\rightarrow\mathbb{R}$
(without external field) is defined by 
\begin{equation}
H^{\mathrm{2D}}(\eta)=\sum_{\{x,\,y\}\subseteq\Lambda^{\mathrm{2D}}:\,x\sim y}\mathbf{1}\{\eta(x)\ne\eta(y)\}\;\;\;\;;\;\eta\in\mathcal{X}^{\mathrm{2D}}\;.\label{e_Ham2D}
\end{equation}
We denote by $\mathbf{s}_{a}^{\mathrm{2D}}$, $a\in S$ the 2D monochromatic
configurations of spin $a$, that is, $\mathbf{s}_{a}^{\mathrm{2D}}(x)=a$
for all $x\in\Lambda^{\mathrm{2D}}$. Then, it is straightforward
that the ground states of this Hamiltonian is also the monochromatic
configurations, i.e., the collection $\mathcal{S}^{\mathrm{2D}}$
of the ground states is given as

\[
\mathcal{S}^{\mathrm{2D}}=\{\mathbf{s}_{1}^{\mathrm{2D}},\,\mathbf{s}_{2}^{\mathrm{2D}},\,\dots,\,\mathbf{s}_{q}^{\mathrm{2D}}\}\;.
\]
Then, we write $\mu_{\beta}^{\mathrm{2D}}(\cdot)$ the associated
2D Gibbs measure, i.e., 
\[
\mu_{\beta}^{\mathrm{2D}}(\eta)=\frac{1}{Z_{\beta}^{\mathrm{2D}}}e^{-\beta H^{\mathrm{2D}}(\eta)}\;\;\;\;;\;\eta\in\mathcal{X}^{\mathrm{2D}}\;.
\]
Here, $Z_{\beta}^{\mathrm{2D}}$ is the 2D partition function with
the property that (cf. \cite[Theorem 2.1]{KS 2D})
\begin{equation}
\lim_{\beta\to\infty}Z_{\beta}^{\mathrm{2D}}=q\;.\label{e_Zbest2D}
\end{equation}
In the 2D model, we also consider the continuous-time Metropolis--Hastings
dynamics whose transition rate is defined as
\[
r_{\beta}^{\mathrm{2D}}(\eta,\,\xi)=\begin{cases}
e^{-\beta[H^{\mathrm{2D}}(\xi)-H^{\mathrm{2D}}(\eta)]_{+}} & \text{if }\xi=\eta^{x,\,a}\ne\eta\text{ for some }x\in\Lambda^{\mathrm{2D}}\text{ and }a\in S\;,\\
0 & \text{otherwise}\;.
\end{cases}
\]
This 2D stochastic Ising/Potts model is thoroughly analyzed in our
companion article \cite{KS 2D}. The remainder of this section presents
a review of our analysis.

\subsection{\label{sec5.2}Energy barrier and canonical transition paths}

It is verified in \cite[Theorem 2.1]{NZ} that the energy barrier
between the ground states of the 2D model is given by 
\[
\Gamma^{\mathrm{2D}}=2K+2\;.
\]
Then, by replacing $\Gamma$ that appears in Definition \ref{d_nbd}
with $\Gamma^{\mathrm{2D}}$, we get two types of neighborhoods $\mathcal{N}^{\mathrm{2D}}$
and $\widehat{\mathcal{N}}^{\mathrm{2D}}$ for the 2D model. In this
subsection, we explain a class of natural optimal transition paths
that achieve this energy level. These paths are denoted as canonical
paths. To define these paths, we first define the so-called canonical
configurations. We note that the constructions given here is a brief
survey of \cite[Section 6.2]{KS 2D}.

\subsubsection*{Canonical configurations}

The following notation is used throughout the article (also for the
3D model).
\begin{notation}
\label{n_frakS}Suppose that $N\ge2$ is a positive integer.
\begin{itemize}
\item Define $\mathfrak{S}_{N}$ as the collection of connected subsets
of $\mathbb{T}_{N}$. For example, if $N=6$, few examples of the
elements of $\mathfrak{S}_{6}$ are $\emptyset$, $\{2,\,3\}$, $\{5,\,6,\,1\}$,
$\{4,\,5,\,6,\,1,\,2\}$, $\mathbb{T}_{6}$, etc.
\item For $P,\,P'\in\mathfrak{S}_{N}$, we write $P\prec P'$ if $P\subseteq P'$
and $|P'|=|P|+1$.
\item A sequence $(P_{m})_{m=0}^{N}$ of sets in $\mathfrak{S}_{N}$ is
called an \textit{increasing sequence} if it satisfies 
\[
\emptyset=P_{0}\prec P_{1}\prec\cdots\prec P_{N}=\mathbb{T}_{N}
\]
so that $|P_{m}|=m$ for all $m\in\llbracket0,\,N\rrbracket$.
\end{itemize}
\end{notation}

We first introduce the pre-canonical configurations which are illustrated
in Figure \ref{fig5.1}.

\begin{figure}
\includegraphics[width=13cm]{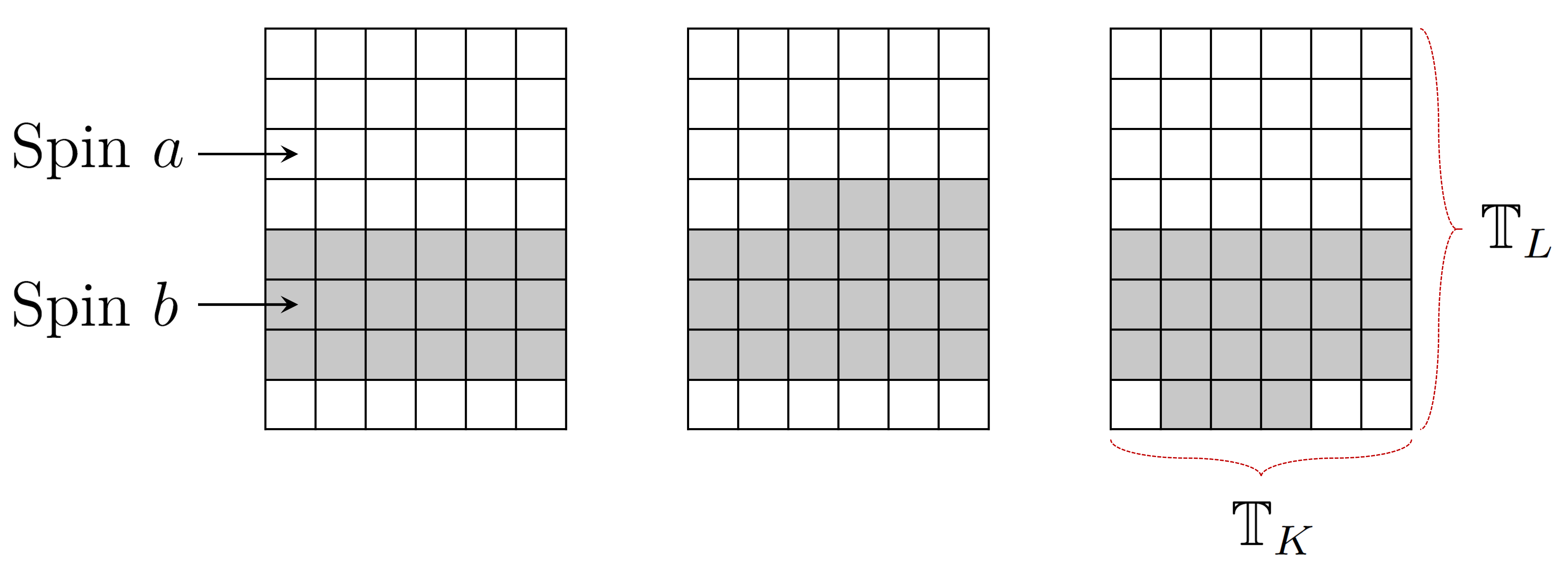}\caption{\label{fig5.1}\textbf{2D pre-canonical configurations.} These configurations
illustrate $\xi_{2,\,3}^{a,\,b}$, $\xi_{2,\,3;\,3,\,4}^{a,\,b,\,+}$,
and $\xi_{2,\,3;\,2,\,3}^{a,\,b,\,-}$, respectively.}
\end{figure}

\begin{defn}[2D pre-canonical configurations]
\label{d_precanreg2} Fix two spins $a,\,b\in S$.
\begin{itemize}
\item For $\ell\in\mathbb{T}_{L}$ and $v\in\llbracket0,\,L\rrbracket$,
we denote by $\xi_{\ell,\,v}^{a,\,b}\in\mathcal{X}^{\mathrm{2D}}$
the configuration whose spins are $b$ on
\[
\mathbb{T}_{K}\times\{\ell+n\in\mathbb{T}_{L}:n\in\llbracket0,\,v-1\rrbracket\subseteq\mathbb{Z}\}\;.
\]
and $a$ on the remainder.
\item For $\ell\in\mathbb{T}_{L}$, $v\in\llbracket0,\,L-1\rrbracket$,
$k\in\mathbb{T}_{K}$, and $h\in\llbracket0,\,K\rrbracket$, we denote
by $\xi_{\ell,\,v;\,k,\,h}^{a,\,b,\,+}\in\mathcal{X}^{\mathrm{2D}}$
the configuration whose spins are $b$ on
\[
\big\{\,x\in\Lambda^{\mathrm{2D}}:\xi_{\ell,\,v}^{a,\,b}(x)=b\,\big\}\cup\big[\,\{k+n\in\mathbb{T}_{K}:n\in\llbracket0,\,h-1\rrbracket\subseteq\mathbb{Z}\}\times\{\ell+v\}\,\big]
\]
and $a$ on the remainder. Similarly, $\xi_{\ell,\,v;\,k,\,h}^{a,\,b,\,-}\in\mathcal{X}^{\mathrm{2D}}$
is the configuration whose spins are $b$ on
\[
\big\{\,x\in\Lambda^{\mathrm{2D}}:\xi_{\ell,\,v}^{a,\,b}(x)=b\,\big\}\cup\big[\,\{k+n\in\mathbb{T}_{K}:n\in\llbracket0,\,h-1\rrbracket\subseteq\mathbb{Z}\}\times\{\ell-1\}\,\big]
\]
and $a$ on the remainder. The configurations defined here are \textit{2D
pre-canonical configurations}.
\end{itemize}
\end{defn}

Based on this definition, the 2D canonical and regular configurations
are defined.
\begin{defn}[2D canonical and regular configurations]
\label{d_canreg2} Fix $a,\,b\in S$. The definitions are slightly
different for the case of $K<L$ and the case of $K=L$. 
\begin{itemize}
\item \textbf{(Case $K<L$) }Collection $\mathcal{C}^{a,\,b,\,\mathrm{2D}}$
of \textit{2D canonical configurations} between $\mathbf{s}_{a}^{\mathrm{2D}}$
and $\mathbf{s}_{b}^{\mathrm{2D}}$ is defined by
\[
\mathcal{C}^{a,\,b,\,\mathrm{2D}}=\bigcup_{\ell\in\mathbb{T}_{L}}\bigcup_{v\in\llbracket0,\,L\rrbracket}\{\xi_{\ell,\,v}^{a,\,b}\}\cup\bigcup_{\ell\in\mathbb{T}_{L}}\bigcup_{v\in\llbracket0,\,L-1\rrbracket}\bigcup_{k\in\mathbb{T}_{K}}\bigcup_{h\in\llbracket1,\,K-1\rrbracket}\{\xi_{\ell,\,v;\,k,\,h}^{a,\,b,\,+},\,\xi_{\ell,\,v;\,k,\,h}^{a,\,b,\,-}\}\;.
\]
Then, the collection of \textit{canonical configurations} is given
as
\begin{equation}
\mathcal{C}^{\mathrm{2D}}=\bigcup_{a,\,b\in S}\mathcal{C}^{a,\,b,\,\mathrm{2D}}\;.\label{e_C2D}
\end{equation}
Similarly,
\[
\mathcal{R}_{v}^{a,\,b,\,\mathrm{2D}}=\bigcup_{\ell\in\mathbb{T}_{L}}\{\xi_{\ell,\,v}^{a,\,b}\}\;\;\;\;\text{and}\;\;\;\;\mathcal{Q}_{v}^{a,\,b,\,\mathrm{2D}}=\bigcup_{\ell\in\mathbb{T}_{L}}\bigcup_{k\in\mathbb{T}_{K}}\bigcup_{h\in\llbracket1,\,K-1\rrbracket}\{\xi_{\ell,\,v;\,k,\,h}^{a,\,b,\,\pm}\}\;,
\]
and then define $\mathcal{R}_{v}^{\mathrm{2D}}=\bigcup_{a,\,b\in S}\mathcal{R}_{v}^{a,\,b,\,\mathrm{2D}}$
and $\mathcal{Q}_{v}^{\mathrm{2D}}=\bigcup_{a,\,b\in S}\mathcal{Q}_{v}^{a,\,b,\,\mathrm{2D}}$.
A configuration in $\mathcal{R}_{v}^{\mathrm{2D}}$ is called a \textit{2D
regular configuration}.
\item \textbf{(Case $K=L$)} Define an operator $\Theta:\mathcal{X}^{\mathrm{2D}}\rightarrow\mathcal{X}^{\mathrm{2D}}$
as a transpose operator, i.e., 
\begin{equation}
(\Theta(\sigma))(k,\,\ell)=\sigma(\ell,\,k)\;\;\;\;;\;k\in\mathbb{T}_{K}\text{ and }\ell\in\mathbb{T}_{L}\;.\label{e_transpose}
\end{equation}
Denote temporarily by $\widetilde{\mathcal{C}}^{a,\,b,\,\mathrm{2D}}$
the collection $\mathcal{C}^{a,\,b,\,\mathrm{2D}}$ defined in the
case of $K<L$ above. Then for $a,\,b\in S$, we define the collections
of \textit{\emph{2D canonical configurations}} between $\mathbf{s}_{a}^{\mathrm{2D}}$
and $\mathbf{s}_{b}^{\mathrm{2D}}$ as
\[
\mathcal{C}^{a,\,b,\,\mathrm{2D}}=\widetilde{\mathcal{C}}^{a,\,b,\,\mathrm{2D}}\cup\Theta(\widetilde{\mathcal{C}}^{a,\,b,\,\mathrm{2D}})\;\;\;\;\text{and}\;\;\;\;\mathcal{C}^{\mathrm{2D}}=\bigcup_{a,\,b\in S}\mathcal{C}^{a,\,b,\,\mathrm{2D}}\;.
\]
Similarly, we may define the collections $\mathcal{R}_{v}^{a,\,b,\,\mathrm{2D}}$,
$\mathcal{R}_{v}^{\mathrm{2D}}$, $\mathcal{Q}_{v}^{a,\,b,\,\mathrm{2D}}$,
and $\mathcal{Q}_{v}^{\mathrm{2D}}$.
\end{itemize}
\end{defn}

\subsubsection*{Canonical paths}

Now, we explain natural optimal paths between monochromatic configurations
(illustrated in Figure \ref{fig5.2}) that consist of canonical configurations.

\begin{figure}
\includegraphics[width=14cm]{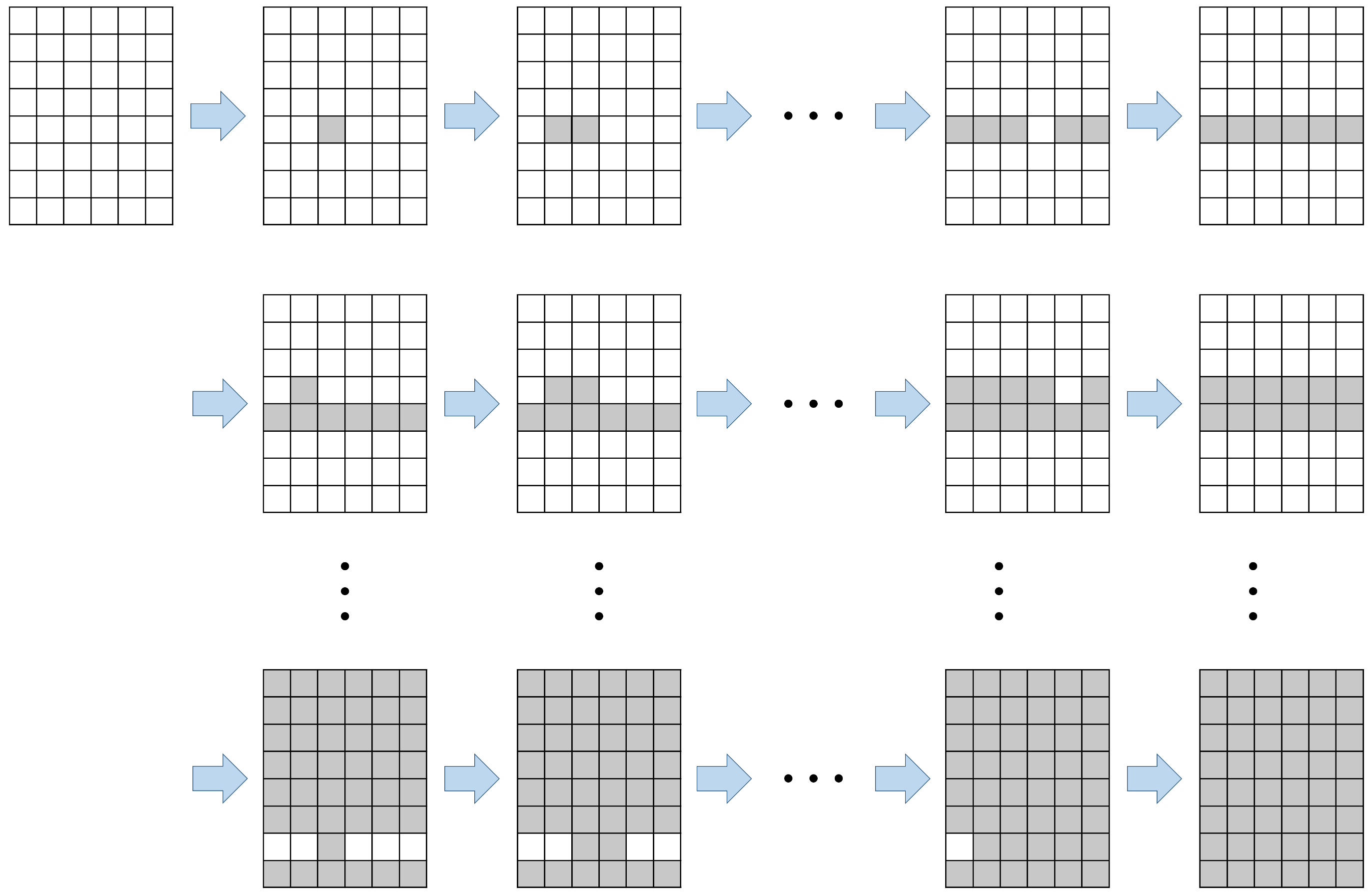}\caption{\label{fig5.2}\textbf{Example of a 2D canonical path from $\mathbf{s}_{a}^{\mathrm{2D}}$
to $\mathbf{s}_{b}^{\mathrm{2D}}$.}}
\end{figure}

\begin{defn}[2D canonical paths]
\label{d_canpath2} The definition below relies on Notation \ref{n_frakS}. 
\begin{enumerate}
\item For $P,\,P'\in\mathfrak{S}_{L}$ with $P\prec P'$, a sequence $(A_{k})_{k=0}^{K}$
of subsets of $\Lambda^{\mathrm{2D}}$ is a \textit{standard sequence}
connecting $\mathbb{T}_{K}\times P$ and $\mathbb{T}_{K}\times P'$
if there exists an increasing sequence $(Q_{k})_{k=0}^{K}$ in $\mathfrak{S}_{K}$
such that 
\[
A_{k}=(\mathbb{T}_{K}\times P)\cup\big[\,Q_{k}\times(P'\setminus P)\,\big]\;\;\;\;;\;k\in\llbracket0,\,K\rrbracket\;.
\]
\item A sequence $(A_{n})_{n=0}^{KL}$ of subsets of $\Lambda^{\mathrm{2D}}$
is a \textit{standard sequence} connecting $\emptyset$ and $\Lambda^{\mathrm{2D}}$
if there exists an increasing sequence $(P_{\ell})_{\ell=0}^{L}$
in $\mathfrak{S}_{L}$ such that $A_{K\ell}=\mathbb{T}_{K}\times P_{\ell}$
for all $\ell\in\llbracket0,\,L\rrbracket$, and furthermore for each
$\ell\in\llbracket0,\,L-1\rrbracket$ the subsequence $(A_{k})_{k=K\ell}^{K(\ell+1)}$
is a standard sequence connecting $\mathbb{T}_{K}\times P_{\ell}$
and $\mathbb{T}_{K}\times P_{\ell+1}$.
\item For $a,\,b\in S$, a sequence~$(\omega_{n})_{n=0}^{KL}$ of 2D configurations
is called a \textit{pre-canonical path} from $\mathbf{s}_{a}^{\mathrm{2D}}$
to $\mathbf{s}_{b}^{\mathrm{2D}}$ if there exists a standard sequence
$(A_{n})_{n=0}^{KL}$ connecting $\emptyset$ and $\Lambda^{\mathrm{2D}}$
such that 
\[
\omega_{n}(x)=\begin{cases}
a & \text{if }x\notin A_{n}\;,\\
b & \text{if }x\in A_{n}\;.
\end{cases}
\]
\item Moreover, a sequence~$(\omega_{n})_{n=0}^{KL}$ of 2D configurations
is called a \textit{canonical path }(cf. Figure \ref{fig5.2}) connecting
$\mathbf{s}_{a}^{\mathrm{2D}}$ and $\mathbf{s}_{b}^{\mathrm{2D}}$
if there exists a pre-canonical path $(\widetilde{\omega}_{n})_{n=0}^{KL}$
such that 
\begin{enumerate}
\item \textbf{(Case $K<L$)} $\omega_{n}=\widetilde{\omega}_{n}$ for all
$n\in\llbracket0,\,KL\rrbracket$,
\item \textbf{(Case $K=L$)} $\omega_{n}=\widetilde{\omega}_{n}$ for all
$n\in\llbracket0,\,KL\rrbracket$ or $\omega_{n}=\Theta(\widetilde{\omega}_{n})$
for all $n\in\llbracket0,\,KL\rrbracket$.
\end{enumerate}
\end{enumerate}
\end{defn}

It holds that $H^{\mathrm{2D}}(\eta)\le2K+2$ for all $\eta\in\mathcal{C}^{a,\,b,\,\mathrm{2D}}$
and
\begin{equation}
H^{\mathrm{2D}}(\eta)=\begin{cases}
2K & \text{if }\eta\in\mathcal{R}_{v}^{a,\,b,\,\mathrm{2D}}\text{ for }v\in\llbracket1,\,L-1\rrbracket\;,\\
2K+2 & \text{if }\eta\in\mathcal{Q}_{v}^{a,\,b,\,\mathrm{2D}}\text{ for }v\in\llbracket1,\,L-2\rrbracket\;.
\end{cases}\label{e_2Denergy}
\end{equation}
Moreover, the following lemma is immediate.
\begin{lem}[{\cite[Lemma 6.12]{KS 2D}}]
\label{l_canpath2} For a 2D canonical path $(\omega_{n})_{n=0}^{KL}$
connecting $\mathbf{s}_{a}^{\mathrm{2D}}$ and $\mathbf{s}_{b}^{\mathrm{2D}}$,
it holds that
\[
\max_{n\in\llbracket0,\,KL\rrbracket}H^{\mathrm{2D}}(\omega_{n})=\Gamma^{\mathrm{2D}}=2K+2\;.
\]
\end{lem}

\subsubsection*{Comment on depth of valleys}

We conclude this subsection with an application of Definition \ref{d_canpath2}
and Lemma \ref{l_canpath2} that is crucially used later to calculate
the 3D valley depths.
\begin{lem}[{\cite[Lemma B.4]{KS 2D}}]
\label{l_depth2} Let $\eta\in\mathcal{X}^{\mathrm{2D}}$ and $a\in S$.
For any standard sequence $(A_{k})_{k=0}^{KL}$ of sets connecting
$\emptyset$ and $\Lambda^{\mathrm{2D}}$ and for $n\in\llbracket0,\,KL\rrbracket$,
we define $\omega_{n}\in\mathcal{X}^{\mathrm{2D}}$ as
\[
\omega_{n}(x)=\begin{cases}
a & \text{if }x\in A_{n}\;,\\
\eta(x) & \text{if }x\in\Lambda^{\mathrm{2D}}\setminus A_{n}\;.
\end{cases}
\]
Then, we have that $H^{\mathrm{2D}}(\omega_{n})\le H^{\mathrm{2D}}(\eta)+\Gamma^{\mathrm{2D}}$
for all $n\in\llbracket0,\,KL\rrbracket$.
\end{lem}

In Lemma \ref{l_depth2}, we have $\omega_{KL}=\mathbf{s}_{a}^{\mathrm{2D}}\in\mathcal{S}^{\mathrm{2D}}$
which implies that every $\eta\in\mathcal{X}^{\mathrm{2D}}$ is connected
to each ground state in $\mathcal{S}^{\mathrm{2D}}$ with maximum
energy $H^{\mathrm{2D}}(\eta)+\Gamma^{\mathrm{2D}}$. This fact implies
that the maximum depth of valleys in the 2D energy landscape is $\Gamma^{\mathrm{2D}}$.

It can be further proved that only the valleys containing the ground
states have maximum depth $\Gamma^{\mathrm{2D}}$, and all the other
valleys have depth strictly less than $\Gamma^{\mathrm{2D}}$. Indeed,
this is a necessary condition for the pathwise approach technique
to metastability; however, this level of precision is not necessarily
needed in our investigation of the 3D energy landscape. Thus, we do
not go further into this direction and refer the interested readers
to \cite[Theorem 2.1-(ii)]{NZ}.

\subsection{Saddle structure}

Crucial configurations in the description of the saddle structure
of the 2D model is the so-called typical configurations, which turn
out to be the elements of the extended neighborhood $\widehat{\mathcal{N}}^{\mathrm{2D}}(\mathcal{S}^{\mathrm{2D}})$
(cf. Proposition \ref{p_typ2prop} below). We present in Figure \ref{fig5.3}
an illustration of the saddle structure explained in this subsection.

\begin{figure}
\includegraphics[width=14cm]{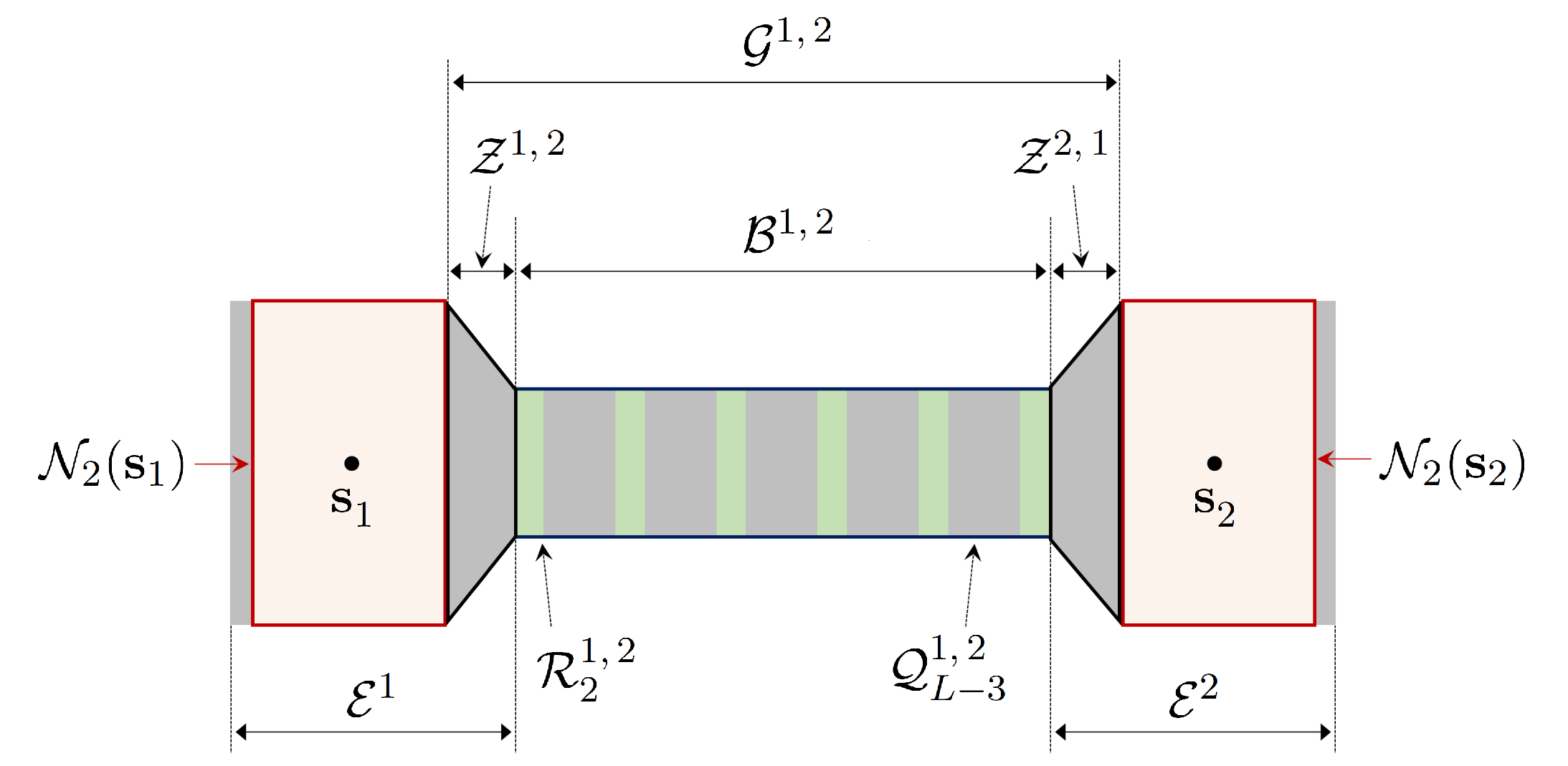}\caption{\label{fig5.3}\textbf{Saddle structure of the 2D Ising model with
$S=\{1,\,2\}$.} For simplicity, we drop the superscripts $\mathrm{2D}$
in this figure.}
\end{figure}

\begin{defn}[2D typical configurations]
\label{d_typ2} There are two different types of typical configurations:
the bulk and edge typical configurations. 
\begin{itemize}
\item For $a,\,b\in S$, the collection of \textit{bulk typical configurations}
(between $\mathbf{s}_{a}^{\mathrm{2D}}$ and $\mathbf{s}_{b}^{\mathrm{2D}}$)
is defined by
\begin{equation}
\mathcal{B}^{a,\,b,\,\mathrm{2D}}=\bigcup_{v\in\llbracket2,\,L-2\rrbracket}\mathcal{R}_{v}^{a,\,b,\,\mathrm{2D}}\cup\bigcup_{v\in\llbracket2,\,L-3\rrbracket}\mathcal{Q}_{v}^{a,\,b,\,\mathrm{2D}}\;.\label{e_Bab}
\end{equation}
Then, we write $\mathcal{B}^{\mathrm{2D}}=\bigcup_{a,\,b\in S}\mathcal{B}^{a,\,b,\,\mathrm{2D}}$.
\item Next, define 
\begin{equation}
\mathcal{B}_{\Gamma}^{a,\,b,\,\mathrm{2D}}=\bigcup_{v\in\llbracket2,\,L-3\rrbracket}\mathcal{Q}_{v}^{a,\,b,\,\mathrm{2D}}\;\;\;\;\text{and}\;\;\;\;\mathcal{B}_{\Gamma}^{\mathrm{2D}}=\bigcup_{a,\,b\in S}\mathcal{B}_{\Gamma}^{a,\,b,\,\mathrm{2D}}=\bigcup_{a,\,b\in S}\bigcup_{v\in\llbracket2,\,L-3\rrbracket}\mathcal{Q}_{v}^{a,\,b,\,\mathrm{2D}}\;.\label{e_BabGamma}
\end{equation}
Then, for $a\in S$, the collection of \textit{edge typical configurations}
with respect to $\mathbf{s}_{a}^{\mathrm{2D}}$ is defined by
\begin{equation}
\mathcal{E}^{a,\,\mathrm{2D}}=\widehat{\mathcal{N}}^{\mathrm{2D}}(\mathbf{s}_{a}^{\mathrm{2D}}\,;\,\mathcal{B}_{\Gamma}^{\mathrm{2D}})\;.\label{e_Ea}
\end{equation}
Finally, we write $\mathcal{E}^{\mathrm{2D}}=\bigcup_{a\in S}\mathcal{E}^{a,\,\mathrm{2D}}$.
\end{itemize}
\end{defn}

Then, the following crucial proposition provides the picture of the
saddle structure of the 2D model. We shall provide a similar result
for the 3D model in Proposition \ref{p_typ}.
\begin{prop}[{\cite[Proposition 6.16]{KS 2D}}]
\label{p_typ2prop}
\begin{enumerate}
\item For spins $a,\,b,\,c\in S$, we have
\[
\mathcal{E}^{a,\,\mathrm{2D}}\cap\mathcal{E}^{b,\,\mathrm{2D}}=\emptyset\;,\;\;\;\mathcal{E}^{a,\,\mathrm{2D}}\cap\mathcal{B}^{a,\,b,\,\mathrm{2D}}=\mathcal{R}_{2}^{a,\,b,\,\mathrm{2D}}\;\;\;\text{and}\;\;\;\mathcal{E}^{a,\,\mathrm{2D}}\cap\mathcal{B}^{b,\,c,\,\mathrm{2D}}=\emptyset\;.
\]
\item It holds that $\mathcal{E}^{\mathrm{2D}}\cup\mathcal{B}^{\mathrm{2D}}=\widehat{\mathcal{N}}^{\mathrm{2D}}(\mathcal{S}^{\mathrm{2D}})$.
\end{enumerate}
\end{prop}

\subsubsection*{Gateway configurations}

Next, we introduce the gateway configurations. 
\begin{defn}[2D gateway configurations]
\label{d_gate2} Fix $a,\,b\in S$. Define
\begin{align}
\mathcal{Z}^{a,\,b,\,\mathrm{2D}}:=\{\eta\in\mathcal{X}^{\mathrm{2D}}: & \;\text{There exists a path }(\omega_{t})_{t=0}^{T}\text{ in }\mathcal{X}^{\mathrm{2D}}\setminus\mathcal{B}_{\Gamma}^{\mathrm{2D}}\text{ with }T\ge1\text{ such that }\nonumber \\
 & \;\omega_{0}\in\mathcal{R}_{2}^{a,\,b,\,\mathrm{2D}}\;,\;\omega_{T}=\eta\text{ and }H^{\mathrm{2D}}(\omega_{t})=\Gamma^{\mathrm{2D}}\text{ for all }t\in\llbracket1,\,T\rrbracket\}\;.\label{e_Zab2def}
\end{align}
Intuitively, this set is the collection of saddle configurations between
$\mathcal{R}_{2}^{a,\,b,\,\mathrm{2D}}$ and $\mathbf{s}_{a}^{\mathrm{2D}}$.
Then, we recall the 2D gateway configurations \cite[Section B.5]{KS 2D}.
The gateway between $\mathbf{s}_{a}^{\mathrm{2D}}$ and $\mathbf{s}_{b}^{\mathrm{2D}}$
is denoted as
\begin{equation}
\mathcal{G}^{a,\,b,\,\mathrm{2D}}=\mathcal{Z}^{a,\,b,\,\mathrm{2D}}\cup\mathcal{B}^{a,\,b,\,\mathrm{2D}}\cup\mathcal{Z}^{b,\,a,\,\mathrm{2D}}\;,\label{e_gate2def}
\end{equation}
which is a decomposition of $\mathcal{G}^{a,\,b,\,\mathrm{2D}}$.
A configuration belonging to $\mathcal{G}^{a,\,b,\,\mathrm{2D}}$
is called a\textit{ gateway configuration} between $\mathbf{s}_{a}^{\mathrm{2D}}$
and $\mathbf{s}_{b}^{\mathrm{2D}}$.
\end{defn}

Here, $\mathcal{G}^{a,\,b,\,\mathrm{2D}}$ is named the collection
of gateway configurations because of the following lemma, which indicates
that it indeed contains the saddle configurations between $\mathbf{s}_{a}^{\mathrm{2D}}$
and $\mathbf{s}_{b}^{\mathrm{2D}}$.
\begin{lem}[{\cite[Lemma B.10]{KS 2D}}]
\label{l_gate2} For $a,\,b\in S$, suppose that two 2D configurations
$\eta$ and $\xi$ satisfy 
\[
\eta\in\mathcal{G}^{a,\,b,\,\mathrm{2D}}\;,\;\xi\notin\mathcal{G}^{a,\,b,\,\mathrm{2D}}\;,\;\eta\sim\xi\;,\;\text{and}\;H^{\mathrm{2D}}(\xi)\le\Gamma^{\mathrm{2D}}\;.
\]
Then, we have either $\xi\in\mathcal{N}^{\mathrm{2D}}(\mathbf{s}_{a}^{\mathrm{2D}})$
and $\eta\in\mathcal{Z}^{a,\,b,\,\mathrm{2D}}$ or $\xi\in\mathcal{N}^{\mathrm{2D}}(\mathbf{s}_{b}^{\mathrm{2D}})$
and $\eta\in\mathcal{Z}^{b,\,a,\,\mathrm{2D}}$. In particular, $\eta\notin\mathcal{B}^{a,\,b,\,\mathrm{2D}}$.
\end{lem}

We note that the construction of regular, canonical, typical, and
gateway configurations, as well as canonical paths for the 2D model,
will be extended to the 3D model in the remainder of the article. 

\subsection{Test function}

We also recall the 2D test function defined in \cite[Section 7]{KS 2D}.
Although the construction therein was carried out for both Ising and
Potts models, we only need the objects for the Ising model in this
article. Hence, in this subsection, \textbf{\emph{we assume that $q=2$.}}

Recall that we always assume $K\le L$. We recall a constant
\begin{equation}
\kappa^{\mathrm{2D}}=\kappa^{\mathrm{2D}}(K,\,L)\label{e_kappa2Ddef}
\end{equation}
from \cite[(4.13)]{KS 2D}, which plays the role of $\kappa$ in the
current article and also satisfies
\begin{equation}
\lim_{K\to\infty}\kappa^{\mathrm{2D}}(K,\,L)=\begin{cases}
1/4 & \text{if }K<L\;,\\
1/8 & \text{if }K=L\;.
\end{cases}\label{e_kappa2Dprop}
\end{equation}
In \cite[Definition 7.2]{KS 2D}, a test function $\widetilde{h}^{\mathrm{2D}}:\mathcal{X}^{\mathrm{2D}}\rightarrow\mathbb{R}$
(corresponding to $\widetilde{h}$ of the 3D model introduced in Proposition
\ref{p_H1approx}) is constructed as an $H^{1}$-approximation of
the equilibrium potential between two ground states. We proclaim that
this function is crucially used in the construction of the 3D test
function $\widetilde{h}$. In the proof of Proposition \ref{p_H1approx},
some estimates of $\widetilde{h}^{\mathrm{2D}}$ are crucially used.
The next estimate is used in the proof of \eqref{e_H1approx2}.
\begin{prop}[{\cite[Proposition C.1]{KS 2D}}]
\label{p_testfcn2} There exists a function $\widetilde{h}^{\mathrm{2D}}:\mathcal{X}^{\mathrm{2D}}\rightarrow\mathbb{R}$
such that
\begin{align*}
\sum_{\{\eta,\,\xi\}\subseteq\mathcal{X}^{\mathrm{2D}}:\,\{\eta,\,\xi\}\cap\mathcal{G}^{1,\,2,\,\mathrm{2D}}\ne\emptyset}\mu_{\beta}^{\mathrm{2D}}(\eta)\,r_{\beta}^{\mathrm{2D}}(\eta,\,\xi)\,\{\widetilde{h}^{\mathrm{2D}}(\xi)-\widetilde{h}^{\mathrm{2D}}(\eta)\}^{2} & =\frac{1+o_{\beta}(1)}{2\kappa^{\mathrm{2D}}}\,e^{-\Gamma^{\mathrm{2D}}\beta}\;.
\end{align*}
\end{prop}

The next one is crucially used in the proof of \eqref{e_H1approx1}.
\begin{prop}[{\cite[Lemmas 7.10-7.16]{KS 2D}}]
\label{p_testfcn2.2}
\begin{enumerate}
\item For all $\eta\in\mathcal{X}^{\mathrm{2D}}\setminus\mathcal{N}^{\mathrm{2D}}(\mathcal{S}^{\mathrm{2D}})$,
it holds that
\[
\sum_{\xi\in\mathcal{X}^{\mathrm{2D}}}\mu_{\beta}^{\mathrm{2D}}(\eta)\,r_{\beta}^{\mathrm{2D}}(\eta,\,\xi)\,[\widetilde{h}^{\mathrm{2D}}(\eta)-\widetilde{h}^{\mathrm{2D}}(\xi)]=o_{\beta}(e^{-\Gamma^{\mathrm{2D}}\beta})\;.
\]
\item We have that
\begin{align*}
\sum_{\eta\in\mathcal{N}^{\mathrm{2D}}(\mathbf{s}_{1}^{\mathrm{2D}})}\,\sum_{\xi\in\mathcal{X}^{\mathrm{2D}}}\mu_{\beta}^{\mathrm{2D}}(\eta)\,r_{\beta}^{\mathrm{2D}}(\eta,\,\xi)\,[\widetilde{h}^{\mathrm{2D}}(\eta)-\widetilde{h}^{\mathrm{2D}}(\xi)] & =(1+o_{\beta}(1))\times\frac{1}{2\kappa^{\mathrm{2D}}}\,e^{-\Gamma^{\mathrm{2D}}\beta}\;,\\
\sum_{\eta\in\mathcal{N}^{\mathrm{2D}}(\mathbf{s}_{2}^{\mathrm{2D}})}\,\sum_{\xi\in\mathcal{X}^{\mathrm{2D}}}\mu_{\beta}^{\mathrm{2D}}(\eta)\,r_{\beta}^{\mathrm{2D}}(\eta,\,\xi)\,[\widetilde{h}^{\mathrm{2D}}(\eta)-\widetilde{h}^{\mathrm{2D}}(\xi)] & =-(1+o_{\beta}(1))\times\frac{1}{2\kappa^{\mathrm{2D}}}\,e^{-\Gamma^{\mathrm{2D}}\beta}\;.
\end{align*}
\end{enumerate}
\end{prop}

\subsection{Auxiliary results}

In this subsection, we summarize two auxiliary results of the 2D model
that are crucially used in our arguments.

\subsubsection*{Bridges, crosses and a bound on 2D Hamiltonian}

For a configuration $\eta\in\mathcal{X}^{\mathrm{2D}}$, a \textit{bridge},
which is a \textit{horizontal }or \textit{vertical bridge}, is a row
or column, respectively, in which all spins are the same. If a bridge
consists of spin $a\in S$, we call this bridge an $a$-bridge. Then,
we denote by $B_{a}(\eta)$ the number of $a$-bridges with respect
to $\eta$. A \textit{cross} (resp. \textit{$a$-cross}) is the union
of a horizontal bridge and a vertical bridge (resp. $a$-bridges).
With this notation, we have the following lower bound.
\begin{lem}[{\cite[Lemma B.2]{KS 2D}}]
\label{l_H2lb} It holds that 
\[
H^{\mathrm{2D}}(\eta)\ge2\Big[\,K+L-\sum_{a\in S}B_{a}(\eta)\,\Big]\;.
\]
\end{lem}

\subsubsection*{Characterization of configurations with low energy}

Let $a\in S$. For $\eta\in\mathcal{X}^{\mathrm{2D}}$ and $\sigma\in\mathcal{X}$
(a 3D configuration), we write 
\begin{equation}
\|\eta\|_{a}=\sum_{x\in\Lambda^{\mathrm{2D}}}\mathbf{1}\{\eta(x)=a\}\;\;\;\;\text{and}\;\;\;\;\|\sigma\|_{a}=\sum_{x\in\Lambda}\mathbf{1}\{\sigma(x)=a\}\;.\label{e_spinnum}
\end{equation}
The following proposition characterizes all the 2D configurations
with energy less than $\Gamma^{\mathrm{2D}}$.
\begin{prop}[{\cite[Proposition B.3]{KS 2D}}]
\label{p_2lowE} Suppose that $\eta\in\mathcal{X}^{\mathrm{2D}}$
satisfies $H^{\mathrm{2D}}(\eta)<\Gamma^{\mathrm{2D}}$. Then, $\eta$
satisfies exactly one of the following properties.
\begin{itemize}
\item \textbf{\textup{(L1)}} There exist $a,\,b\in S$ and $v\in\llbracket2,\,L-2\rrbracket$
such that $\eta\in\mathcal{R}_{v}^{a,\,b,\,\mathrm{2D}}$. Here, $\mathcal{N}^{\mathrm{2D}}(\eta)=\{\eta\}$.
\item \textbf{\textup{(L2)}} There exist $a,\,b\in S$ such that $\eta\in\mathcal{R}_{1}^{a,\,b,\,\mathrm{2D}}$.
In this case, $\mathcal{N}^{\mathrm{2D}}(\eta)=\mathcal{N}^{\mathrm{2D}}(\mathbf{s}_{a}^{\mathrm{2D}})$.
\item \textbf{\textup{(L3)}} For some $a\in S$, $\eta$ has an $a$-cross.
Then, $\mathcal{N}^{\mathrm{2D}}(\eta)=\mathcal{N}^{\mathrm{2D}}(\mathbf{s}_{a}^{\mathrm{2D}})$
and
\begin{equation}
\sum_{b\ne a}\|\eta\|_{b}\le\frac{H^{\mathrm{2D}}(\eta)^{2}}{16}\le\frac{(2K+1)^{2}}{16}\;.\label{e_2lowE}
\end{equation}
\end{itemize}
\end{prop}

\section{\label{sec6}Canonical Configurations and Paths}

Analyzing the energy landscape of the 3D model is far more complex
than that of the 2D model; below, we briefly list the main differences
between them that serve to complexify the problem.
\begin{enumerate}
\item In the 2D model, the energy of the gateway configuration is either
$\Gamma^{\mathrm{2D}}$ or $\Gamma^{\mathrm{2D}}-2$. Thus, a $\Gamma^{\mathrm{2D}}$-path
on the gateway configurations does not have the freedom to move. On
the other hand, in the 3D model, the energy of the gateway configuration
ranges from $\Gamma-2K-2$ to $\Gamma$. This implies that the behavior
of a $\Gamma$-path around a gateway configuration of energy $\Gamma-2K-2$
(which is a regular configuration) cannot be characterized precisely.
\item In the 2D model, a $\Gamma^{\mathrm{2D}}$-path from $\mathbf{s}_{a}^{\mathrm{2D}}$
to $\mathbf{s}_{b}^{\mathrm{2D}}$ must visit a configuration in $\mathcal{R}_{2}^{a,\,b,\,\mathrm{2D}}$.
Then, it successively visits $\mathcal{R}_{3}^{a,\,b,\,\mathrm{2D}}$,
..., $\mathcal{R}_{L-2}^{a,\,b,\,\mathrm{2D}}$ and finally arrives
at $\mathbf{s}_{b}^{\mathrm{2D}}$. Remarkably, this path does not
need to visit a configuration in $\mathcal{R}_{1}^{a,\,b,\,\mathrm{2D}}$
and in $\mathcal{R}_{L-1}^{a,\,b,\,\mathrm{2D}}$; this fact essentially
arises from the features of the 2D geometry. In the 3D model, we observe
a similar phenomenon. To explain this, let us temporarily denote by
$\mathcal{R}_{v}^{a,\,b}$, $v\in\llbracket1,\,L-1\rrbracket$ the
collection of 3D configurations such that there are $v$ consecutive
$K\times L$ slabs of spins $b$ and such that the spins at the remaining
sites are $a$. Then, there exists an integer $n=n_{K,\,L,\,M}$ such
that any $\Gamma$-path connecting $\mathbf{s}_{a}$ and $\mathbf{s}_{b}$
must successively visit configurations in $\mathcal{R}_{n}^{a,\,b}$,
$\mathcal{R}_{n+1}^{a,\,b}$, $\dots,$ $\mathcal{R}_{M-n}^{a,\,b}$
but need not visit $\mathcal{R}_{i}^{a,\,b}$ for $i\in\llbracket1,\,n-1\rrbracket$
and $i\in\llbracket M-n+1,\,M-1\rrbracket$. In the 2D model, the
number corresponding to this $n=n_{K,\,L,\,M}$ is $2$. We guess
that in the 3D model, $n\sim K^{1/2}$; however, we cannot determine
the exact value of $n$. This fact reveals the complex structure of
the energy landscape in the 3D model. Instead, we prove below (cf.
Propositions \ref{p_nKLMlb} and \ref{p_Elb}) that 
\[
\lfloor K^{1/2}\rfloor\le n\le\lfloor K^{2/3}\rfloor\;.
\]
Fortunately, this bound suffices to complete our analysis without
identifying the exact value of $n$.
\item In the 2D model, the $\mathcal{N}^{\mathrm{2D}}$-neighborhoods are
fully characterized in Proposition \ref{p_2lowE}; meanwhile, in the
3D case, we cannot obtain such a specific and simple result. We overcome
the absence of this result by using the 2D result obtained in Proposition
\ref{p_2lowE}, through suitably applying it to the analysis of the
3D model. \textit{Indeed, this absence is a crucial difficulty in
extending the analysis to the four- or higher-dimensional models.}
\item Because of the aforementioned complexity of the energy landscape,
the transition may encounter a dead-end with energy $\Gamma$, even
in the bulk part of the transition; this is not the case in the 2D
model. Therefore, another technical challenge is that of carefully
characterizing these dead-ends and appropriately excluding them from
the computation. 
\end{enumerate}
As explained above, the energy landscape of the 3D model is more complex
than that of the 2D one, and we are unable to present a complete description
of the energy landscape for the former. Nevertheless, we analyze the
landscape with the precision required to prove our main results.

In Section \ref{sec6}, we introduce canonical configurations and
paths. Their definitions are direct generalizations of those in the
2D model. Then, we explain several applications of these canonical
objects.

We first collect several notation which will be frequently used throughout
the remainder of the article.

\begin{figure}
\includegraphics[width=13cm]{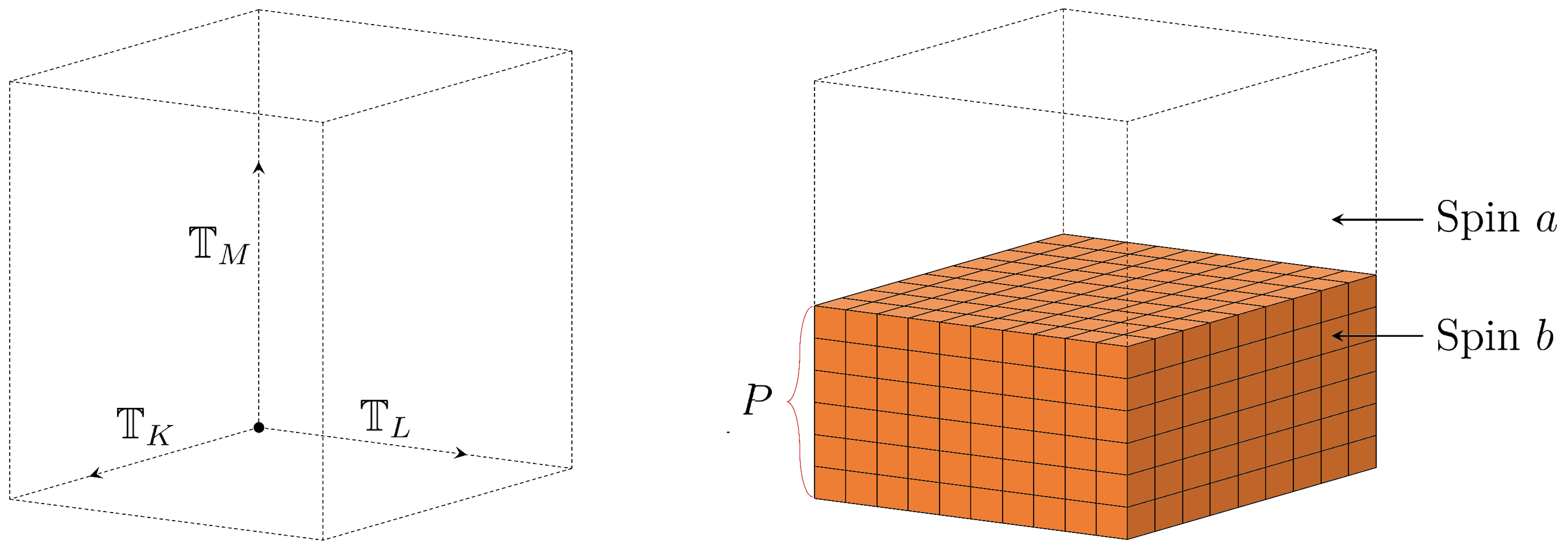}

\caption{\label{fig6.1}\textbf{Figures on Notation \ref{n_can}. }This form
of figure is used throughout the remainder of the article to illustrate
a 3D configuration consisting of two types of spins only. The large
dotted box denotes $\Lambda=\mathbb{T}_{K}\times\mathbb{T}_{L}\times\mathbb{T}_{M}$.
The orange unit boxes denote the sites with spin $b$, and the empty
part denotes the cluster of spin $a$. For some cases when we only
concern the shape of the cluster of spin $b$ (e.g. in Figure \ref{fig6.2}),
we omit the dotted box representing $\Lambda$.}
\end{figure}

\begin{notation}
\label{n_can}We refer to Figure \ref{fig6.1}\footnote{In fact, this figure and all the 3D figures below contradict our assumption
that $K\ge2829$. However, we believe that there will be absolutely
no confusion with these figures which only provide simple illustrations
of complicated notions.} for an illustration of the notation below.
\begin{itemize}
\item For $m\in\mathbb{T}_{M}$, the slab $\mathbb{T}_{K}\times\mathbb{T}_{L}\times\{m\}\subseteq\Lambda$
is called an \textit{$m$-th floor}. For each configuration $\sigma\in\mathcal{X}$,
we denote by $\sigma^{(m)}$ the configuration of $\sigma$ at the
$m$-th floor, i.e., 
\begin{equation}
\sigma^{(m)}(k,\,\ell)=\sigma(k,\,\ell,\,m)\;\;\;\;;\;k\in\mathbb{T}_{K}\;,\;\ell\in\mathbb{T}_{L}\;.\label{e_sigmam}
\end{equation}
Thus, $\sigma^{(m)}\in\mathcal{X}^{\mathrm{2D}}$ is a spin configuration
in $\Lambda^{\mathrm{2D}}=\mathbb{T}_{K}\times\mathbb{T}_{L}.$
\item For $a,\,b\in S$ and $P\subseteq\mathbb{T}_{M}$, we denote by $\sigma_{P}^{a,\,b}\in\mathcal{X}$
the configuration satisfying 
\begin{equation}
\sigma_{P}^{a,\,b}(k,\,\ell,\,m)=b\cdot\mathbf{1}\{m\in P\}+a\cdot\mathbf{1}\{m\notin P\}\;.\label{e_sigmaP}
\end{equation}
\end{itemize}
\end{notation}

\subsection{\label{sec6.1}Canonical configurations}

The following notation is used frequently.
\begin{notation}
\label{n_upe}We first introduce several maps on $\mathcal{X}$. If
$K=L$, we define a bijection $\Theta^{(12)}:\mathcal{X}\rightarrow\mathcal{X}$
as the map switching the first and second coordinates, i.e., for all
$\sigma\in\mathcal{X}$ and $(k,\,\ell,\,m)\in\Lambda$, 
\[
(\Theta^{(12)}(\sigma))(k,\,\ell,\,m)=\sigma(\ell,\,k,\,m)\;.
\]
If $L=M$, we can similarly define a bijection $\Theta^{(23)}$ on
$\mathcal{X}$ switching the second and third coordinates. Finally,
for the case of $K=L=M$, we can even define the bijection $\Theta^{(13)}$
on $\mathcal{X}$ switching the first and third coordinates.

Then, for $\mathcal{A}\subseteq\mathcal{X}$, we define $\Upsilon(\mathcal{A})$
as 
\[
\Upsilon(\mathcal{A})=\begin{cases}
\mathcal{A} & \text{if }K<L<M\;,\\
\mathcal{A}\cup\Theta^{(12)}(\mathcal{A}) & \text{if }K=L<M\;,\\
\mathcal{A}\cup\Theta^{(23)}(\mathcal{A}) & \text{if }K<L=M\;,\\
\begin{aligned} & \mathcal{A}\cup\Theta^{(12)}(\mathcal{A})\cup\Theta^{(23)}(\mathcal{A})\cup\Theta^{(13)}(\mathcal{A})\\
 & \cup(\Theta^{(12)}\circ\Theta^{(23)})(\mathcal{A})\cup(\Theta^{(23)}\circ\Theta^{(12)})(\mathcal{A})
\end{aligned}
 & \text{if }K=L=M\;.
\end{cases}
\]
Note that the set $\Upsilon(\mathcal{A})$ for the case of $K=L=M$
denotes the set of all configurations obtained by permuting the coordinates
of the configurations in $\mathcal{A}$.
\end{notation}

Now, we define canonical configurations of our 3D model. 

\begin{figure}
\includegraphics[width=14.5cm]{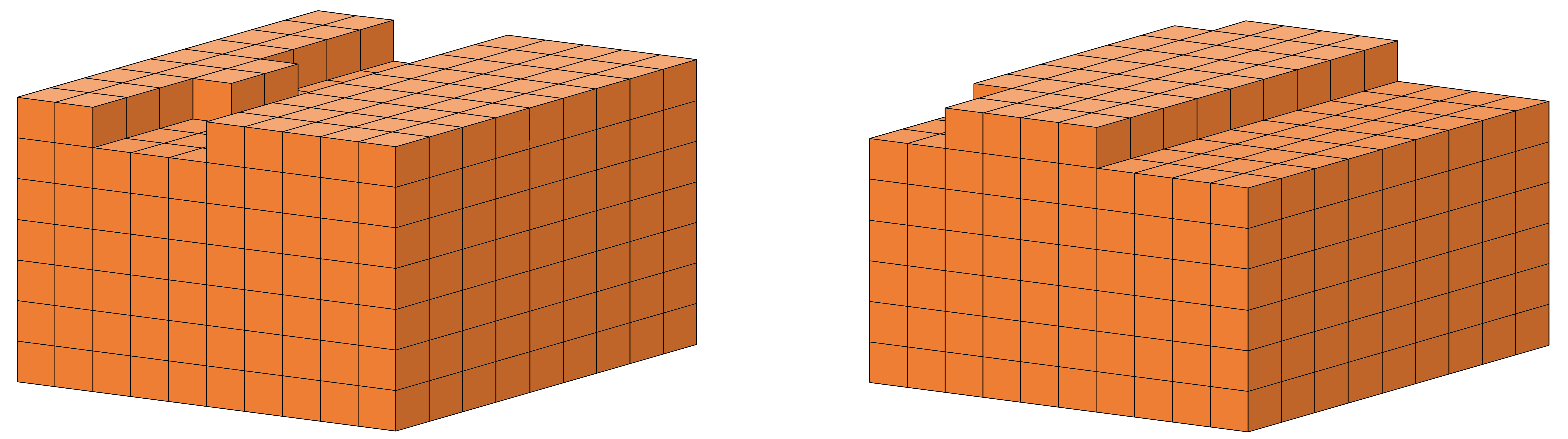}

\caption{\label{fig6.2}\textbf{Canonical configurations. }These two configurations
belong to $\mathcal{C}_{6}^{a,\,b}$ (if the orange boxes represent
the sites with spin $b$ as in Figure \ref{fig6.1}), since the 2D
configurations at the $7$-th floor are 2D canonical configurations
$\xi_{6,\,7;\,5,\,2}^{a,\,b,\,+}$ and $\xi_{3,\,4;\,2,\,6}^{a,\,b,\,-}$,
respectively.}
\end{figure}

\begin{defn}[Canonical configurations]
\label{d_can} We refer to Figure \ref{fig6.2} for a visualization
of the objects introduced below. Recall Notation \ref{n_frakS}.
\begin{enumerate}
\item We first introduce some building blocks in the definition of canonical
and gateway configurations. For $a,\,b\in S$ and $P,\,Q\in\mathfrak{S}_{M}$
with $P\prec Q$, we define $\widetilde{\mathcal{C}}_{P,\,Q}^{a,\,b}\subseteq\mathcal{X}$
as
\[
\sigma\in\widetilde{\mathcal{C}}_{P,\,Q}^{a,\,b}\;\;\;\Leftrightarrow\;\;\;\begin{cases}
\sigma^{(m)}=\mathbf{s}_{b}^{\mathrm{2D}} & \text{if }m\in P\;,\\
\sigma^{(m)}=\mathbf{s}_{a}^{\mathrm{2D}} & \text{if }m\in Q^{c}\;,\\
\sigma^{(m)}\in\mathcal{C}^{a,\,b,\,\mathrm{2D}} & \text{if }m\in Q\setminus P\;,
\end{cases}
\]
where the 2D objects are defined in Section \ref{sec5.2}. Then, we
set\textbf{
\begin{equation}
\mathcal{C}_{P,\,Q}^{a,\,b}=\Upsilon(\widetilde{\mathcal{C}}_{P,\,Q}^{a,\,b})\;.\label{e_can}
\end{equation}
}We then define, for $i\in\llbracket0,\,M-1\rrbracket$, 
\begin{align}
\mathcal{C}_{i}^{a,\,b} & =\bigcup_{P,\,Q\in\mathfrak{S}_{M}:\,|P|=i\text{ and }P\prec Q}\mathcal{C}_{P,\,Q}^{a,\,b}\;\;\;\;\text{and}\;\;\;\;\mathcal{C}^{a,\,b}=\bigcup_{i=0}^{M-1}\mathcal{C}_{i}^{a,\,b}\;.\label{e_can2}
\end{align}
Finally, for a proper partition $(A,\,B)$ of $S$, we write
\[
\mathcal{C}_{i}^{A,\,B}=\bigcup_{a\in A}\bigcup_{b\in B}\mathcal{C}_{i}^{a,\,b}\;\;\;\;\text{and}\;\;\;\;\mathcal{C}^{A,\,B}=\bigcup_{a\in A}\bigcup_{b\in B}\mathcal{C}^{a,\,b}\;.
\]
A configuration belonging to $\mathcal{C}^{a,\,b}$ for some $a,\,b\in S$
is called a \textit{canonical configuration }between $\mathbf{s}_{a}$
and $\mathbf{s}_{b}$. 
\end{enumerate}
\end{defn}

In view of the definition above, the role of the map $\Upsilon$ is
clear. When $K<L<M$ there is only one direction of transition, if
$K=L<M$ or $K<L=M$ there are $2=2!$ possible directions, while
if $K=L=M$ there are $6=3!$ possible directions. The map $\Upsilon$
reflects this observation into the definition. Next, let us define
regular configurations which are the special ones among the canonical
configurations. 
\begin{defn}[Regular configurations]
 For $a,\,b\in S$ and $P\in\mathfrak{S}_{M}$, recall the configuration
$\sigma_{P}^{a,\,b}$ from \eqref{e_sigmaP} and define
\begin{equation}
\widetilde{\mathcal{R}}_{i}^{a,\,b}=\{\sigma_{P}^{a,\,b}:P\in\mathfrak{S}_{M},\;|P|=i\}\;\;\;\;;\;i\in\llbracket0,\,M\rrbracket\;.\label{e_reg1}
\end{equation}
Note that $\widetilde{\mathcal{R}}_{i}^{a,\,b}$ is a collection of
configurations consisting of spins $a$ and $b$ only, where spins
$a$ and $b$ are located at slabs $\mathbb{T}_{K}\times\mathbb{T}_{L}\times(\mathbb{T}_{M}\setminus P)$
and $\mathbb{T}_{K}\times\mathbb{T}_{L}\times P$, respectively, for
$P\in\mathfrak{S}_{M}$ with $|P|=i$. Then, define (cf. Notation
\ref{n_upe}) 
\begin{gather}
\mathcal{R}_{i}^{a,\,b}=\Upsilon(\widetilde{\mathcal{R}}_{i}^{a,\,b})\;.\label{e_reg2}
\end{gather}
A configuration belonging to $\mathcal{R}_{i}^{a,\,b}$ for some $i\in\llbracket0,\,M\rrbracket$
is called a \textit{regular configuration}. Clearly, we have $\mathcal{R}_{0}^{a,\,b}=\{\mathbf{s}_{a}\}$
and $\mathcal{R}_{M}^{a,\,b}=\{\mathbf{s}_{b}\}$. For a proper partition
$(A,\,B)$ of $S$, we write
\begin{equation}
\mathcal{R}_{i}^{A,\,B}=\bigcup_{a\in A}\bigcup_{b\in B}\mathcal{R}_{i}^{a,\,b}\;.\label{e_reg3}
\end{equation}
\end{defn}

\subsection{\label{sec6.2}Energy of canonical configurations}

One can compute the energy of canonical configurations readily by
elementary computations, but we provide a more systematic approach
which will be frequently used in later computations. To this end,
we first introduce a notation.
\begin{notation}
\label{n_sigmakl}For $(k,\,\ell)\in\mathbb{T}_{K}\times\mathbb{T}_{L}$,
we denote by $\sigma^{\langle k,\,\ell\rangle}\in S^{\mathbb{T}_{M}}$
the configuration of $\sigma$ on the $(k,\,\ell)$-th pillar $\{k\}\times\{\ell\}\times\mathbb{T}_{M}$,
i.e., 
\begin{equation}
\sigma^{\langle k,\,\ell\rangle}(m)=\sigma(k,\,\ell,\,m)\;\;\;;\;m\in\mathbb{T}_{M}\;.\label{e_sigmakl}
\end{equation}
The energy of the one-dimensional (1D) configuration $\sigma^{\langle k,\,\ell\rangle}$
is denoted by 
\begin{equation}
H^{\mathrm{1D}}(\sigma^{\langle k,\,\ell\rangle})=\sum_{m\in\mathbb{T}_{M}}\mathbf{1}\{\sigma(k,\,\ell,\,m)\neq\sigma(k,\,\ell,\,m+1)\}\;.\label{e_Ham1D}
\end{equation}
\end{notation}

In the following lemma, we decompose the 3D energy into lower-dimensional
ones. 
\begin{lem}
\label{l_decH}For each $\sigma\in\mathcal{X}$, it holds that 
\begin{equation}
H(\sigma)=\sum_{m\in\mathbb{T}_{M}}H^{\mathrm{2D}}(\sigma^{(m)})+\sum_{(k,\,\ell)\in\mathbb{T}_{K}\times\mathbb{T}_{L}}H^{\mathrm{1D}}(\sigma^{\langle k,\,\ell\rangle})\;.\label{e_decH}
\end{equation}
\end{lem}

\begin{proof}
We can write $H(\sigma)$ as 
\begin{align*}
 & \sum_{m\in\mathbb{T}_{M}}\Big[\,\sum_{k\in\mathbb{T}_{K}}\sum_{\ell\in\mathbb{T}_{L}}\mathbf{1}\{\sigma(k+1,\,\ell,\,m)\neq\sigma(k,\,\ell,\,m)\}+\mathbf{1}\{\sigma(k,\,\ell+1,\,m)\neq\sigma(k,\,\ell,\,m)\}\,\Big]\\
 & +\sum_{k\in\mathbb{T}_{K}}\sum_{\ell\in\mathbb{T}_{L}}\Big[\,\sum_{m\in\mathbb{T}_{M}}\mathbf{1}\{\sigma(k,\,\ell,\,m)\neq\sigma(k,\,\ell,\,m+1)\}\,\Big]\;.
\end{align*}
The first and second lines correspond to the first and second terms
at the right-hand side of \eqref{e_decH}, respectively.
\end{proof}
Based on the previous expression, we deduce the following proposition. 
\begin{prop}[Energy of canonical configurations]
\label{p_energy} The following properties hold.
\begin{enumerate}
\item For each canonical configuration $\sigma$, we have $H(\sigma)\le\Gamma$.
\item For each configuration $\sigma\in\mathcal{C}_{i}^{a,\,b}$ for some
$a,\,b\in S$ and $i\in\llbracket1,\,M-2\rrbracket$, we have 
\[
H(\sigma)\in\llbracket\Gamma-2K-2,\,\Gamma\rrbracket\;.
\]
\end{enumerate}
\end{prop}

\begin{proof}
Observe that for a canonical configuration $\sigma$, we have $H_{1}(\sigma^{\langle k,\,\ell\rangle})\le2$
for all $(k,\,\ell)\in\mathbb{T}_{K}\times\mathbb{T}_{L}$, and $H_{2}(\sigma^{(m)})=0$
for all $m\in\mathbb{T}_{M}\setminus\{m_{0}\}$ for some $m_{0}\in\mathbb{T}_{M}$,
at which it holds that $H_{2}(\sigma^{(m_{0})})\le2K+2$ (cf. \eqref{e_2Denergy}).
Thus, by Lemma \ref{l_decH}, 
\[
H(\sigma)\le(2K+2)+2KL=\Gamma\;.
\]
For part (2), it suffices to additionally observe that $H_{1}(\sigma^{\langle k,\,\ell\rangle})=2$
for all $(k,\,\ell)\in\mathbb{T}_{K}\times\mathbb{T}_{L}$ if $i\in\llbracket1,\,M-2\rrbracket$
and thus 
\[
H(\sigma)\ge2KL=\Gamma-2K-2\;.
\]
\end{proof}
\begin{rem}
\label{r_energy}In particular, we have $H(\sigma)=\Gamma-2K-2=2KL$
for any $\sigma\in\mathcal{R}_{i}^{A,\,B}$, $i\in\llbracket1,\,M-1\rrbracket$.
Hence, a $\Gamma$-path at a regular configuration can evolve in a
non-canonical way, since we still have a spare of $2K+2$ to reach
the energy barrier $\Gamma$. Incorporating all these behaviors in
the metastability analysis is a demanding part of the 3D model. For
this reason, the regular configuration plays a crucial role. We remark
that for the 2D case \cite{BGN,KS 2D,NZ}, any optimal path at a regular
configuration does not have freedom, and that helped a lot simplifying
the arguments.
\end{rem}

\subsection{\label{sec6.3}Canonical paths }

In this subsection, we define 3D canonical paths between ground states.
They generalize the 2D paths recalled in Definition \ref{d_canpath2}.
Refer to Figure \ref{fig6.3} for an illustration.

\begin{figure}
\includegraphics[width=14.5cm]{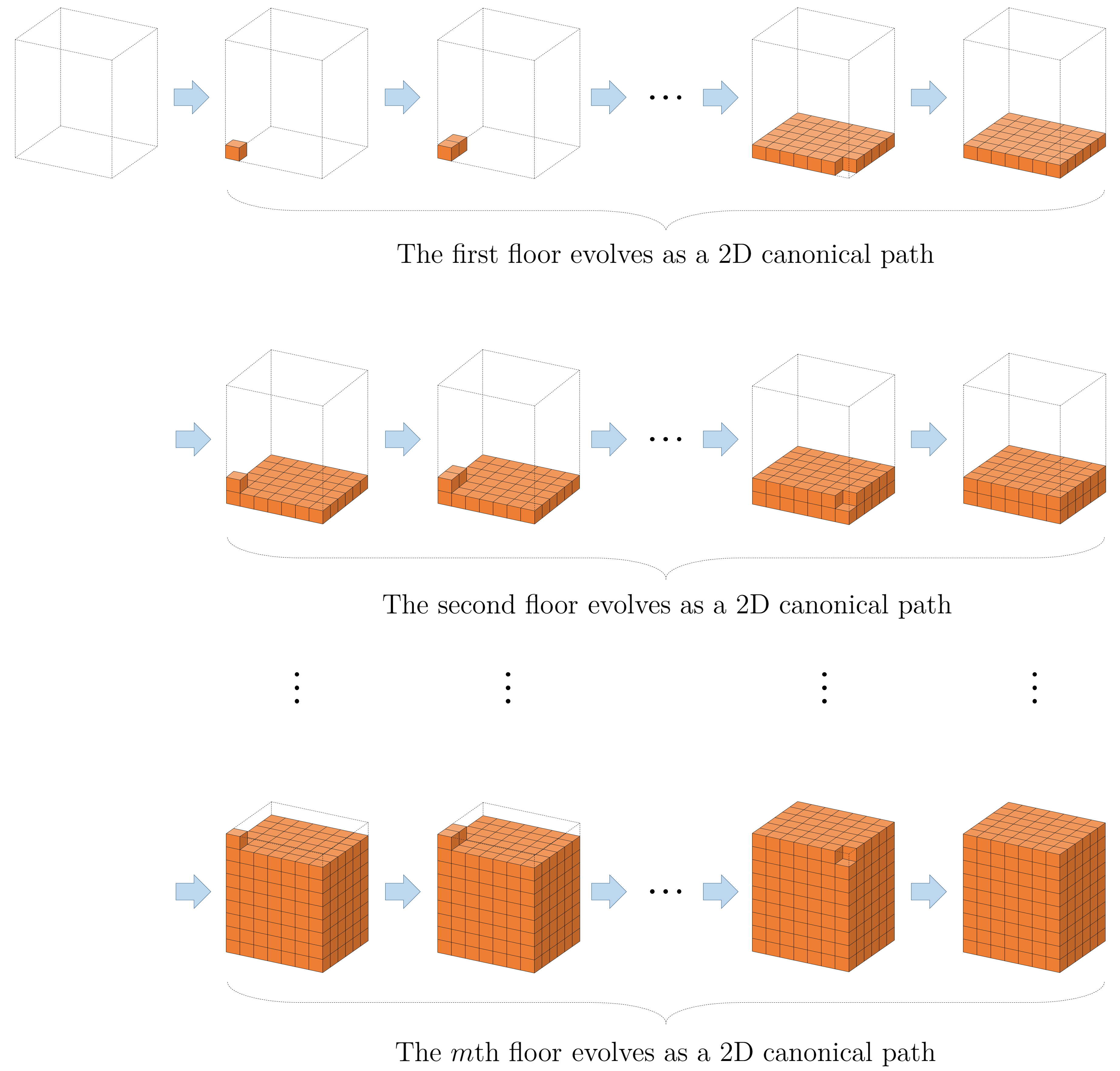}

\caption{\label{fig6.3}\textbf{Canonical path connecting $\mathbf{s}_{a}$
and $\mathbf{s}_{b}$.} }
\end{figure}

\begin{defn}[Canonical paths]
\label{d_canpath} We recall Notation \ref{n_frakS}. Let us fix
$a,\,b\in S$. A path $(\omega_{t})_{t=0}^{KLM}$ is called a \textit{pre-canonical
path} connecting $\mathbf{s}_{a}$ and $\mathbf{s}_{b}$ if there
exists an increasing sequence $(P_{i})_{i=0}^{M}$ in $\mathfrak{S}_{M}$
such that
\begin{itemize}
\item for each $i\in\llbracket0,\,M\rrbracket$, we have that $\omega_{KLi}=\sigma_{P_{i}}^{a,\,b}$
(cf. \eqref{e_sigmaP}), and
\item for each $i\in\llbracket0,\,M-1\rrbracket$, there exists a 2D canonical
path $(\gamma_{t}^{i})_{t=0}^{KL}$ from $\mathbf{s}_{a}^{\mathrm{2D}}$
to $\mathbf{s}_{b}^{\mathrm{2D}}$ defined in Definition \ref{d_canpath2}
such that 
\[
\omega_{t}^{(m)}=\begin{cases}
\mathbf{s}_{b}^{\mathrm{2D}} & \text{if }m\in P_{i}\;,\\
\mathbf{s}_{a}^{\mathrm{2D}} & \text{if }m\in\mathbb{T}_{M}\setminus P_{i+1}\;,\\
\gamma_{t-KLi}^{i} & \text{if }m\in P_{i+1}\setminus P_{i}\;,
\end{cases}\;\;\;\;\text{for all }t\in\llbracket KLi,\,KL(i+1)\rrbracket\;.
\]
\end{itemize}
If $K<L<M$, a path is called a \textit{canonical path} if it is a
pre-canonical path. If $K=L<M$, a path is called a \textit{canonical
path} if it is either a pre-canonical one or the image of a pre-canonical
one with respect to the map $\Theta^{(12)}.$ We can define canonical
paths for the cases of $K<L=M$ and $K=L=M$ in a similar manner. 
\end{defn}

\begin{rem}
\label{r_canpath}We emphasize that for a canonical path $(\omega_{t})_{t=0}^{KLM}$,
all configurations $\omega_{t}$, $t\in\llbracket0,\,KLM\rrbracket$,
are canonical configurations, and hence any canonical path is a $\Gamma$-path
by part (1) of Proposition \ref{p_energy}.
\end{rem}

Canonical paths provide optimal paths between two ground states, and
hence we can confirm the following upper bound for the energy barrier.
\begin{prop}
\label{p_Eub}For $\mathbf{s},\,\mathbf{s}'\in\mathcal{S}$, we have
that $\Phi(\mathbf{s},\,\mathbf{s}')\le\Gamma$.
\end{prop}

\begin{proof}
By Remark \ref{r_canpath}, it suffices to take a canonical path connecting
$\mathbf{s}$ and $\mathbf{s}'$.
\end{proof}
We prove $\Phi(\mathbf{s},\,\mathbf{s}')\ge\Gamma$ in Section \ref{sec8}
to verify $\Phi(\mathbf{s},\,\mathbf{s}')=\Gamma$. This reversed
inequality requires a much more complicated proof.

\subsection{\label{sec6.4}Characterization of the deepest valleys}

We show in this subsection that using the canonical paths, the valleys
in the energy landscape, except for the ones associated to the ground
states, have depths less than $\Gamma$. Note that Theorem \ref{t_energy barrier},
although not yet proved, indicates that the valleys associated to
the ground states have depth $\Gamma$. This characterization of the
depths of other valleys is essentially required since we have to reject
the possibility of being trapped in a deeper valley in the course
of transition. This fact is crucially used in the application of the
pathwise approach to metastability.
\begin{notation}
\label{n_pseudo}For the convenience of notation, we call $(\omega_{t})_{t=0}^{T}$
a pseudo-path if either $\omega_{t}\sim\omega_{t+1}$ or $\omega_{t}=\omega_{t+1}$
for all $t\in\llbracket0,\,T-1\rrbracket$. 
\end{notation}

\begin{prop}
\label{p_depth}For $\sigma\in\mathcal{X}\setminus\mathcal{S}$, we
have 
\[
\Phi(\sigma,\,\mathcal{S})-H(\sigma)\le\Gamma-2<\Gamma\;.
\]
\end{prop}

\begin{proof}
Main idea of the proof is inherited from the proof of \cite[Theorem 2.1]{NZ}.
Let us find two spins $a,\,b\in S$ so that $\sigma$ has spins $a$
and $b$ at some sites, which is clearly possible since $\sigma\notin\mathcal{S}$.
Let us fix a canonical path $(\omega_{t})_{t=0}^{KLM}$ connecting
$\mathbf{s}_{a}$ and $\mathbf{s}_{b}$. Then, we write 
\[
A_{t}=\{x\in\Lambda:\omega_{t}(x)=b\}\;\;\;\;;\;t\in\llbracket0,\,KLM\rrbracket\;,
\]
so that we have $\emptyset=A_{0}\subseteq A_{1}\subseteq\cdots\subseteq A_{KLM}=\Lambda$
and $|A_{t}|=t$ for all $t\in\llbracket0,\,KLM\rrbracket$. We can
take the path $(\omega_{t})_{t=0}^{KLM}$ in a way that
\begin{equation}
A_{1}=\{x_{0}\}\;\;\;\;\text{and}\;\;\;\;\sigma(x_{0})=b\;.\label{e_depth}
\end{equation}
Now, we define a pseudo-path (cf. Notation \ref{n_pseudo}) $(\widetilde{\omega}_{t})_{t=0}^{KLM}$
connecting $\sigma$ and $\mathbf{s}_{b}$ as 
\[
\widetilde{\omega}_{t}(x)=\begin{cases}
\sigma(x) & \text{if }x\notin A_{t}\;,\\
b & \text{if }x\in A_{t}\;.
\end{cases}
\]
In other words, we update the spins in an exactly same manner with
the canonical path $(\omega_{t})_{t=0}^{KLM}$. We claim that 
\begin{equation}
H(\widetilde{\omega}_{t})-H(\sigma)\le2KL+2K=\Gamma-2\;\;\;\;\text{for all }t\in\llbracket0,\,KLM\rrbracket\;.\label{e_depth2}
\end{equation}
It is immediate that this claim concludes the proof. To prove this
claim, we recall the decomposition obtained in Lemma \ref{l_decH}
and write $\widetilde{\omega}_{t}=\zeta$. Then, we can write $H(\zeta)-H(\sigma)$
as
\begin{equation}
\sum_{m\in\mathbb{T}_{M}}[H^{\mathrm{2D}}(\zeta^{(m)})-H^{\mathrm{2D}}(\sigma^{(m)})]+\sum_{(k,\,\ell)\in\mathbb{T}_{K}\times\mathbb{T}_{L}}[H^{\mathrm{1D}}(\zeta^{\langle k,\,\ell\rangle})-H^{\mathrm{1D}}(\sigma^{\langle k,\,\ell\rangle})]\;.\label{e_depth3}
\end{equation}
Let us first consider the first summation of \eqref{e_depth3}. We
suppose that $t\in\llbracket KLi,\,KL(i+1)\rrbracket$ and write $\omega_{KLi}=\sigma_{P}^{a,\,b}$
and $\omega_{KL(i+1)}=\sigma_{Q}^{a,\,b}$ where $P\prec Q$. Write
$Q\setminus P=\{m'\}$. Then, we have that 
\begin{equation}
H^{\mathrm{2D}}(\zeta^{(m)})-H^{\mathrm{2D}}(\sigma^{(m)})=\begin{cases}
-H^{\mathrm{2D}}(\sigma^{(m)})\le0 & \text{if }m\in P\;,\\
0 & \text{if }m\in Q^{c}\;,
\end{cases}\label{e_depth4}
\end{equation}
since $\zeta^{(m)}=\mathbf{s}_{b}^{\mathrm{2D}}$ for $m\in P$ and
$\zeta^{(m)}=\sigma^{(m)}$ for $m\in Q^{c}$. On the other hand,
by Lemma \ref{l_depth2}, we have that 
\begin{equation}
H^{\mathrm{2D}}(\zeta^{(m')})-H^{\mathrm{2D}}(\sigma^{(m')})\le2K+2\;.\label{e_depth5}
\end{equation}
By \eqref{e_depth4} and \eqref{e_depth5}, we conclude that 
\begin{equation}
\sum_{m\in\mathbb{T}_{M}}[H^{\mathrm{2D}}(\zeta^{(m)})-H^{\mathrm{2D}}(\sigma^{(m)})]\le2K+2\;.\label{e_depth6}
\end{equation}
Now, we turn to the second summation of \eqref{e_depth3}. Note that
$\zeta^{\langle k,\,\ell\rangle}$ is obtained from $\sigma^{\langle k,\,\ell\rangle}$
by flipping the spins in consecutive sites in $Q$ to $b$. From this,
we can readily deduce that
\begin{equation}
H^{\mathrm{1D}}(\zeta^{\langle k,\,\ell\rangle})-H^{\mathrm{1D}}(\sigma^{\langle k,\,\ell\rangle})\le2\;\;\;\;\text{for all }k\in\mathbb{T}_{K}\text{ and }\ell\in\mathbb{T}_{L}\;.\label{e_depth7}
\end{equation}
Moreover, if $x_{0}=(k_{0},\,\ell_{0},\,m_{0})$, we can check that
\begin{equation}
H^{\mathrm{1D}}(\zeta^{\langle k_{0},\,\ell_{0}\rangle})-H^{\mathrm{1D}}(\sigma^{\langle k_{0},\,\ell_{0}\rangle})\le0\;.\label{e_depth8}
\end{equation}
By \eqref{e_depth7} and \eqref{e_depth8}, we get 
\begin{equation}
\sum_{(k,\,\ell)\in\mathbb{T}_{K}\times\mathbb{T}_{L}}[H^{\mathrm{1D}}(\zeta^{\langle k,\,\ell\rangle})-H^{\mathrm{1D}}(\sigma^{\langle k,\,\ell\rangle})]\le2(KL-1)\;.\label{e_depth9}
\end{equation}
Now, the claim \eqref{e_depth2} follows from \eqref{e_depth3}, \eqref{e_depth6},
and \eqref{e_depth9}.
\end{proof}

\subsection{\label{sec6.5}Auxiliary result on saddle configurations}

In the 2D case, in the analysis of the energy landscape, the collection
$\mathcal{R}_{2}^{\mathrm{2D}}$ plays a significant role since to
make an optimal transition (not exceeding the energy barrier $2K+2$),
we may skip the collection $\mathcal{R}_{1}^{\mathrm{2D}}$ but must
pass through $\mathcal{R}_{2}^{\mathrm{2D}}$. Thus, the integer $2$
worked as some kind of a threshold for metastable transitions. We
expect a similar pattern in the 3D case, and we briefly explain this
phenomenon in this subsection.

Let us define
\begin{equation}
\mathfrak{m}_{K}=\lfloor K^{2/3}\rfloor\;.\label{e_mK}
\end{equation}
Then, we shall prove in Corollary \ref{c_H3lb} below that
\begin{equation}
\Phi(\mathbf{s}_{a},\,\sigma_{\llbracket1,\,n\rrbracket}^{a,\,b})=\Gamma\;\;\;\;\text{for all }n\in\llbracket\mathfrak{m}_{K},\,M-\mathfrak{m}_{K}\rrbracket\;.\label{e_mKprop}
\end{equation}
Thus, we can define (cf. Figure \ref{fig7.2} below)
\begin{equation}
n_{K,\,L,\,M}=\min\,\{n\in\llbracket1,\,M-1\rrbracket:\Phi(\mathbf{s}_{a},\,\sigma_{\llbracket1,\,n\rrbracket}^{a,\,b})=\Gamma\}\;.\label{e_nKLM}
\end{equation}
We strongly believe that this quantity does not depend on $M$, but
we do not have a proof for it at the moment. Note that this number
was just $2$ in the 2D case. In the 3D model, we do not know this
number exactly, since non-canonical movements at the early stage of
transitions are hard to characterize. However, the upper bound $n_{K,\,L,\,M}\le\mathfrak{m}_{K}=\lfloor K^{2/3}\rfloor$
obtained from \eqref{e_mKprop} is enough for our purpose, as we shall
see later.

The main result of this subsection is the corresponding lower bound.
This result will not be used in the proofs later, but emphasizes the
complexity of the energy landscape near ground states. 
\begin{prop}
\label{p_nKLMlb}We have $n_{K,\,L,\,M}\ge\lfloor K^{1/2}\rfloor$.
\end{prop}

\begin{proof}
It suffices to prove that 
\[
\Phi(\mathbf{s}_{1},\,\sigma_{\llbracket1,\,n\rrbracket}^{1,\,2})\le\Gamma-2\;\;\;\;\text{ for all }n\in\llbracket1,\,\lfloor K^{1/2}\rfloor-1\rrbracket\;.
\]
We fix such an $n$ and write $\sigma=\sigma_{\llbracket1,\,n\rrbracket}^{1,\,2}$.
We now construct an explicit path from $\sigma$  to $\mathbf{s}_{1}$
without exceeding the energy $\Gamma-2$. Note that $\sigma_{\llbracket1,\,n\rrbracket}^{1,\,2}$
has spins $2$ at $\mathbb{T}_{K}\times\mathbb{T}_{L}\times\llbracket1,\,n\rrbracket$
and spins $1$ at all the other sites. In this proof, we regard $\mathbb{T}_{K}=\llbracket1,\,K\rrbracket$
and $\mathbb{T}_{L}=\llbracket1,\,L\rrbracket$ in order to simplify
the explanation of the order of spin flips in a lexicographic manner.
\begin{itemize}
\item First, starting from $\sigma$, we change spins $2$ to $1$ in $\llbracket1,\,K\rrbracket\times\llbracket1,\,n\rrbracket\times\llbracket1,\,n\rrbracket$
in ascending lexicographic order. Denote by $\zeta\in\mathcal{X}$
the obtained spin configuration, which has spins $2$ only on $\llbracket1,\,K\rrbracket\times\llbracket n+1,\,L\rrbracket\times\llbracket1,\,n\rrbracket$.
Then, the variation of the Hamiltonian from $\sigma$ to $\zeta$
can be expressed by the following $n\times n$ matrices:
\[
\begin{bmatrix}+2\;+0\;\cdots\;+0\\
+4\;+2\;\cdots\;+2\\
\vdots\\
+4\;+2\;\cdots\;+2
\end{bmatrix}\;,\;\;\;\begin{bmatrix}+0\;-2\;\cdots\;-2\\
\mathbf{+2}\;+0\;\cdots\;+0\\
\vdots\\
+2\;+0\;\cdots\;+0
\end{bmatrix}\times(K-2)\;,\;\;\;\text{and }\begin{bmatrix}-2\;-4\;\cdots\;-4\\
+0\;-2\;\cdots\;-2\\
\vdots\\
+0\;-2\;\cdots\;-2
\end{bmatrix}\;.
\]
Here, each $n\times n$ matrix represents $\{i\}\times\llbracket1,\,n\rrbracket\times\llbracket1,\,n\rrbracket$
for $1\le i\le K$, in which the numbers represent the variation of
the energy which should be read in ascending lexicographic order.
From this path, we obtain
\begin{equation}
\Phi(\sigma,\,\zeta)\le2KL+2n^{2}+2n-2\;,\label{e_nKLMlb}
\end{equation}
where the maximum of the energy is obtained right after flipping the
spin at $(2,\,1,\,n-1)$, which is denoted by bold font at the matrices
above.
\item Next, starting from $\zeta$, we change spins $2$ to $1$ in $\llbracket1,\,K\rrbracket\times\{i\}\times\llbracket1,\,n\rrbracket$
in the ascending lexicographic order for $i\in\llbracket n+1,\,L-1\rrbracket$,
from $i=n+1$ to $i=L-1$. Denote by $\zeta'\in\mathcal{X}$ the obtained
spin configuration, which has spins $2$ only on $\llbracket1,\,K\rrbracket\times\{L\}\times\llbracket1,\,n\rrbracket$.
In each step, the variation of the Hamiltonian is represented by the
$n\times K$ matrix
\[
\begin{bmatrix}+0\;-2\;\cdots\;-2\;-4\\
+2\;+0\;\cdots\;+0\;-2\\
\vdots\\
\mathbf{+2}\;+0\;\cdots\;+0\;-2
\end{bmatrix}\;.
\]
Since $H(\zeta)=2KL$, we can verify that 
\begin{equation}
\Phi(\zeta,\,\zeta')\le2KL+2\;,\label{e_nKLMlb2}
\end{equation}
where the maximum is obtained right after flipping the spin at $(1,\,n+1,\,1)$
(cf. bold font $\mathbf{+2}$).
\item Finally, starting from $\zeta'$, we change spins $2$ to $1$ in
the ascending lexicographic order. The variation of the Hamiltonian
is represented by
\[
\begin{bmatrix}-2\;-4\;\cdots\;-4\;-6\\
+0\;-2\;\cdots\;-2\;-4\\
\vdots\\
+0\;-2\;\cdots\;-2\;-4
\end{bmatrix}\;.
\]
Hence, the Hamiltonian monotonically decreases from $H(\zeta')=2K(n+1)$
to arrive at $H(\mathbf{s}_{1})=0$. Hence, we have
\begin{equation}
\Phi(\zeta',\,\mathbf{s}_{1})\le2K(n+1)\;.\label{e_nKLMlb3}
\end{equation}
\end{itemize}
Therefore, by \eqref{e_nKLMlb}, \eqref{e_nKLMlb2}, and \eqref{e_nKLMlb3},
we have
\[
\Phi(\sigma,\,\mathbf{s}_{1})\le2KL+2n^{2}+2n-2\;.
\]
Since $n\in\llbracket1,\,\lfloor K^{1/2}\rfloor-1\rrbracket$, it
holds that $2n^{2}+2n-2\le2K$. This concludes the proof.
\end{proof}

\section{\label{sec7}Gateway Configurations}

In the analysis of the 3D model, a crucial notion is the concept of
gateway configurations. The gateway configurations of the 3D model
play a far more significant role than those of the 2D model.

We fix a proper partition $(A,\,B)$ of $S$ throughout this section.

\subsection{\label{sec7.1}Gateway configurations}

We refer to Figure \ref{fig7.1} for an illustration of gateway configurations
defined below.

\begin{figure}
\includegraphics[width=14.5cm]{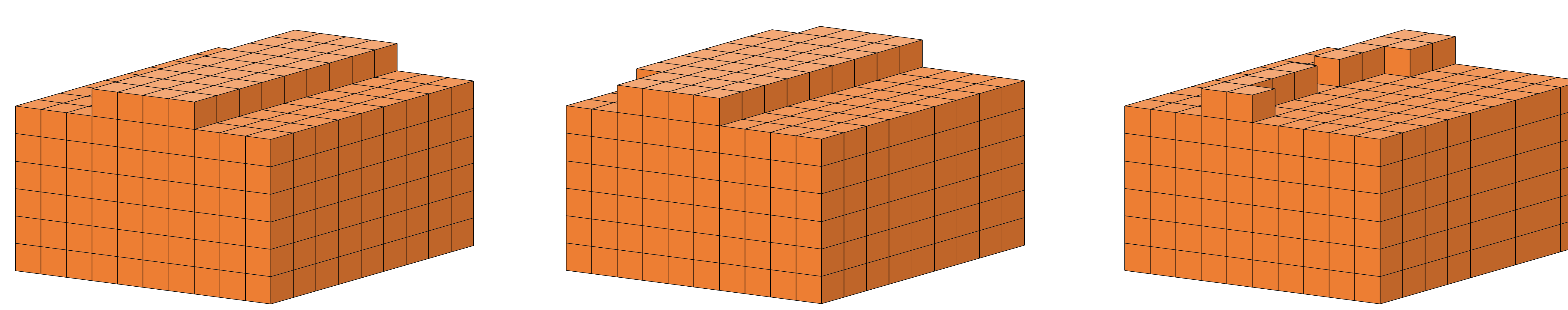}\caption{\label{fig7.1}\textbf{Examples of gateway configurations.} Each configuration
above represents a gateway configuration of type $1$ (left), type
$2$ (middle), or type $3$ (right), respectively.}
\end{figure}

\begin{defn}[Gateway configurations]
\label{d_gate} For $a,\,b\in S$ and $P,\,Q\in\mathfrak{S}_{M}$
with $P\prec Q$, we define $\mathcal{\widetilde{G}}_{P,\,Q}^{a,\,b}\subseteq\widetilde{\mathcal{C}}_{P,\,Q}^{a,\,b}$
as 
\[
\sigma\in\widetilde{\mathcal{G}}_{P,\,Q}^{a,\,b}\;\;\;\Leftrightarrow\;\;\;\begin{cases}
\sigma^{(m)}=\mathbf{s}_{b}^{\mathrm{2D}} & \text{if }m\in P\;,\\
\sigma^{(m)}=\mathbf{s}_{a}^{\mathrm{2D}} & \text{if }m\in Q^{c}\;,\\
\sigma^{(m)}\in\mathcal{G}^{a,\,b,\,\mathrm{2D}} & \text{if }m\in Q\setminus P\;,
\end{cases}
\]
where $\mathcal{G}^{a,\,b,\,\mathrm{2D}}$ is defined in Definition
\ref{d_gate2}. Then, we define (cf. Notation \ref{n_upe})
\[
\mathcal{G}_{P,\,Q}^{a,\,b}=\Upsilon(\widetilde{\mathcal{G}}_{P,\,Q}^{a,\,b})\;.
\]
Then, recall $\mathfrak{m}_{K}$ from \eqref{e_mK} and define, for
$i\in\llbracket0,\,M-1\rrbracket$, 
\begin{align}
\mathcal{G}_{i}^{a,\,b} & =\bigcup_{P,\,Q\in\mathfrak{S}_{M}:\,|P|=i\text{ and }P\prec Q}\mathcal{G}_{P,\,Q}^{a,\,b}\;\;\;\;\text{and}\;\;\;\;\mathcal{G}^{a,\,b}=\bigcup_{i=\mathfrak{m}_{K}-1}^{M-\mathfrak{m}_{K}}\mathcal{G}_{i}^{a,\,b}\;.\label{e_gate}
\end{align}
Notice that the crucial difference between \eqref{e_gate} and \eqref{e_can2}
is the fact that the second union in \eqref{e_gate} is taken only
over $i\in\llbracket\mathfrak{m}_{K}-1,\,M-\mathfrak{m}_{K}\rrbracket$.
This is related to \eqref{e_mKprop}, and we give a more detailed
reasoning in Section \ref{sec7.2}. A configuration belonging to $\mathcal{G}^{a,\,b}$
for some $a,\,b\in S$ is called a \textit{gateway configuration}.

Finally, for a proper partition $(A,\,B)$ of $S$ (which is fixed
throughout the current section), we write for $i\in\llbracket0,\,M-1\rrbracket$,
\begin{equation}
\mathcal{G}_{i}^{A,\,B}=\bigcup_{a\in A}\bigcup_{b\in B}\mathcal{G}_{i}^{a,\,b}\;\;\;\;\text{and}\;\;\;\;\mathcal{G}^{A,\,B}=\bigcup_{a\in A}\bigcup_{b\in B}\mathcal{G}^{a,\,b}\;.\label{e_gate2}
\end{equation}
\end{defn}

\begin{notation}
\label{n_gate}For $a,\,b\in S$ and $P,\,Q\in\mathfrak{S}_{M}$ with
$P\prec Q$, $Q\setminus P=\{m_{0}\}$, and $|P|\in\llbracket\mathfrak{m}_{K}-1,\,M-\mathfrak{m}_{K}\rrbracket$,
we decompose 
\[
\widetilde{\mathcal{G}}_{P,\,Q}^{a,\,b}=\widetilde{\mathcal{G}}_{P,\,Q}^{a,\,b,\,[1]}\cup\widetilde{\mathcal{G}}_{P,\,Q}^{a,\,b,\,[2]}\cup\widetilde{\mathcal{G}}_{P,\,Q}^{a,\,b,\,[3]}\;,
\]
where (cf. \eqref{e_Bab}, \eqref{e_BabGamma}, and \eqref{e_gate2def})
\begin{align*}
\widetilde{\mathcal{G}}_{P,\,Q}^{a,\,b,\,[1]} & =\{\sigma\in\widetilde{\mathcal{G}}_{P,\,Q}^{a,\,b}:\sigma^{(m_{0})}\in\mathcal{B}^{a,\,b,\,\mathrm{2D}}\setminus\mathcal{B}_{\Gamma}^{a,\,b,\,\mathrm{2D}}\}\;,\\
\widetilde{\mathcal{G}}_{P,\,Q}^{a,\,b,\,[2]} & =\{\sigma\in\widetilde{\mathcal{G}}_{P,\,Q}^{a,\,b}:\sigma^{(m_{0})}\in\mathcal{B}_{\Gamma}^{a,\,b,\,\mathrm{2D}}\}\;,\\
\widetilde{\mathcal{G}}_{P,\,Q}^{a,\,b,\,[3]} & =\{\sigma\in\widetilde{\mathcal{G}}_{P,\,Q}^{a,\,b}:\sigma^{(m_{0})}\in\mathcal{Z}^{a,\,b,\,\mathrm{2D}}\cup\mathcal{Z}^{b,\,a,\,\mathrm{2D}}\}\;.
\end{align*}
Then, write $\mathcal{G}_{P,\,Q}^{a,\,b,\,[n]}=\Upsilon(\widetilde{\mathcal{G}}_{P,\,Q}^{a,\,b,\,[n]})$,
$n\in\{1,\,2,\,3\}$. A configuration $\sigma\in\mathcal{G}^{A,\,B}$
is called a \textit{gateway configuration of type $n$}, $n\in\{1,\,2,\,3\}$,\textit{
if }$\sigma\in\mathcal{G}_{P,\,Q}^{a,\,b,\,[n]}$ for some $a\in A$,
$b\in B$ and $P,\,Q\in\mathfrak{S}_{M}$ with $P\prec Q$.
\end{notation}

The following proposition is direct from the definition of gateway
configurations.
\begin{prop}
\label{p_gateE}For $\sigma\in\mathcal{G}^{A,\,B}$, we have $H(\sigma)\in\{\Gamma-2,\,\Gamma\}$.
Moreover, we have $H(\sigma)=\Gamma-2$ if and only if $\sigma$ is
a gateway configuration of type $1$ and $H(\sigma)=\Gamma$ if and
only if $\sigma$ is a gateway configuration of type $2$ or $3$.
\end{prop}

\begin{proof}
Let $\sigma\in\widetilde{\mathcal{G}}_{P,\,Q}^{a,\,b}$ for some $a\in A,\,b\in B$
and $P,\,Q\in\mathfrak{S}_{M}$ with $P\prec Q$, $Q\setminus P=\{m_{0}\}$,
and $|P|\in\llbracket\mathfrak{m}_{K}-1,\,M-\mathfrak{m}_{K}\rrbracket$.
Then, by Lemma \ref{l_decH}, we can write 
\[
H(\sigma)=H^{\mathrm{2D}}(\sigma^{(m_{0})})+2KL
\]
since $H^{\mathrm{2D}}(\sigma^{(m)})=0$ for all $m\neq m_{0}$ and
$H^{\mathrm{1D}}(\sigma^{\langle k,\,\ell\rangle})=2$ for all $k\in\mathbb{T}_{K}$
and $\ell\in\mathbb{T}_{L}$. Hence, by definition, we have 
\[
H(\sigma)=\begin{cases}
2KL+2K=\Gamma-2 & \text{if }\sigma\in\widetilde{\mathcal{G}}_{P,\,Q}^{a,\,b,\,[1]}\;,\\
2KL+2K+2=\Gamma & \text{if }\sigma\in\widetilde{\mathcal{G}}_{P,\,Q}^{a,\,b,\,[2]}\cup\widetilde{\mathcal{G}}_{P,\,Q}^{a,\,b,\,[3]}\;.
\end{cases}
\]
Since the Hamiltonian is invariant under $\Upsilon$, the proof is
completed.
\end{proof}

\subsection{\label{sec7.2}Properties of gateway configurations}

Next, we investigate several crucial properties of the gateway configurations
which will be used frequently in the following discussions. The following
notation will be useful in the remaining parts of the article. 
\begin{notation}
\label{n_union}For any integers $u,\,v$ such that $0\le u<v\le M$,
we write 
\[
\mathcal{K}_{[u,\,v]}^{a,\,b}=\bigcup_{i=u}^{v}\mathcal{K}_{i}^{a,\,b}\;\;\;\;\text{and}\;\;\;\;\mathcal{K}_{[u,\,v]}^{A,\,B}=\bigcup_{i=u}^{v}\mathcal{K}_{i}^{A,\,B}\;,
\]
where $\mathcal{K}\in\{\mathcal{C},\,\mathcal{G},\,\mathcal{R}\}$.
In particular, by \eqref{e_gate} and \eqref{e_gate2}, we can write
\begin{equation}
\mathcal{G}^{A,\,B}=\mathcal{G}_{[\mathfrak{m}_{K}-1,\,M-\mathfrak{m}_{K}]}^{A,\,B}\;.\label{e_gate3}
\end{equation}
\end{notation}

In this section, we focus on the relation between gateway configurations
and neighborhoods of regular configurations. We refer to Figure \ref{fig7.2}
for an illustration of the relations obtained in the current subsection. 

\begin{figure}
\includegraphics[width=14.5cm]{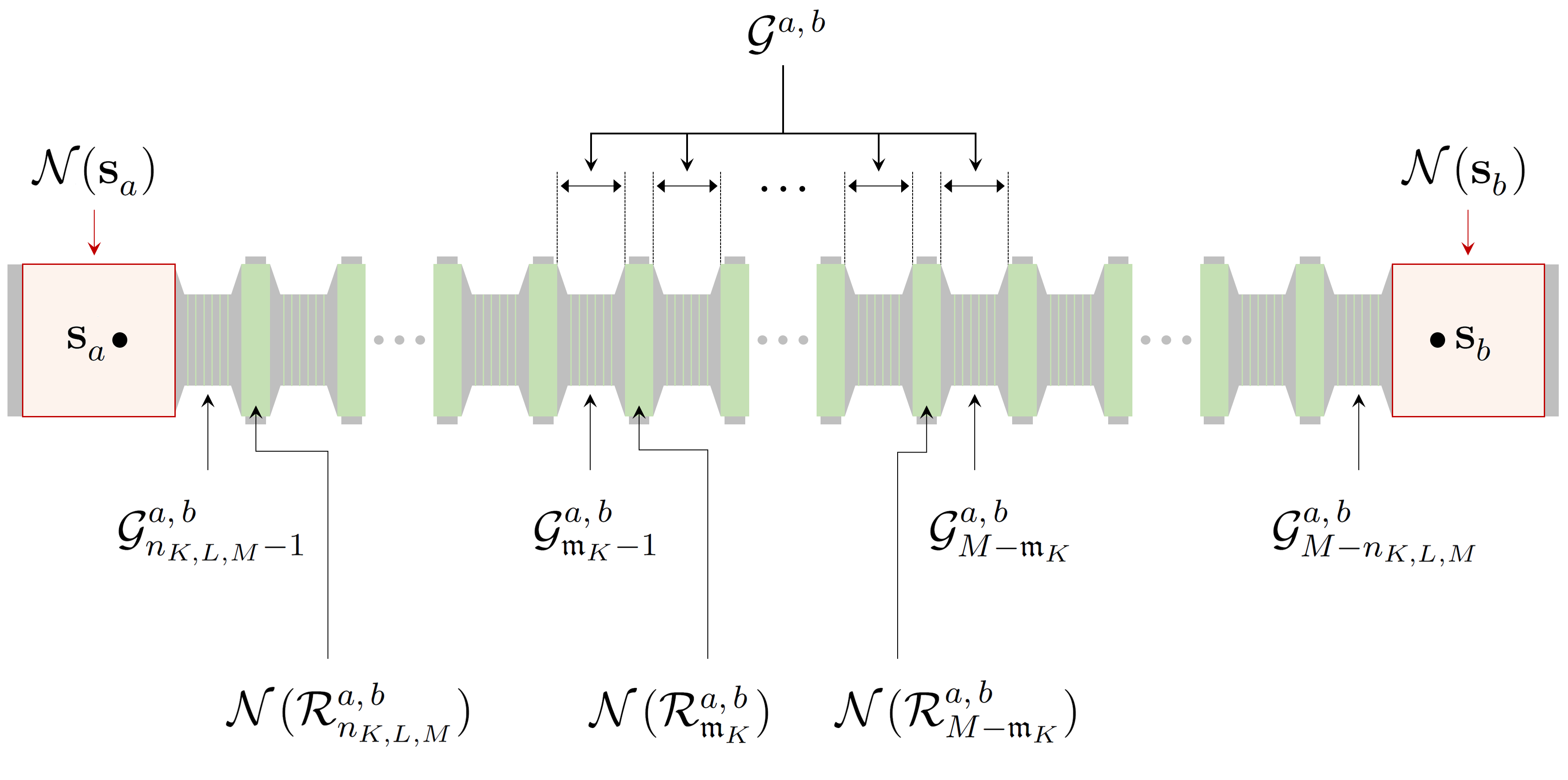}

\caption{\label{fig7.2}\textbf{Structure of gateway configurations between
$\mathbf{s}_{a}$ and $\mathbf{s}_{b}$.} The grey regions consist
of configurations of energy $\Gamma$. The green boxes denote the
sets of the form $\mathcal{N}(\mathcal{R}_{i}^{a,\,b})$ for $i\in\llbracket n_{K,\,L,\,M},\,M-n_{K,\,L,\,M}\rrbracket$
(cf. \eqref{e_nKLM}), while the green lines denote the gateway configurations
of type $1$ whose energy is $\Gamma-2$ (cf. Proposition \ref{p_gateE}).
Later in Proposition \ref{p_Elb}, we shall show that $\mathfrak{m}_{K}\ge n_{K,\,L,\,M}$.
The structure given in this figure (especially between $\mathcal{G}_{\mathfrak{m}_{K}-1}^{a,\,b}$
and $\mathcal{G}_{M-\mathfrak{m}_{K}}^{a,\,b}$) is confirmed in Lemma
\ref{l_gate}. We remark that the dead-ends are attached to $\mathcal{N}(\mathbf{s}_{a})$,
$\mathcal{N}(\mathbf{s}_{b})$, and $\mathcal{N}(\mathcal{R}_{i}^{a,\,b})$,
$i\in\llbracket n_{K,\,L,\,M},\,M-n_{K,\,L,\,M}\rrbracket$. In particular,
the configurations in $\mathcal{G}_{i}^{a,\,b}$ with $i<n_{K,\,L,\,M}-1$
belong to the dead-ends attached to the set $\mathcal{N}(\mathbf{s}_{a})$.}
\end{figure}

The first one below states that we have to escape from a gateway configuration
via a neighborhood of regular configurations, unless we touch a configuration
with energy higher than $\Gamma$.
\begin{lem}
\label{l_gate}For a proper partition $(A,\,B)$ of $S$, the following
statements hold.
\begin{enumerate}
\item For $a\in A$, $b\in B$, and $i\in\llbracket\mathfrak{m}_{K}-1,\,M-\mathfrak{m}_{K}\rrbracket$,
we suppose that $\sigma\in\mathcal{G}_{i}^{a,\,b}$ and $\zeta\in\mathcal{X}\setminus\mathcal{G}_{i}^{a,\,b}$
satisfy $\sigma\sim\zeta$ and $H(\zeta)\le\Gamma$. Then, we have
$\zeta\in\mathcal{N}(\mathcal{R}_{[i,\,i+1]}^{a,\,b})$, and moreover
$\sigma$ is a gateway configuration of type $3$.
\item Suppose that $\sigma\in\mathcal{G}^{A,\,B}$ and $\zeta\in\mathcal{X}\setminus\mathcal{G}^{A,\,B}$
satisfy $\sigma\sim\zeta$ and $H(\zeta)\le\Gamma$. Then, we have
$\zeta\in\mathcal{N}(\mathcal{R}_{[\mathfrak{m}_{K}-1,\,M-\mathfrak{m}_{K}+1]}^{A,\,B})$,
and moreover $\sigma$ is a gateway configuration of type $3$. 
\end{enumerate}
\end{lem}

\begin{proof}
We first suppose that $\sigma\in\mathcal{\widetilde{G}}_{P,\,Q}^{a,\,b}$
and $\zeta\in\mathcal{X}\setminus\mathcal{\widetilde{G}}_{P,\,Q}^{a,\,b}$
for some $a\in A$, $b\in B$ and $P,\,Q\in\mathfrak{S}_{M}$ with
$P\prec Q$ and $|P|\in\llbracket\mathfrak{m}_{K}-1,\,M-\mathfrak{m}_{K}\rrbracket$.
We write $Q\setminus P=\{m_{0}\}$. Then, we claim that $\zeta\in\mathcal{N}(\{\sigma_{P}^{a,\,b},\,\sigma_{Q}^{a,\,b}\})$,
and $\sigma$ is of type $3$.

Let us first show that $\sigma$ is a gateway configuration of type
$3$. If $\sigma$ is of type $1$, then we have $H(\sigma)=\Gamma-2$,
$H^{\mathrm{2D}}(\sigma^{(m_{0})})=2K$, and $\sigma^{(m_{0})}\in\mathcal{B}^{a,\,b,\,\mathrm{2D}}$.
To update a spin in $\sigma$ without increasing the energy by $3$
or more, it can be readily observed that we have to update a spin
of $\sigma$ at the $m_{0}$-th floor to get $\zeta$ with $H^{\mathrm{2D}}(\zeta^{(m_{0})})\le2K+2$.
In such a situation, Lemma \ref{l_gate2} asserts that $\sigma^{(m_{0})}\notin\mathcal{B}^{a,\,b,\,\mathrm{2D}}$
and we get a contradiction. A similar argument can be applied if $\sigma$
is of type $2$, and hence we can conclude that $\sigma$ is of type
$3$.

Now, since $\sigma$ is of type $3$, we have $H(\sigma)=\Gamma$,
$H^{\mathrm{2D}}(\sigma^{(m_{0})})=2K+2$, and $\sigma^{(m_{0})}\in\mathcal{Z}^{a,\,b,\,\mathrm{2D}}\cup\mathcal{Z}^{b,\,a,\,\mathrm{2D}}$
(cf. \eqref{e_Zab2def}). In order not to increase the energy by flipping
a site of $\sigma,$ it is clear that we have to flip a spin at the
$m_{0}$-th floor (cf. Figure \ref{fig7.1}). This means that, by
Lemma \ref{l_gate2}, we have $\zeta^{(m_{0})}\in\mathcal{N}^{\mathrm{2D}}(\mathbf{s}_{a}^{\mathrm{2D}})\cup\mathcal{N}^{\mathrm{2D}}(\mathbf{s}_{b}^{\mathrm{2D}})$.
Now, we suppose first that $\zeta^{(m_{0})}\in\mathcal{N}^{\mathrm{2D}}(\mathbf{s}_{a}^{\mathrm{2D}})$.
Then, there exists a 2D $(2K+1)$-path $(\omega_{t})_{t=0}^{T}$ in
$\mathcal{X}^{\mathrm{2D}}=S^{\Lambda^{\mathrm{2D}}}$ such that $\omega_{0}=\mathbf{s}_{a}^{\mathrm{2D}}$
and $\omega_{T}=\zeta^{(m_{0})}$. Define a 3D path $(\widetilde{\omega}_{t})_{t=0}^{T}$
as 
\[
\widetilde{\omega}_{t}^{(m)}=\begin{cases}
\omega_{t}^{(m)} & \text{if }m=m_{0}\;,\\
\zeta^{(m)}=\sigma^{(m)} & \text{if }m\neq m_{0}\;.
\end{cases}
\]
Then, $(\widetilde{\omega}_{t})_{t=0}^{T}$ is a $(\Gamma-1)$-path
connecting $\sigma_{P}^{a,\,b}$ and $\zeta$, and thus we get $\zeta\in\mathcal{N}(\sigma_{P}^{a,\,b})$.
Similarly, we can deduce that $\zeta^{(m_{0})}\in\mathcal{N}^{\mathrm{2D}}(\mathbf{s}_{b}^{\mathrm{2D}})$
implies $\zeta\in\mathcal{N}(\sigma_{Q}^{a,\,b})$. This concludes
the proof of the claim. 

Now, we return to the lemma. For part (1), suppose that $\sigma\in\mathcal{G}_{P,\,Q}^{a,\,b}$
for some $a\in A$, $b\in B$ and $P,\,Q\in\mathfrak{S}_{M}$ with
$|P|=i\in\llbracket\mathfrak{m}_{K}-1,\,M-\mathfrak{m}_{K}\rrbracket\text{ and }P\prec Q$.
If $\sigma\in\mathcal{\widetilde{G}}_{P,\,Q}^{a,\,b}$, then by the
claim above, we get 
\[
\zeta\in\mathcal{N}(\{\sigma_{P}^{a,\,b},\,\sigma_{Q}^{a,\,b}\})\subseteq\mathcal{N}(\mathcal{R}_{[i,\,i+1]}^{a,\,b})\;,
\]
and moreover $\sigma$ is a gateway configuration of type $3$. On
the other hand, if $\sigma\in\Theta(\mathcal{\widetilde{G}}_{P,\,Q}^{a,\,b})$
for some permutation operator $\Theta$ that appears in Notation \ref{n_upe},
then by the same logic as above, we obtain that 
\[
\zeta\in\mathcal{N}(\{\Theta(\sigma_{P}^{a,\,b}),\,\Theta(\sigma_{Q}^{a,\,b})\})\subseteq\mathcal{N}(\Theta(\widetilde{\mathcal{R}}_{[i,\,i+1]}^{a,\,b}))\subseteq\mathcal{N}(\mathcal{R}_{[i,\,i+1]}^{a,\,b})\;,
\]
and that $\sigma$ is a gateway configuration of type $3$. This completes
the proof of part (1). Part (2) is direct from part (1). 
\end{proof}
Next, we establish a relation between $\mathcal{G}^{A,\,B}$ and $\mathcal{N}(\mathcal{R}_{[0,\,M]}^{A,\,B})$
for proper partitions $(A,\,B)$ of $S$.
\begin{lem}
\label{l_gateR}For a proper partition $(A,\,B)$ of $S$, the two
sets $\mathcal{G}^{A,\,B}$ and $\mathcal{N}(\mathcal{R}_{[0,\,M]}^{A,\,B})$
are disjoint and moreover, it holds that 
\begin{equation}
\widehat{\mathcal{N}}\big(\,\mathcal{G}^{A,\,B}\,;\,\mathcal{N}(\mathcal{R}_{[0,\,M]}^{A,\,B})\,\big)=\mathcal{G}^{A,\,B}\;.\label{e_gateR}
\end{equation}
\end{lem}

\begin{proof}
We first claim that, for any $a\in A$, $b\in B$, and $P,\,Q\in\mathfrak{S}_{M}$
with $P\prec Q$ and $|P|\in\llbracket\mathfrak{m}_{K}-1,\,M-\mathfrak{m}_{K}\rrbracket$\footnote{In fact, it holds even if $|P|\in\llbracket0,\,M-1\rrbracket$.},
\begin{equation}
\mathcal{\widetilde{G}}_{P,\,Q}^{a,\,b}\cap\mathcal{N}(\mathcal{R}_{[0,\,M]}^{A,\,B})=\emptyset\;.\label{e_gateR2}
\end{equation}
Suppose the contrary that we can take a configuration $\sigma\in\mathcal{\widetilde{G}}_{P,\,Q}^{a,\,b}\cap\mathcal{N}(\mathcal{R}_{[0,\,M]}^{A,\,B})$.
Then, since $\sigma\in\mathcal{\widetilde{G}}_{P,\,Q}^{a,\,b}$ and
since $H(\sigma)<\Gamma$ as $\sigma\in\mathcal{N}(\mathcal{R}_{[0,\,M]}^{A,\,B})$,
the configuration $\sigma$ must be a gateway configuration of type
$1$ by Proposition \ref{p_gateE}. Since $\sigma\in\mathcal{N}(\mathcal{R}_{[0,\,M]}^{A,\,B})$,
there exists a ($\Gamma-1)$-path connecting $\sigma$ and $\mathcal{R}_{[0,\,M]}^{A,\,B}$.
However, it is clear that (cf. Figure \ref{fig7.1}) any configuration
$\zeta$ such that $\zeta\sim\sigma$ has energy at least $\Gamma$.
This yields a contradiction. By the same argument, we can show that
$\Theta(\mathcal{\widetilde{G}}_{P,\,Q}^{a,\,b})$ is also disjoint
with $\mathcal{N}(\mathcal{R}_{[0,\,M]}^{A,\,B})$ where $\Theta$
is one of the permutation operators introduced in Notation \ref{n_upe},
and hence it holds that $\mathcal{G}_{P,\,Q}^{a,\,b}$ is disjoint
with $\mathcal{N}(\mathcal{R}_{[0,\,M]}^{A,\,B})$. Hence, the two
sets $\mathcal{G}^{A,\,B}$ and $\mathcal{N}(\mathcal{R}_{[0,\,M]}^{A,\,B})$
are disjoint.

Next, we turn to \eqref{e_gateR}. Since $\mathcal{G}^{A,\,B}\subseteq\widehat{\mathcal{N}}(\mathcal{G}^{A,\,B}\,;\,\mathcal{N}(\mathcal{R}_{[0,\,M]}^{A,\,B}))$
easily follows from \eqref{e_gateR2}, it suffices to show that 
\[
\widehat{\mathcal{N}}(\mathcal{G}^{A,\,B}\,;\,\mathcal{N}(\mathcal{R}_{[0,\,M]}^{A,\,B}))\subseteq\mathcal{G}^{A,\,B}\;.
\]
Suppose the contrary that we can take $\sigma\in\widehat{\mathcal{N}}(\mathcal{G}^{A,\,B}\,;\,\mathcal{N}(\mathcal{R}_{[0,\,M]}^{A,\,B}))$
which does not belong to $\mathcal{G}^{A,\,B}$. Let $(\omega_{t})_{t=0}^{T}$
be a $\Gamma$-path in $\mathcal{X}\setminus\mathcal{N}(\mathcal{R}_{[0,\,M]}^{A,\,B})$
connecting $\mathcal{G}^{A,\,B}$ and $\sigma$. Since we have assumed
that $\sigma\notin\mathcal{G}^{A,\,B}$, we can take 
\[
t_{0}=\min\,\{t:\omega_{t}\notin\mathcal{G}^{A,\,B}\}\;.
\]
Since $\omega_{t_{0}-1}\in\mathcal{G}^{A,\,B}$, $\omega_{t_{0}}\notin\mathcal{G}^{A,\,B}$,
and $\omega_{t_{0}-1}\sim\omega_{t_{0}}$, by Lemma \ref{l_gate},
we have $\omega_{t_{0}-1}\in\mathcal{N}(\mathcal{R}_{[\mathfrak{m}_{K}-1,\,M-\mathfrak{m}_{K}+1]}^{A,\,B})$.
This contradicts the fact that $(\omega_{t})_{t=0}^{T}$ is a path
in $\mathcal{X}\setminus\mathcal{N}(\mathcal{R}_{[0,\,M]}^{A,\,B})$. 
\end{proof}

\section{\label{sec8}Energy Barrier between Ground States}

The main objective of the current section is to analyze the energy
barrier and optimal paths between ground states. In this section,
we fix a proper partition $(A,\,B)$ of $S$. The main result of the
current section is the following result regarding the energy barrier
between the ground states.
\begin{prop}
\label{p_Elb}The following statements hold.
\begin{enumerate}
\item For $\mathbf{s},\,\mathbf{s}'\in\mathcal{S}$, we have that $\Phi(\mathbf{s},\,\mathbf{s}')\ge\Gamma$.
\item Let $(\omega_{t})_{t=0}^{T}$ be a path in $\mathcal{X}\setminus\mathcal{G}^{A,\,B}$
connecting $\mathcal{S}(A)$ and $\mathcal{S}(B)$. Then, there exists
$t\in\llbracket0,\,T\rrbracket$ such that $H(\omega_{t})\ge\Gamma+1$.
\end{enumerate}
\end{prop}

Part (1) of the previous proposition gives an opposite bound of Proposition
\ref{p_Eub} and hence completes the proof of the characterization
of the energy barrier. \emph{Moreover, in part (2), it is verified
that any optimal path connecting $\mathcal{S}(A)$ and $\mathcal{S}(B)$
must visit a gateway configuration between them. }Before proceeding
further, we officially conclude the proof of Theorem \ref{t_energy barrier}
by assuming Proposition \ref{p_Elb}.
\begin{proof}[Proof of Theorem \ref{t_energy barrier}]
 The conclusion of the theorem holds by Proposition \ref{p_Eub}
and part (1) of Proposition \ref{p_Elb}. 
\end{proof}
We provide the proof of Proposition \ref{p_Elb} in Sections \ref{sec8.1}
and \ref{sec8.2}. Then, in Section \ref{sec8.3}, we prove the large
deviation-type results, namely Theorem \ref{t_LDT results}, based
on the analysis of energy landscape that we carried out so far.

\subsection{\label{sec8.1}Preliminary analysis on energy landscape}

The purpose of this subsection is to provide a lemma (cf. Lemma \ref{l_Hlb2}
below) regarding the communication height between two far away configurations,
which will be the crucial tool in the proof of Proposition \ref{p_Elb}. 

Before proceeding to this result, we first introduce a lower bound
on the Hamiltonian $H$ which will be used frequently in the remaining
computations of the current section. For $\sigma\in\mathcal{X}$
and $a\in S$, denote by $\mathcal{D}_{a}(\sigma)\subseteq\mathbb{T}_{K}\times\mathbb{T}_{L}$
the collection of monochromatic pillars in $\sigma$ of spin $a$:
\begin{align*}
\mathcal{D}_{a}(\sigma) & =\{(k,\,\ell)\in\mathbb{T}_{K}\times\mathbb{T}_{L}:\sigma^{\langle k,\,\ell\rangle}(m)=a\text{ for all }m\in\mathbb{T}_{M}\}\;.
\end{align*}
Then, let $\mathcal{D}(\sigma)=\bigcup_{a\in S}\mathcal{D}_{a}(\sigma)$
and write 
\begin{equation}
d_{a}(\sigma)=|\mathcal{D}_{a}(\sigma)|\;\;\;\;\text{and}\;\;\;\;d(\sigma)=|\mathcal{D}(\sigma)|=\sum_{a\in S}d_{a}(\sigma)\;.\label{e_da}
\end{equation}
Now, we derive a lower bound on $H$. Recall the 1D and 2D Hamiltonians
from \eqref{e_Ham1D} and \eqref{e_Ham2D}, respectively.
\begin{lem}
\label{l_Hlb}For each $\sigma\in\mathcal{X}$, it holds that
\begin{equation}
H(\sigma)\ge2KL-2d(\sigma)+\sum_{m\in\mathbb{T}_{M}}H^{\mathrm{2D}}(\sigma^{(m)})\;,\label{e_H3lb}
\end{equation}
and the equality holds if and only if $H^{\mathrm{1D}}(\sigma^{\langle k,\,\ell\rangle})=2$
for all $(k,\,\ell)\in(\mathbb{T}_{K}\times\mathbb{T}_{L})\setminus\mathcal{D}(\sigma)$.
\end{lem}

\begin{proof}
Since $H^{\mathrm{1D}}(\sigma^{\langle k,\,\ell\rangle})=0$ if $(k,\,\ell)\in\mathcal{D}(\sigma)$
and $H^{\mathrm{1D}}(\sigma^{\langle k,\,\ell\rangle})\ge2$ otherwise,
we have that
\begin{equation}
\sum_{(k,\,\ell)\in\mathbb{T}_{K}\times\mathbb{T}_{L}}H^{\mathrm{1D}}(\sigma^{\langle k,\,\ell\rangle})\ge2(KL-d(\sigma))\;.\label{e_H3lb.2}
\end{equation}
Hence, we can deduce \eqref{e_H3lb} from Lemma \ref{l_decH}. The
conclusion on the equality condition is immediate from the argument
above.
\end{proof}
Now, we proceed to the main result of this subsection. For the simplicity
of notation, we write, for $a\in S$,
\begin{equation}
\mathcal{V}^{a}:=\mathcal{N}^{\mathrm{2D}}(\mathbf{s}_{a}^{\mathrm{2D}})\subseteq\mathcal{X}^{\mathrm{2D}}\;\;\;\;\text{and}\;\;\;\;\Delta^{\mathrm{2D}}:=\mathcal{X}^{\mathrm{2D}}\setminus\bigcup_{a=1}^{q}\mathcal{V}^{a}\label{e_2dnbd}
\end{equation}
so that we have the following natural decomposition of the set $\mathcal{X}^{\mathrm{2D}}$:
\begin{equation}
\mathcal{X}^{\mathrm{2D}}=\Big(\,\bigcup_{a=1}^{q}\mathcal{V}^{a}\,\Big)\cup\Delta^{\mathrm{2D}}\;.\label{e_2dnbd2}
\end{equation}
Note that the set $\Delta^{\mathrm{2D}}$ is non-empty by the definition
of $\mathcal{N}^{\mathrm{2D}}$. Recall $\mathfrak{m}_{K}\in\mathbb{N}$
from \eqref{e_mK}. The following lemma, which is the main technical
result in the analysis of the energy landscape, asserts that \textit{we
have to overcome an energy barrier of $\Gamma$ in order to change
a 2D configuration at a certain floor from a neighborhood of a ground
state to a neighborhood of another ground state.}
\begin{lem}
\label{l_Hlb2}Suppose that $a,\,b\in S$. Moreover, let $U$ and
$V$ be two disjoint subsets of $\mathbb{T}_{M}$ satisfying $|U|,\,|V|\ge\mathfrak{m}_{K}$,
and let $\sigma\in\mathcal{X}$ be a configuration satisfying 
\[
\sigma^{(m)}\in\mathcal{V}^{a}\text{ for all }m\in U\;\;\;\;\text{and}\;\;\;\;\sigma^{(m)}\in\mathcal{V}^{b}\text{ for all }m\in V\;.
\]
Suppose that another configuration $\zeta\in\mathcal{X}$ satisfies
either $\zeta^{(m)}\in\mathcal{V}^{a_{1}}$ for some $m\in U$ and
$a_{1}\ne a$ or $\zeta^{(m)}\in\mathcal{V}^{b_{1}}$ for some $m\in V$
and $b_{1}\ne b$. Finally, we assume that $\sigma$ satisfies
\begin{equation}
d(\sigma)<200\;.\label{e_dcond}
\end{equation}
Then, both of the following statements hold.
\begin{enumerate}
\item It holds that $\Phi(\sigma,\,\zeta)\ge\Gamma$.
\item For any path $(\omega_{t})_{t=0}^{T}$ in $\mathcal{X}\setminus\mathcal{G}^{a,\,b}$
connecting $\sigma$ and $\zeta$, there exists $t\in\llbracket0,\,T\rrbracket$
such that $H(\omega_{t})\ge\Gamma+1$. 
\end{enumerate}
\end{lem}

\begin{proof}
We first consider part (1). Let $(\omega_{t})_{t=0}^{T}$ be a path
connecting $\sigma$ and $\zeta$. For convenience of notation, we
define a collection $(c_{m})_{m\in U\cup V}$ such that 
\begin{equation}
c_{m}=\begin{cases}
a & \text{if }m\in U\;,\\
b & \text{if }m\in V\;.
\end{cases}\label{e_cm}
\end{equation}
Then, we define
\[
T_{0}=\min\,\{t:H^{\mathrm{2D}}(\omega_{t}^{(m)})\notin\mathcal{V}^{c_{m}}\text{ for some }m\in U\cup V\}\;,
\]
where the existence of $t\in\llbracket1,\,T-1\rrbracket$ such that
$H^{\mathrm{2D}}(\omega_{t}^{(m)})\notin\mathcal{V}^{c_{m}}$ for
some $m\in U\cup V$ is guaranteed by the conditions on $\sigma$
and $\zeta$. Now, we find $m_{0}\in U\cup V$ such that 
\begin{equation}
H^{\mathrm{2D}}(\omega_{T_{0}}^{(m_{0})})\notin\mathcal{V}^{c_{m_{0}}}\;.\label{e_t0m0}
\end{equation}
By the definitions of $\mathcal{V}^{a}$ and $T_{0}$, we have that
\begin{equation}
H^{\mathrm{2D}}(\omega_{T_{0}}^{(m_{0})})\ge\Gamma^{\mathrm{2D}}=2K+2\;.\label{e_Hlb2}
\end{equation}
If $H(\omega_{T_{0}})\ge\Gamma$, there is nothing to prove. Hence,
let us assume from now on that 
\begin{equation}
H(\omega_{T_{0}})<\Gamma\;.\label{e_Hlb2.2}
\end{equation}
Then, by Lemma \ref{l_Hlb} with $\sigma=\omega_{T_{0}}$ and by recalling
the definition \eqref{e_da} of $d(\sigma)$, we have 
\begin{equation}
2\sum_{n\in S}d_{n}(\omega_{T_{0}})+2K+2>\sum_{m\in\mathbb{T}_{M}}H^{\mathrm{2D}}(\omega_{T_{0}}^{(m)})\;.\label{e_Hlb2.3}
\end{equation}
Since we get a contradiction to \eqref{e_Hlb2} if $\mathcal{D}_{n}(\omega_{T_{0}})=\emptyset$
for all $n\in S$, there exists $n_{0}\in S$ such that $\mathcal{D}_{n_{0}}(\omega_{T_{0}})\neq\emptyset$.
Suppose first that $n_{0}\in S\setminus\{b\}$. For this case, we
claim that
\begin{equation}
H^{\mathrm{2D}}(\omega_{T_{0}}^{(m)})\ge4\text{ for all }m\in V\;.\label{e_Hlb2.4}
\end{equation}
Assume not, so that we have $\omega_{T_{0}}^{(m)}=\mathbf{s}_{n_{0}}^{\mathrm{2D}}$
for some $m\in V$. If $m=m_{0}$, this obviously cannot happen. On
the other hand, if $m\in V\setminus\{m_{0}\}$, we have $\omega_{T_{0}}^{(m)}\in\mathcal{V}^{b}$
by the definition of $T_{0}$ and thus $\omega_{T_{0}}^{(m)}$ cannot
be $\mathbf{s}_{n_{0}}^{\mathrm{2D}}$ as $b\neq n_{0}$. Therefore,
we verified \eqref{e_Hlb2.4}. Similarly, if $n_{0}\in S\setminus\{a\}$,
we obtain
\begin{equation}
H^{\mathrm{2D}}(\omega_{T_{0}}^{(m)})\ge4\text{ for all }m\in U\;.\label{e_Hlb2.5}
\end{equation}
Since either \eqref{e_Hlb2.4} or \eqref{e_Hlb2.5} must happen, and
since $|U|,\,|V|\ge\mathfrak{m}_{K}$, we get from \eqref{e_Hlb2}
and \eqref{e_Hlb2.3} that 
\begin{equation}
2\sum_{n\in S}d_{n}(\omega_{T_{0}})+2K+2>(2K+2)+4(\mathfrak{m}_{K}-1)\;,\label{e_Hlb2.6}
\end{equation}
and hence 
\begin{equation}
\sum_{n\in S}d_{n}(\omega_{T_{0}})\ge2\mathfrak{m}_{K}-1\;.\label{e_Hlb2.7}
\end{equation}
Thus, we have either 
\[
\sum_{n\in S\setminus\{a\}}d_{n}(\omega_{T_{0}})\ge\mathfrak{m}_{K}\;\;\;\;\text{or}\;\;\;\;\sum_{n\in S\setminus\{b\}}d_{n}(\omega_{T_{0}})\ge\mathfrak{m}_{K}\;.
\]
Then for $K$ satisfying the condition in Theorem \ref{t_energy barrier},
we have $\mathfrak{m}_{K}\ge200$ and thus by the condition \eqref{e_dcond},
we can take $T_{1}<T_{0}$ such that 
\begin{equation}
T_{1}=\min\,\Big\{\,t:\sum_{n\in S\setminus\{a\}}d_{n}(\omega_{t})=h_{K}^{2}\;\;\;\text{or}\;\;\;\sum_{n\in S\setminus\{b\}}d_{n}(\omega_{t})=h_{K}^{2}\,\Big\}\label{e_T1}
\end{equation}
where $h_{K}=\lfloor\sqrt{\mathfrak{m}_{K}-1}\rfloor$. Since $T_{1}<T_{0}$,
by the definition of $T_{0}$, we have 
\begin{equation}
\omega_{T_{1}}^{(m)}\in\mathcal{V}^{a}\;,\;\;\;\forall m\in U\;\;\;\text{and}\;\;\;\omega_{T_{1}}^{(m)}\in\mathcal{V}^{b}\;,\;\;\;\forall m\in V\;.\label{e_Hlb2.8}
\end{equation}
We first suppose that $\sum_{n\in S\setminus\{a\}}d_{n}(\omega_{T_{1}})=h_{K}^{2}$.
Since (cf. \eqref{e_spinnum})
\[
\Vert\omega_{T_{1}}^{(m)}\Vert_{n}\ge d_{n}(\omega_{T_{1}})\;\;\;\;\text{for all }m\in\mathbb{T}_{M}\;,\;n\in S\;,
\]
we can assert from \eqref{e_Hlb2.8} and \textbf{(L2)}, \textbf{(L3)}
of Proposition \ref{p_2lowE} that 
\begin{equation}
H^{\mathrm{2D}}(\omega_{T_{1}}^{(m)})\ge4\Big(\,\sum_{n\in S\setminus\{a\}}d_{n}(\omega_{T_{1}})\,\Big)^{1/2}=4h_{K}\text{ for all }m\in U\;.\label{e_Hlb2.9}
\end{equation}
Therefore, by Lemma \ref{l_Hlb} with $\sigma=\omega_{T_{1}}$, the
definition of $T_{1}$, and \eqref{e_Hlb2.9}, we get 
\[
H(\omega_{T_{1}})\ge2KL-4h_{K}^{2}+4h_{K}|U|\ge2KL-4h_{K}^{2}+4h_{K}\mathfrak{m}_{K}>2KL+2K+2=\Gamma\;,
\]
where the last inequality holds for $K\ge32$. Of course, we get the
same conclusion for the case of $\sum_{n\in S\setminus\{b\}}d_{n}(\omega_{T_{1}})=h_{K}^{2}$
by an identical argument. Therefore, we can conclude that $H(\omega_{T_{1}})>\Gamma$,
and thus part (1) is verified.\medskip{}

Now, we turn to part (2). We now assume that, for some $\sigma$ and
$\zeta$ satisfying the assumptions of the lemma, there exists a path
$(\omega_{t})_{t=0}^{T}$ in $\mathcal{X}\setminus\mathcal{G}^{a,\,b}$
connecting $\sigma$ and $\zeta$ with 
\begin{equation}
H(\omega_{t})\le\Gamma\;\;\;\;\;\text{for all }t\in\llbracket0,\,T\rrbracket\;.\label{e_Hlb2.10}
\end{equation}
Without loss of generality, we can assume that the triple $(\sigma,\,\zeta,\,(\omega_{t})_{t=0}^{T})$
that we selected has the smallest path length $T$ among all such
triples.

Recall $T_{0}$ from the proof of the first part. If $\mathcal{D}_{n}(\omega_{T_{0}})\neq\emptyset$
for some $n\in S$, we can repeat the same argument with part (1)
to deduce $H(\omega_{T_{1}})>\Gamma$, where $T_{1}$ is defined in
\eqref{e_T1}. This contradicts \eqref{e_Hlb2.10}.

Next, we consider the case when\textbf{ $\mathcal{D}_{n}(\omega_{T_{0}})=\emptyset$}
for all $n\in S$. The contradiction for this case is more involved
than that of the corresponding case of part (1). By Lemma \ref{l_Hlb},
we have that 
\begin{equation}
2K+2\ge\sum_{m\in\mathbb{T}_{M}}H^{\mathrm{2D}}(\omega_{T_{0}}^{(m)})\;.\label{e_Hlb2.11}
\end{equation}
Recall $m_{0}$ from \eqref{e_t0m0}. Since $H^{\mathrm{2D}}(\omega_{T_{0}}^{(m_{0})})=2K+2$
by \eqref{e_Hlb2}, we not only have 
\begin{equation}
H^{\mathrm{2D}}(\omega_{T_{0}}^{(m)})=0\text{\;\;\;\;for all }m\in\mathbb{T}_{M}\setminus\{m_{0}\}\;,\label{e_H3lb2.12}
\end{equation}
but also the equality in \eqref{e_Hlb2.11} holds, i.e., 
\begin{equation}
\sum_{m\in\mathbb{T}_{M}}H^{\mathrm{2D}}(\omega_{T_{0}}^{(m)})=2K+2\;.\label{e_H3lb2.13}
\end{equation}
Hence, by the last part of Lemma \ref{l_Hlb}, we must have 
\begin{equation}
H^{\mathrm{1D}}(\omega_{T_{0}}^{\langle k,\,\ell\rangle})=2\;\;\;\;\text{for all }(k,\,\ell)\in\mathbb{T}_{K}\times\mathbb{T}_{L}\;.\label{e_H3lb2.14}
\end{equation}
From these observations, we can deduce the following facts:
\begin{itemize}
\item By \eqref{e_H3lb2.13}, \eqref{e_H3lb2.14}, and Lemma \ref{l_Hlb},
we have $H(\omega_{T_{0}})=\Gamma$.
\item By \eqref{e_H3lb2.12} and \eqref{e_H3lb2.14}, we have $\omega_{T_{0}}^{(m)}\in\{\mathbf{s}_{a}^{\mathrm{2D}},\,\mathbf{s}_{b}^{\mathrm{2D}}\}$
for all $m\in\mathbb{T}_{M}\setminus\{m_{0}\}$.
\end{itemize}
Moreover, the spins must be aligned so that \eqref{e_H3lb2.14} holds.
Without loss of generality, we assume that $m_{0}\in U$, since the
case $m_{0}\in V$ can be handled in an identical manner. Starting
from $\omega_{T_{0}}$, suppose that we flip a spin at $m$-th floor,
$m\neq m_{0}$, without decreasing the 2D energy of the $m_{0}$-th
floor. Then, since each non-$m_{0}$-th floor is monochromatic and
\eqref{e_H3lb2.14} holds, the 3D energy of $\sigma$ increases by
at least four and we obtain a contradiction to the fact that $(\omega_{t})_{t=0}^{T}$
is a $\Gamma$-path. Thus, we must decrease the 2D energy of the $m_{0}$-th
floor before modifying the other floors. Define 
\[
T_{2}=\min\,\{t>T_{0}:H^{\mathrm{2D}}(\omega_{t}^{(m_{0})})<2K+2\}\;.
\]
Then, by Proposition \ref{p_2lowE}, it suffices to consider the following
two cases:
\begin{itemize}
\item \textbf{(Case 1: $\omega_{T_{2}}^{(m_{0})}\in\mathcal{V}^{n}$ for
some $n\in S$) }Since $\omega_{T_{0}}^{(m_{0})}\in\mathcal{X}^{\mathrm{2D}}$
is the first escape from the valley $\mathcal{V}^{a}$, it holds from
the minimality of $T_{2}$ that $\omega_{T_{2}}^{(m_{0})}\notin\mathcal{V}^{n}$
for $n\in S\setminus\{a\}$ (the 2D path must visit a number of regular
configurations first; see part (1) of Proposition \ref{p_typ2prop}).
On the other hand, if $\omega_{T_{2}}^{(m_{0})}\in\mathcal{V}^{a}$,
then we obtain a contradiction from the minimality of the length of
$(\omega_{t})_{t=0}^{T}$, as we have a shorter path from $\omega_{T_{2}}$
to $\zeta$ where $\omega_{T_{2}}$ clearly satisfies the conditions
imposed to $\sigma$.
\item \textbf{(Case 2: $\omega_{T_{2}}^{(m_{0})}$ is a 2D regular configuration)
}Since we have assumed that $m_{0}\in U$, we have $\omega_{T_{2}}^{(m_{0})}\in\mathcal{R}_{2}^{a,\,b'}$
for some $b'\in S\setminus\{a\}$ (by the minimality of $T_{2}$ and
part (1) of Proposition \ref{p_typ2prop}). Now, we claim that $b'=b$.
To this end, let us suppose that $b'\ne b$. Then as $\omega_{T_{2}}^{(m)}\in\{\mathbf{s}_{a}^{\mathrm{2D}},\,\mathbf{s}_{b}^{\mathrm{2D}}\}$
for $m\ne m_{0}$, we have $H^{\mathrm{1D}}(\omega_{T_{2}}^{\langle k,\,\ell\rangle})\ge3$
for $(k,\,\ell)\in\mathbb{T}_{K}\times\mathbb{T}_{L}$ satisfying
$\omega_{T_{2}}^{(m_{0})}(k,\,\ell)=b'$. Because there are exactly
$2K$ such $(k,\,\ell)$, by Lemma \ref{l_Hlb}, we have
\begin{align*}
H(\omega_{T_{2}}) & =\sum_{(k,\,\ell)\in\mathbb{T}_{K}\times\mathbb{T}_{L}}H^{\mathrm{1D}}(\omega_{T_{2}}^{\langle k,\,\ell\rangle})+\sum_{m\in\mathbb{T}_{M}}H^{\mathrm{2D}}(\omega_{T_{2}}^{(m)})\\
 & \ge3\times2K+2\times(KL-2K)+2K>\Gamma\;,
\end{align*}
where at the first inequality we used the fact that $H^{\mathrm{2D}}(\omega_{T_{2}}^{(m_{0})})=2K$.
This contradicts the fact that $(\omega_{t})_{t=0}^{T}$ is a $\Gamma$-path.
Therefore, we must have $b'=b$, which implies along with \eqref{e_H3lb2.14}
that $\omega_{T_{2}}\in\mathcal{G}^{a,\,b}$. Hence, we get a contradiction
as we assumed that $(\omega_{t})_{t=0}^{T}$ is a path in $\mathcal{X}\setminus\mathcal{G}^{a,\,b}$.
\end{itemize}
Since we get a contradiction for both cases, we completed the proof
of part (2).
\end{proof}
\begin{rem}
\label{r_Kcond}We remark that \eqref{e_T1} is exactly the place
from which the lower bound $2829$ of $K$ in Theorem \ref{t_energy barrier}
originates.
\end{rem}

The following is a direct consequence of the previous lemma which
will be used later.
\begin{cor}
\label{c_H3lb}Suppose that $P,\,Q\in\mathfrak{S}_{M}$ and $|P|\in\llbracket\mathfrak{m}_{K},\,M-\mathfrak{m}_{K}\rrbracket$.
Then for $a,\,b\in S$, we have $\Phi(\sigma_{P}^{a,\,b},\,\sigma_{Q}^{a,\,b})=\Gamma$.
In particular, we have $\Phi(\sigma_{P}^{a,\,b},\,\mathbf{s}_{a})=\Gamma$.
\end{cor}

\begin{proof}
We can apply Lemma \ref{l_Hlb2} with $\sigma=\sigma_{P}^{a,\,b}$
and $\zeta=\sigma_{Q}^{a,\,b}$ to get 
\begin{equation}
\Phi(\sigma_{P}^{a,\,b},\,\sigma_{Q}^{a,\,b})\ge\Gamma\;.\label{e_PQub}
\end{equation}
On the other hand, by taking a canonical path connecting $\mathbf{s}_{a}$
and $\sigma_{P}^{a,\,b}$, we get $\Phi(\mathbf{s}_{a},\,\sigma_{P}^{a,\,b})\le\Gamma$.
Similarly, we get $\Phi(\mathbf{s}_{a},\,\sigma_{Q}^{a,\,b})\le\Gamma$.
Hence, we obtain 
\begin{equation}
\Phi(\sigma_{P}^{a,\,b},\,\sigma_{Q}^{a,\,b})\le\max\,\{\Phi(\mathbf{s}_{a},\,\sigma_{P}^{a,\,b}),\,\Phi(\mathbf{s}_{a},\,\sigma_{Q}^{a,\,b})\}\le\Gamma\;.\label{e_PQlb}
\end{equation}
Combining \eqref{e_PQub} and \eqref{e_PQlb} proves $\Phi(\sigma_{P}^{a,\,b},\,\sigma_{Q}^{a,\,b})=\Gamma$.
By inserting $Q=\emptyset$, we get $\Phi(\sigma_{P}^{a,\,b},\,\mathbf{s}_{a})=\Gamma$.
\end{proof}

\subsection{\label{sec8.2}Proof of Proposition \ref{p_Elb}}

Recall \eqref{e_spinnum}. Note that
\begin{equation}
\|\sigma\|_{a}=\sum_{m\in\mathbb{T}_{M}}\Vert\sigma^{(m)}\Vert_{a}\;.\label{e_decnorm}
\end{equation}
We are now ready to prove Proposition \ref{p_Elb}. We first prove
this proposition when $q=2$. Then, the general case can be verified
from this result via a projection-type argument. 
\begin{proof}[Proof of Proposition \ref{p_Elb}: $q=2$]
 Since $q=2$, we only have two spins $1$ and $2$ and hence we
let $\mathbf{s}=\mathbf{s}_{1}$ and $\mathbf{s}'=\mathbf{s}_{2}$.
We fix an arbitrary path $(\omega_{t})_{t=0}^{T}$ connecting $\mathbf{s}$
and $\mathbf{s}'$, and take $\sigma\in(\omega_{t})_{t=0}^{T}$ such
that 
\begin{equation}
\|\sigma\|_{1}=\lfloor KLM/2\rfloor+1\;.\label{e_p1}
\end{equation}
Since there is nothing to prove if $H(\sigma)\ge\Gamma+1$, we assume
that 
\begin{equation}
H(\sigma)\le\Gamma\;.\label{e_p2}
\end{equation}
\textit{Then, we claim that there exists $t\in\llbracket0,\,T\rrbracket$
such that $H(\omega_{t})=\Gamma$. Moreover, we claim that if $(\omega_{t})_{t=0}^{T}$
is a path in $\mathcal{X}\setminus\mathcal{G}^{1,\,2}$, there exists
$t\in\llbracket0,\,T\rrbracket$ such that $H(\omega_{t})=\Gamma+1$.
}It is clear that a verification of these claims immediately proves
the case of $q=2$.

We recall the decomposition \eqref{e_2dnbd2} of $\mathcal{X}^{\mathrm{2D}}$
and write 
\begin{align*}
P_{n} & =P_{n}(\sigma)=\{m\in\mathbb{T}_{M}:\sigma^{(m)}\in\mathcal{V}^{n}\}\;\;\;\;;\;n\in\{1,\,2\}\;,\\
R & =R(\sigma)=\{m\in\mathbb{T}_{M}:\sigma^{(m)}\in\Delta^{\mathrm{2D}}\}\;,
\end{align*}
so that $\mathbb{T}_{M}$ can be decomposed into $\mathbb{T}_{M}=P_{1}\cup P_{2}\cup R$.
Write $p_{1}=|P_{1}|$, $p_{2}=|P_{2}|$, and $r=|R|$ so that the
previous decomposition of $\mathbb{T}_{M}$ implies 
\begin{equation}
p_{1}+p_{2}+r=M\;.\label{e_p3}
\end{equation}
We also write $d_{1}=d_{1}(\sigma)$, $d_{2}=d_{2}(\sigma)$, and
$d=d(\sigma)$ so that $d=d_{1}+d_{2}$. The following facts are crucially
used:
\begin{itemize}
\item By Lemma \ref{l_Hlb} and \eqref{e_p2}, it holds that 
\begin{equation}
d_{1}+d_{2}+K+1\ge\frac{1}{2}\sum_{m\in\mathbb{T}_{M}}H^{\mathrm{2D}}(\sigma^{(m)})\;.\label{e_p4}
\end{equation}
\item We have
\begin{equation}
\begin{cases}
H^{\mathrm{2D}}(\sigma^{(m)})\ge4\,\|\sigma^{(m)}\|_{2}^{1/2}\ge4d_{2}^{1/2} & \text{if }m\in P_{1}\;,\\
H^{\mathrm{2D}}(\sigma^{(m)})\ge4\,\|\sigma^{(m)}\|_{1}^{1/2}\ge4d_{1}^{1/2} & \text{if }m\in P_{2}\;,\\
H^{\mathrm{2D}}(\sigma^{(m)})\ge2K & \text{if }m\in R\;,
\end{cases}\label{e_p5}
\end{equation}
where the first two bounds follow from \textbf{(L2)} and \textbf{(L3)}
of Proposition \ref{p_2lowE}, while the last one follows from \textbf{(L1)}
of Proposition \ref{p_2lowE}.
\item By inserting \eqref{e_p5} to \eqref{e_p4}, we get
\begin{equation}
d_{1}+d_{2}+K+1\ge2p_{1}d_{2}^{1/2}+2p_{2}d_{1}^{1/2}+Kr\;.\label{e_p6}
\end{equation}
\end{itemize}
We consider four cases separately based on the conditions on $p_{1}$,
$p_{2}$, and $r$. Recall that we assumed $K\ge2829$; several arguments
below require $K$ to be large enough, and they indeed hold for $K$
in this range.\medskip{}

\noindent \textbf{(Case 1:} $p_{1},\,p_{2}\ge1$\textbf{)} Since both
$P_{1}$ and $P_{2}$ are non-empty, the first two bounds in \eqref{e_p5}
activate and thus
\begin{equation}
d_{1},\,d_{2}\le\frac{(2K+1)^{2}}{16}\;.\label{e_p7}
\end{equation}
We note that, since the function $f(x)=x-2ax^{1/2}$ is convex on
$[0,\,\infty)$ for $a>0$, by \eqref{e_p7} we have 
\begin{equation}
\begin{aligned} & d_{1}-2p_{2}d_{1}^{1/2}\le\max\,\Big\{\,0,\,\frac{(2K+1)^{2}}{16}-\frac{2K+1}{2}p_{2}\,\Big\}\;,\\
 & d_{2}-2p_{1}d_{2}^{1/2}\le\max\,\Big\{\,0,\,\frac{(2K+1)^{2}}{16}-\frac{2K+1}{2}p_{1}\,\Big\}\;.
\end{aligned}
\label{e_p8}
\end{equation}
Inserting \eqref{e_p8} to \eqref{e_p6}, we get 
\begin{equation}
Kr\le K+1+\sum_{i=1}^{2}\,\max\,\Big\{\,0,\,\frac{(2K+1)^{2}}{16}-\frac{2K+1}{2}p_{i}\,\Big\}\;.\label{e_p9}
\end{equation}
We now consider three sub-cases:
\begin{itemize}
\item $p_{1},\,p_{2}\le(2K+1)/8$: For this case, we can rewrite \eqref{e_p9}
as 
\[
Kr\le K+1+\frac{(2K+1)^{2}}{8}-\frac{2K+1}{2}(p_{1}+p_{2})<K+1+\frac{(2K+1)^{2}}{8}-K(p_{1}+p_{2})\;.
\]
Inserting \eqref{e_p3} yields a contradiction since $K\le M$.
\item $p_{1}\le(2K+1)/8<p_{2}$ or $p_{2}\le(2K+1)/8<p_{1}$: By symmetry,
it suffices to consider the former case, for which we can rewrite
\eqref{e_p9} as 
\[
Kr\le K+1+\frac{(2K+1)^{2}}{16}-\frac{2K+1}{2}p_{1}<2K+\frac{(2K+1)^{2}}{16}-Kp_{1}\;.
\]
Thus, we get 
\[
p_{1}+r\le2+\frac{(2K+1)^{2}}{16K}\;,
\]
and thus by the second bound in \eqref{e_p5},
\[
\|\sigma\|_{2}\ge p_{2}\Big(\,KL-\frac{(2K+1)^{2}}{16}\,\Big)\ge\Big(\,M-2-\frac{(2K+1)^{2}}{16K}\,\Big)\Big(\,KL-\frac{(2K+1)^{2}}{16}\,\Big)\;.
\]
We get a contradiction to \eqref{e_p1} since the right-hand side
is greater than $\lfloor KLM/2\rfloor+1$.
\item $p_{1},\,p_{2}>(2K+1)/8$: By \eqref{e_p9}, we can notice that $r=0$
or $1$. By \eqref{e_p1}, the first bound in \eqref{e_p5}, and \eqref{e_decnorm},
we get
\[
\Big\lfloor\frac{KLM}{2}\Big\rfloor+1=\|\sigma\|_{1}\ge\sum_{m\in P_{1}}\Vert\sigma^{(m)}\Vert_{1}\ge p_{1}\Big(\,KL-\frac{(2K+1)^{2}}{16}\,\Big)\;,
\]
and thus, 
\begin{equation}
p_{2}\ge M-1-p_{1}\ge M-1-\frac{\lfloor KLM/2\rfloor+1}{KL-(2K+1)^{2}/16}\ge\frac{2K+1}{7}\;.\label{e_p10}
\end{equation}
Similarly, we get 
\begin{equation}
p_{1}\ge\frac{2K+1}{7}\;.\label{e_p11}
\end{equation}
Now, by \eqref{e_p7}, \eqref{e_p10} and \eqref{e_p11}, it holds
that 
\[
p_{2}\ge\frac{4}{7}d_{1}^{1/2}\;\;\;\;\text{and}\;\;\;\;p_{1}\ge\frac{4}{7}d_{2}^{1/2}\;.
\]
Inserting this along with \eqref{e_p10} and \eqref{e_p11} to the
right-hand side of \eqref{e_p6}, we get 
\[
2p_{1}d_{2}^{1/2}+2p_{2}d_{1}^{1/2}+Kr\ge\Big(\,d_{2}+\frac{1}{4}p_{1}d_{2}^{1/2}\,\Big)+\Big(\,d_{1}+\frac{1}{4}p_{2}d_{1}^{1/2}\,\Big)\ge d+\frac{2K+1}{28}d^{1/2}\;,
\]
where the last inequality follows from the inequality $x^{1/2}+y^{1/2}\ge(x+y)^{1/2}$.
Applying this to \eqref{e_p6}, we conclude that 
\[
d\le\Big(\,\frac{28(K+1)}{2K+1}\,\Big)^{2}<200\;.
\]
This proves the condition \eqref{e_dcond} for $\sigma.$ Moreover,
since $p_{1},\,p_{2}>(2K+1)/8\ge\mathfrak{m}_{K}$, we can now apply
part (1) of Lemma \ref{l_Hlb2} to deduce $\Phi(\sigma,\,\mathbf{s}')\ge\Gamma$,
and this proves the first part of the claim. Moreover, if $(\omega_{t})_{t=0}^{T}$
is a path in $\mathcal{X}\setminus\mathcal{G}^{1,\,2}$, then the
sub-path from $\sigma$ to $\omega_{T}=\mathbf{s}'$ is also in $\mathcal{X}\setminus\mathcal{G}^{1,\,2}$,
and thus part (2) of Lemma \ref{l_Hlb2} verifies the second assertion
of the claim as well.\medskip{}
\end{itemize}
\textbf{(Case 2: }$p_{1}\ge1$, $p_{2}=0$, $r\ge1$ or $p_{1}\ge1$,
$p_{2}=0$, $r\ge1$\textbf{) }By symmetry, it suffices to consider
the former case. Similarly as in \textbf{(Case 1)}, we can apply the
first bound in \eqref{e_p5} to deduce 
\begin{equation}
d_{2}\le\frac{(2K+1)^{2}}{16}\;.\label{e_p12}
\end{equation}
Again by the first bound in \eqref{e_p5}, we have 
\[
\Vert\sigma^{(m)}\Vert_{1}\ge KL-\frac{(2K+1)^{2}}{16}\;\;\;\;\text{for all }m\in P_{1}\;,
\]
and thus we get
\[
\sum_{m\in R}\Vert\sigma^{(m)}\Vert_{1}=\|\sigma\|_{1}-\sum_{m\in P_{1}}\Vert\sigma^{(m)}\Vert_{1}\le\frac{KLM}{2}+1-p_{1}\Big(\,KL-\frac{(2K+1)^{2}}{16}\,\Big)\;.
\]
Therefore, there exists $m_{0}\in R$ such that
\begin{align*}
\Vert\sigma^{(m_{0})}\Vert_{1} & \le\frac{1}{r}\Big[\,\frac{KLM}{2}+1-p_{1}\Big(\,KL-\frac{(2K+1)^{2}}{16}\,\Big)\,\Big]\\
 & =KL-\frac{(2K+1)^{2}}{16}+\frac{1}{r}\Big[\,-\frac{KLM}{2}+\frac{(2K+1)^{2}M}{16}+1\,\Big]\\
 & \le KL-\frac{(2K+1)^{2}}{16}-\frac{1}{r}\Big[\,\frac{KLM}{4}-\frac{K^{2}M}{20}\,\Big]\;,
\end{align*}
where at the second line we used $p_{1}=M-r$. Thus, we have 
\begin{equation}
d_{1}\le\Vert\sigma^{(m_{0})}\Vert_{1}\le KL-\frac{(2K+1)^{2}}{16}-\frac{1}{r}\Big[\,\frac{KLM}{4}-\frac{K^{2}M}{20}\,\Big]\;.\label{e_p13}
\end{equation}
Inserting this to \eqref{e_p6}, we get
\[
2p_{1}d_{2}^{1/2}+Kr\le d_{2}+\Big[\,KL-\frac{(2K+1)^{2}}{16}-\frac{1}{r}\Big(\,\frac{KLM}{4}-\frac{K^{2}M}{20}\,\Big)\,\Big]+K+1\;.
\]
Reorganizing and applying a similar estimate as in \eqref{e_p8},
we get 
\begin{align}
Kr+\frac{1}{r}\Big[\,\frac{KLM}{4}- & \frac{K^{2}M}{20}\,\Big]\le KL-\frac{(2K+1)^{2}}{16}+K+1+(d_{2}-2p_{1}d_{2}^{1/2})\nonumber \\
 & \le KL-\frac{(2K+1)^{2}}{16}+K+1+\max\,\Big\{\,0,\,\frac{(2K+1)^{2}}{16}-\frac{2K+1}{2}p_{1}\,\Big\}\nonumber \\
 & =KL+K+1-\min\,\Big\{\,\frac{(2K+1)^{2}}{16},\,\frac{2K+1}{2}p_{1}\,\Big\}\;.\label{e_p14}
\end{align}
Now, we analyze two sub-cases separately.
\begin{itemize}
\item $p_{1}\le(2K+1)/8$: Then, we can rewrite \eqref{e_p14} as 
\[
Kr+\frac{1}{r}\Big[\,\frac{KLM}{4}-\frac{K^{2}M}{20}\,\Big]\le KL+K+1-\frac{2K+1}{2}p_{1}\le KL+K\;.
\]
Multiplying $r/K$ in both sides, we reorganize the previous inequality
as 
\begin{equation}
\Big(\,r-\frac{L+1}{2}\,\Big)^{2}\le\frac{(L+1)^{2}}{4}-\frac{LM}{4}+\frac{KM}{20}\le\frac{L^{2}+10L+5}{20}\;.\label{e_p15}
\end{equation}
Since $p_{1}\le(2K+1)/8$, we have 
\begin{equation}
r\ge M-\frac{2K+1}{8}\ge\frac{3}{4}L-1\;.\label{e_p16}
\end{equation}
Inserting \eqref{e_p16} to \eqref{e_p15} yields a contradiction
for $L\ge K\ge2829$.
\item $p_{1}>(2K+1)/8$: For this case, \eqref{e_p14} becomes
\[
Kr+\frac{1}{r}\Big[\,\frac{KLM}{4}-\frac{K^{2}M}{20}\,\Big]\le KL+K+1-\frac{(2K+1)^{2}}{16}\le KL-\frac{K^{2}}{4}+K\;.
\]
Multiplying both sides by $r/K$ and reorganizing, we get 
\[
\Big(\,r-\frac{L}{2}+\frac{K}{8}-\frac{1}{2}\,\Big)^{2}\le\frac{1}{64}(4L-K+4)^{2}-\frac{LM}{4}+\frac{KM}{20}\;.
\]
Since the right-hand side is negative for $K\ge9$, we get a contradiction.\medskip{}
\end{itemize}
\textbf{(Case 3: }$p_{1}\ge1$, $p_{2}=0$, $r=0$ or $p_{1}=0$,
$p_{2}\ge1$, $r=0$\textbf{) }As before, we only consider the former
case. In this case, indeed $P_{1}=\mathbb{T}_{M}$. Thus, by the first
bound in \eqref{e_p5} we have
\[
\|\sigma\|_{1}=\sum_{m\in\mathbb{T}_{M}}\|\sigma^{(m)}\|_{1}\ge M\Big(\,KL-\frac{(2K+1)^{2}}{16}\,\Big)>\frac{KLM}{2}+2\;.
\]
This contradicts \eqref{e_p1}.\medskip{}

\noindent \textbf{(Case 4: $p_{1}=p_{2}=0$) }For this case, we have
$\sigma^{(m)}\in\Delta^{\mathrm{2D}}$ for all $m\in\mathbb{T}_{M}$.
Hence, $H^{\mathrm{2D}}(\sigma^{(m)})\ge2K$ for all $m\in\mathbb{T}_{M}$
by \textbf{(L1)} of Proposition \ref{p_2lowE}, and thus by \eqref{e_p4}
we get 
\begin{equation}
d_{1}+d_{2}\ge K(M-1)-1\;.\label{e_p17}
\end{equation}
Since $d_{1}+d_{2}=d\le KL$, we get $M=L+1$ or $L$.

If $M=L+1$, we must have $d_{1}+d_{2}=KL$ or $KL-1$. If this is
$KL$, then all floors should have the same configuration, which is
impossible since $\|\sigma\|_{1}=\lfloor KLM/2\rfloor+1$ cannot be
a multiple of $M$. If this is $KL-1$, then the equality in \eqref{e_p17}
must hold and thus we have $H^{\mathrm{2D}}(\sigma^{(m)})=2K$ for
all $m\in\mathbb{T}_{M}$. Hence, by \textbf{(L1)} of Proposition
\ref{p_2lowE}, $\|\sigma^{(m)}\|_{1}$, $m\in\mathbb{T}_{M}$, is
a multiple of $K$, and thus $\|\sigma\|_{1}=\sum_{m\in\mathbb{T}_{M}}\|\sigma^{(m)}\|_{1}$
is also a multiple of $K$. This is impossible since $\|\sigma\|_{1}=\lfloor KLM/2\rfloor+1$
is not a multiple of $K$.

It remains to consider the case of $M=L$. For this case, \eqref{e_p17}
becomes
\begin{equation}
d_{1}+d_{2}\ge K(L-1)-1\;.\label{e_p18}
\end{equation}
Define
\begin{equation}
\mathcal{E}(\sigma)=(\mathbb{T}_{K}\times\mathbb{T}_{L})\setminus(\mathcal{D}_{1}(\sigma)\cup\mathcal{D}_{2}(\sigma))\label{e_p19}
\end{equation}
so that we have $|\mathcal{E}(\sigma)|\le K+1$ by \eqref{e_p18}.
We now have three sub-cases. We note that $H^{\mathrm{2D}}(\sigma^{(m)})$
is an even integer for each $m\in\mathbb{T}_{M}$, as $q=2$.
\begin{itemize}
\item First, we assume that $H^{\mathrm{2D}}(\sigma^{(m)})=2K$ for all
$m\in\mathbb{T}_{M}$. Then, as in the previous discussion on the
case of $M=L+1$, we get a contradiction since $\|\sigma\|_{1}$ must
be a multiple of $K$ for this case. 
\item Next, we assume that $H^{\mathrm{2D}}(\sigma^{(m)})\ge2K+2$ for all
$m\in\mathbb{T}_{M}$. Then, by \eqref{e_p4}, 
\begin{equation}
d=d_{1}+d_{2}\ge(K+1)(L-1)\;.\label{e_p20}
\end{equation}
If $d=KL$, then as in the case of $M=L+1$, we get a contradiction
since $\|\sigma\|_{1}=\lfloor KLM/2\rfloor+1$ cannot be a multiple
of $M$. Hence, we have $d\le KL-1$, and combining this with \eqref{e_p20}
implies that we must have $K=L$ (and thus $K=L=M$), and moreover
\[
d=KL-1\;\;\;\;\text{and}\;\;\;\;H^{\mathrm{2D}}(\sigma^{(m)})=2K+2\;\;\;\;\text{for all }m\in\mathbb{T}_{M}\;.
\]
Hence, we have $|\mathcal{E}(\sigma)|=KL-d=1$. Write $\mathcal{E}(\sigma)=\{(k_{0},\,\ell_{0})\}$.
By Lemma \ref{l_H2lb}, we can deduce that the configuration $\sigma^{(m)}$
has at least $L-1\ge3$ monochromatic bridges, and thus we have at
least one monochromatic bridge of the form $\mathbb{T}_{K}\times\{\ell\}$
or $\{k\}\times\mathbb{T}_{L}$ that does not touch $\mathcal{E}(\sigma)$,
so that it is a subset of either $\mathcal{D}_{1}(\sigma)$ or $\mathcal{D}_{2}(\sigma)$.
Suppose first that this bridge is $\mathbb{T}_{K}\times\{\ell\}$
for some $\ell\in\mathbb{T}_{L}$. Then, the slab $\mathbb{T}_{K}\times\{\ell\}\times\mathbb{T}_{M}$
is monochromatic. Therefore, by replacing the role of the second and
third coordinates, which is possible since $K=L=M$, the proof is
reduced to one of \textbf{(Case 1)}, \textbf{(Case 2)}, and \textbf{(Case
3)} as there is a monochromatic floor so that either $p_{1}$ or $p_{2}$
is positive. This completes the proof. Similarly, if the monochromatic
bridge is $\{k\}\times\mathbb{T}_{L}$, then we replace the role of
the first and third coordinates to complete the proof.
\item Now, we lastly assume that $H^{\mathrm{2D}}(\sigma^{(i_{0})})=2K$
for some $i_{0}\in\mathbb{T}_{M}$ and $H^{\mathrm{2D}}(\sigma^{(j_{0})})\ge2K+2$
for some $j_{0}\in\mathbb{T}_{M}$. By \eqref{e_p4}, we get 
\begin{equation}
d\ge KL-K\;,\label{e_p21}
\end{equation}
and hence, we have $|\mathcal{E}(\sigma)|\le K$ (cf. \eqref{e_p19}).
Now, we consider two sub-sub-cases separately.
\begin{itemize}
\item $|\mathcal{E}(\sigma)|\le K-1$: First, suppose that $K<L$. By \textbf{(L1)}
of Proposition \ref{p_2lowE}, we have $\sigma^{(i_{0})}\in\mathcal{R}_{v}^{1,\,2}$
for some $v\in\llbracket2,\,L-2\rrbracket$. Since $|\mathcal{E}(\sigma)|\le K-1\le L-1$,
there exists $\ell_{1}\in\mathbb{T}_{L}$ such that 
\[
(\mathbb{T}_{K}\times\{\ell_{1}\})\cap\mathcal{E}(\sigma)=\emptyset\;.
\]
We further have $\mathbb{T}_{K}\times\{\ell_{1}\}\subseteq\mathcal{D}_{1}(\sigma)$
or $\mathbb{T}_{K}\times\{\ell_{1}\}\subseteq\mathcal{D}_{2}(\sigma)$
since $\sigma^{(i_{0})}\in\mathcal{R}_{v}^{1,\,2}$. This implies
that all sites in the slab $\mathbb{T}_{K}\times\{\ell_{1}\}\times\mathbb{T}_{M}$
have the same spin $n$ under $\sigma$. Since $L=M$, we can replace
the role of the second and third coordinates to reduce the proof to
one of \textbf{(Case 1)}, \textbf{(Case 2)} and \textbf{(Case 3)}.
This completes the proof. Next, if $K=L$, then since there further
exists $k_{1}\in\mathbb{T}_{K}$ such that $(\{k_{1}\}\times\mathbb{T}_{L})\cap\mathcal{E}(\sigma)=\emptyset$,
we can use the same argument as above to handle this case as well.
\item $|\mathcal{E}(\sigma)|=K$: The equality in \eqref{e_p21} must hold,
and thus we get $H^{\mathrm{2D}}(\sigma^{(j_{0})})=2K+2$ and $H^{\mathrm{2D}}(\sigma^{(m)})=2K$
for all $m\in\mathbb{T}_{M}\setminus\{j_{0}\}$. We first suppose
that $K<L$. By \textbf{(L1)} of Proposition \ref{p_2lowE}, we get
$\sigma^{(m)}=\xi_{\ell_{m},\,v_{m}}^{1,\,2}$ for some $\ell_{m}\in\mathbb{T}_{L}$
and $v_{m}\in\llbracket2,\,L-2\rrbracket$ for all $m\in\mathbb{T}_{M}\setminus\{j_{0}\}$.
Then, since $L$ is strictly bigger than $K=|\mathcal{E}(\sigma)|$,
we can always find a row in $\mathbb{T}_{K}\times\mathbb{T}_{L}$
which is either a subset of $\mathcal{D}_{1}(\sigma)$ or $\mathcal{D}_{2}(\sigma)$.
Thus, by changing the role of the second and third coordinates, which
is possible since $L=M$, we find a monochromatic floor and the proof
is reduced to one of \textbf{(Case 1)}, \textbf{(Case 2)} and \textbf{(Case
3)}. Next, we handle the case $K=L$, so that for all $m\in\mathbb{T}_{M}\setminus\{j_{0}\}$,
$\sigma^{(m)}=\xi_{\ell_{m},\,v_{m}}^{1,\,2}$ or $\Theta(\xi_{\ell_{m},\,v_{m}}^{1,\,2})$
(cf. Definition \ref{d_canreg2}) for some $\ell_{m}\in\mathbb{T}_{L}$
and $v_{m}\in\llbracket2,\,L-2\rrbracket$. First of all, assume that
all of them are of the same direction. Without loss of generality,
assume that $\sigma^{(m)}=\xi_{\ell_{m},\,v_{m}}^{1,\,2}$ for all
$m\in\mathbb{T}_{M}\setminus\{j_{0}\}$. If $\sigma^{(m_{1})}\neq\sigma^{(m_{2})}$
for some $m_{1},\,m_{2}\in\mathbb{T}_{M}\setminus\{j_{0}\}$, then
$\mathcal{E}(\sigma)$ must be exactly the line where they differ
and hence we can write $\mathcal{E}(\sigma)=\mathbb{T}_{K}\times\{\ell_{0}\}$
for some $\ell_{0}\in\mathbb{T}_{L}$. Then, by taking any $\ell\in\mathbb{T}_{L}\setminus\{\ell_{0}\}$,
we notice that $\mathbb{T}_{K}\times\{\ell\}$ is not only monochromatic
in $\sigma^{(m)}$ with $m\in\mathbb{T}_{M}\setminus\{j_{0}\}$, but
also a subset of either $\mathcal{D}_{1}(\sigma)$ or $\mathcal{D}_{2}(\sigma)$;
hence, $\mathbb{T}_{K}\times\{\ell\}\times\mathbb{T}_{M}$ is a monochromatic
slab. By replacing the role of the second and third coordinates, which
is possible since $L=M$, we find a monochromatic floor and the proof
is reduced to one of \textbf{(Case 1)}, \textbf{(Case 2)} and \textbf{(Case
3)}. On the contrary, suppose that $\sigma^{(m_{1})}=\sigma^{(m_{2})}$
for all $m_{1},\,m_{2}\in\mathbb{T}_{M}\setminus\{j_{0}\}$. If there
exists a row or column which is disjoint with $\mathcal{E}(\sigma)$,
then we can argue as above. If not, then we can easily deduce that
for the $j_{0}$-th floor,
\[
H^{\mathrm{2D}}(\sigma^{(j_{0})})\ge4(K-4)>2K+2\;,
\]
which contradicts the assumption that $H^{\mathrm{2D}}(\sigma^{(j_{0})})=2K+2$.
Finally, we consider the case when $\sigma^{(m)}=\xi_{\ell,\,v}^{1,\,2}$
and $\sigma^{(m')}=\Theta(\xi_{\ell',\,v'}^{1,\,2})$ for some $m,\,m'\in\mathbb{T}_{M}\setminus\{j_{0}\}$
simultaneously. In this case, we have
\[
d_{1}\le(K-v)(K-v')\;\;\;\;\text{and}\;\;\;\;d_{2}\le vv'\;.
\]
Thus, we get a contradiction since 
\begin{align*}
|\mathcal{E}(\sigma)| & \ge K^{2}-(K-v)(K-v')-vv'\\
 & =\frac{1}{2}K^{2}-\frac{1}{2}(K-2v)(K-2v')\ge\frac{1}{2}K^{2}-\frac{1}{2}(K-4)^{2}=4K-8>K\;,
\end{align*}
where the second inequality holds since $v,\,v'\in\llbracket2,\,K-2\rrbracket$.
\end{itemize}
\end{itemize}
\end{proof}
Now, we consider the general case of Proposition \ref{p_Elb}.
\begin{proof}[Proof of Proposition \ref{p_Elb}: general case]
 We fix a proper partition $(A,\,B)$ of $S$ and then fix $a\in A$
and $b\in B$. Let $(\omega_{t})_{t=0}^{T}$ be a path connecting
$\mathbf{s}_{a}$ and $\mathbf{s}_{b}.$ For each $\sigma\in\mathcal{X}$,
we denote by $\widetilde{\sigma}$ the configuration obtained from
$\sigma$ by changing all spins in $A$ to $1$ and spins in $B$
to $2$. Thus, $\widetilde{\sigma}$ becomes an Ising configuration,
i.e. a spin configuration for $q=2$. Note that 
\begin{align}
H(\widetilde{\sigma}) & =\sum_{\{x,\,y\}\subseteq\Lambda:\,x\sim y}\mathbf{1}\{\widetilde{\sigma}(x)\neq\widetilde{\sigma}(y)\}\nonumber \\
 & =\sum_{\{x,\,y\}\subseteq\Lambda:\,x\sim y}\mathbf{1}\{\sigma(x)\in A,\;\sigma(y)\in B\;\text{or}\;\sigma(x)\in B,\,\sigma(y)\in A\}\label{e_p22}\\
 & \le\sum_{\{x,\,y\}\subseteq\Lambda:\,x\sim y}\mathbf{1}\{\sigma(x)\neq\sigma(y)\}=H(\sigma)\;.\nonumber 
\end{align}
Now, we consider the induced pseudo-path $(\widetilde{\omega}_{t})_{t=0}^{T}$
of $(\omega_{t})_{t=0}^{T}$ (cf. Notation \ref{n_pseudo}). Thus,
by the proof above for $q=2$, there exists $t_{1}\in\llbracket0,\,T\rrbracket$
such that $H(\widetilde{\omega}_{t_{1}})\ge\Gamma$. Thus, we get
from \eqref{e_p22} that 
\[
\Gamma\le H(\widetilde{\omega}_{t_{1}})\le H(\omega_{t_{1}})\;,
\]
and we complete the proof for part (1).

For part (2), suppose that $(\omega_{t})_{t=0}^{T}$ is a path such
that 
\begin{equation}
H(\omega_{t})\le\Gamma\;\;\;\text{for all }t\in\llbracket0,\,T\rrbracket\;.\label{e_p23}
\end{equation}
Then, by \eqref{e_p22}, we have $H(\widetilde{\omega}_{t})\le\Gamma$
for all $t\in\llbracket0,\,T\rrbracket$. Thus, by the proof above
for $q=2$, there exists $s\in\llbracket0,\,T\rrbracket$ such that
$\widetilde{\omega}_{s}\in\mathcal{G}^{1,\,2}$. We now claim that
$\omega_{s}\in\mathcal{G}^{A,\,B}$. It is immediate that this claim
finishes the proof.

To prove this claim, we write 
\[
\mathcal{U}_{n}=\{x\in\Lambda:\widetilde{\omega}_{s}(x)=n\}\;\;\;\;;\;n=1,\,2\;.
\]
Then, we have 
\begin{equation}
\omega_{s}(x)\in A\;\text{for }x\in\mathcal{U}_{1}\;\;\;\;\text{and}\;\;\;\;\omega_{s}(x)\in B\;\text{for }x\in\mathcal{U}_{2}\;.\label{e_p24}
\end{equation}
Now, we assume that 
\begin{equation}
\omega_{s}(x)\neq\omega_{s}(y)\text{ for some }x,\,y\in\mathcal{U}_{1}\text{ or }x,\,y\in\mathcal{U}_{2}\text{ with }x\sim y\;.\label{e_p25}
\end{equation}
We now express the energy $H(\omega_{s})$ as 
\begin{align}
H(\omega_{s}) & =\Big[\,\sum_{\{x,\,y\}\subseteq\mathcal{U}_{1}\text{ or }\{x,\,y\}\subseteq\mathcal{U}_{2}}+\sum_{x\in\mathcal{U}_{1},\,y\in\mathcal{U}_{2}}\,\Big]\,\mathbf{1}\{\omega_{s}(x)\neq\omega_{s}(y)\}\;,\label{e_p26}
\end{align}
where the summation is carried over $x,\,y$ satisfying $x\sim y$.
Note that the second summation is equal to $H(\widetilde{\omega}_{s})$
by \eqref{e_p24}. On the other hand, we can readily deduce from Figure
\ref{fig7.1} that the first summation of \eqref{e_p26} is at least
$4$ if $\widetilde{\omega}_{s}$ is a gateway configuration of type
$1$, and at least $2$ if $\widetilde{\omega}_{s}$ is a gateway
configuration of type $2$ or $3$ (cf. Notation \ref{n_gate}).
Thus, by Proposition \ref{p_gateE}, we can conclude that the right-hand
side of \eqref{e_p26} is at least $\Gamma+2$; i.e., we get $H(\omega_{s})\ge\Gamma+2$.
This contradicts \eqref{e_p23} and hence, we cannot have \eqref{e_p25}.
This finally implies that there exist $a_{0}\in A$ and $b_{0}\in B$
such that 
\[
\omega_{s}(x)=\begin{cases}
a_{0} & \text{if }x\in\mathcal{U}_{1}\;,\\
b_{0} & \text{if }x\in\mathcal{U}_{2}\;,
\end{cases}
\]
and thus we have $\omega_{s}\in\mathcal{G}^{a_{0},\,b_{0}}\subseteq\mathcal{G}^{A,\,B}$
as claimed.
\end{proof}

\subsection{\label{sec8.3}Proof of Theorem \ref{t_LDT results}}

Theorem \ref{t_LDT results} is now a consequence of our analysis
on the energy landscape and the general theory developed in \cite{NZ,NZB}.
\begin{proof}[Proof of Theorem \ref{t_LDT results}]
 We have two results on the energy barrier; Theorem \ref{t_energy barrier}
and Proposition \ref{p_depth}. The theory developed in \cite{NZB}
implies that these two are sufficient to conclude Theorem \ref{t_LDT results}\footnote{We remark that the second convergence of \eqref{e_nz3} is not a consequence
of an analysis of the energy barrier, but of the first convergence
of \eqref{e_nz3} and the symmetry of the model. This argument is
also given in \cite[Section 3]{NZ} for $d=2$, and an identical one
works for $d=3$.}. This implication has been rigorously verified in \cite{NZ} for
the case of $d=2$, and this argument extends to the case of $d=3$
without a modification. Hence, we do not repeat the argument here,
and refer the readers to \cite[Section 3]{NZ} for a detailed proof.
\end{proof}

\section{\label{sec9}Typical Configurations and Optimal Paths}

In the previous sections, we proved large deviation-type results regarding
the metastable behavior by analyzing the energy barrier in terms of
canonical and gateway configurations. In order to get precise quantitative
results such as Theorems \ref{t_EK} and \ref{t_MC} or to get a characterization
of optimal paths, we need a more refined analysis of the energy landscape
based on the typical configurations which will be introduced and analyzed
in the current section.

We fix a proper partition $(A,\,B)$ of $S$ throughout the section.

\subsection{\label{sec9.1}Typical configurations }

Let us start by defining the typical configurations. We consistently
refer to Figure \ref{fig9.1} for an illustration of our construction. 

\begin{figure}
\includegraphics[width=14.5cm]{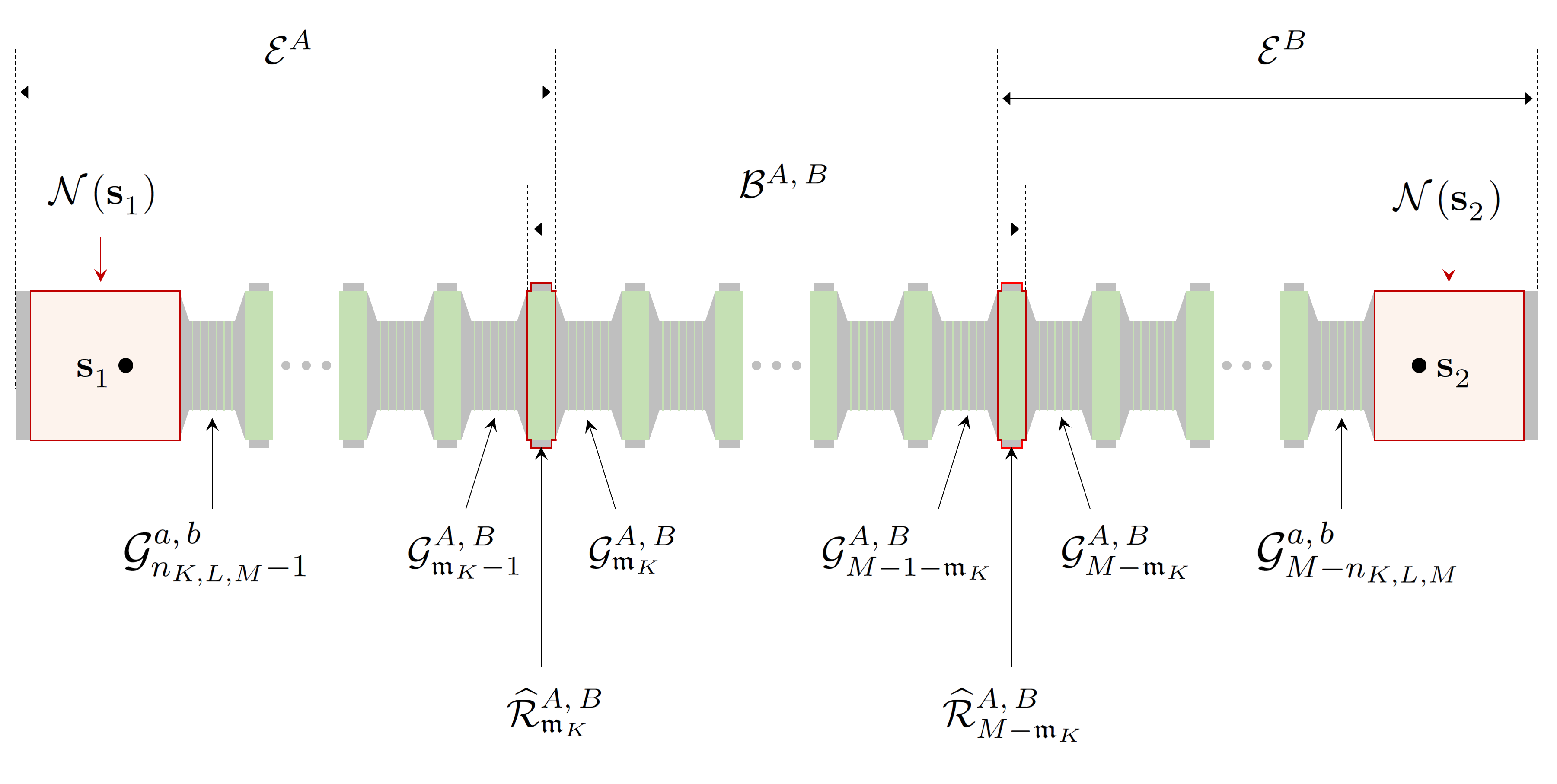}

\caption{\label{fig9.1}\textbf{3D typical configurations for the Ising model.
}Suppose that $q=2$, $A=\{1\}$, and $B=\{2\}$ (one can compare
this figure with Figure \ref{fig7.2}). The given figure provides
an illustration of the complete structure of the set $\widehat{\mathcal{N}}(\mathcal{S})$.
This characterization is verified in Proposition \ref{p_typ}. We
take bulk ones only from $\mathfrak{m}_{K}$ to $M-\mathfrak{m}_{K}$
instead of $n_{K,\,L,\,M}$ (cf. \eqref{e_nKLM}) to $M-n_{K,\,L,\,M}$
since we do not know the exact value of $n_{K,\,L,\,M}$. Because
of this, the structure of the edge typical configurations is a little
bit more complicated than in the 2D case.}
\end{figure}

For $a,\,b\in S$ and $i\in\llbracket0,\,M\rrbracket$, we define
\begin{equation}
\widehat{\mathcal{R}}_{i}^{a,\,b}=\widehat{\mathcal{N}}(\mathcal{R}_{i}^{a,\,b}\,;\,\mathcal{G}^{a,\,b})\;.\label{e_Rhat}
\end{equation}
We also define 
\[
\widehat{\mathcal{R}}_{i}^{A,\,B}=\widehat{\mathcal{N}}(\mathcal{R}_{i}^{A,\,B}\,;\,\mathcal{G}^{A,\,B})\;\;\;\;;\;i\in\llbracket0,\,M\rrbracket\;.
\]

\begin{rem}
\label{r_uniR}For $i\in\llbracket\mathfrak{m}_{K},\,M-\mathfrak{m}_{K}\rrbracket$,
we have that 
\[
\widehat{\mathcal{R}}_{i}^{A,\,B}=\bigcup_{a\in A,\,b\in B}\widehat{\mathcal{R}}_{i}^{a,\,b}\;.
\]
To check this, it suffices to check $\widehat{\mathcal{N}}(\mathcal{R}_{i}^{a,\,b}\,;\,\mathcal{G}^{a,\,b})=\widehat{\mathcal{N}}(\mathcal{R}_{i}^{a,\,b}\,;\,\mathcal{G}^{A,\,B})$
provided that $a\in A$ and $b\in B$. This follows from Lemma \ref{l_Hlb2}
since $\mathcal{R}_{i}^{a,\,b}$ cannot be connected to a configuration
in $\mathcal{G}^{A,\,B}\setminus\mathcal{G}^{a,\,b}$ via a $\Gamma$-path
in $\mathcal{X}\setminus\mathcal{G}^{a,\,b}$ by part (2) of Lemma
\ref{l_Hlb2}.
\end{rem}

\begin{rem}
\label{r_disjR}By Lemma \ref{l_Hlb2}, two sets $\widehat{\mathcal{R}}_{i}^{A,\,B}$
and $\widehat{\mathcal{R}}_{j}^{A,\,B}$ for different $i,\,j$ are
disjoint if either $i\in\llbracket\mathfrak{m}_{K},\,M-\mathfrak{m}_{K}\rrbracket$
or $j\in\llbracket\mathfrak{m}_{K},\,M-\mathfrak{m}_{K}\rrbracket$.
Moreover by Proposition \ref{p_Elb}, they are disjoint if $i\in\llbracket0,\,\mathfrak{m}_{K}-1\rrbracket$
and $j\in\llbracket M-\mathfrak{m}_{K}+1,\,M\rrbracket$ or vice versa.
On the other hand, they might be the same set if $i,\,j\in\llbracket0,\,\mathfrak{m}_{K}-1\rrbracket$
or $i,\,j\in\llbracket M-\mathfrak{m}_{K}+1,\,M\rrbracket$. In particular,
we have 
\[
\widehat{\mathcal{R}}_{0}^{A,\,B}=\widehat{\mathcal{R}}_{1}^{A,\,B}=\cdots=\widehat{\mathcal{R}}_{n_{K,L,M}-1}^{A,\,B}\;,
\]
where $n_{K,\,L,\,M}$ is defined in \eqref{e_nKLM}. The same result
holds for $\widehat{\mathcal{R}}_{i}^{a,\,b}$ instead of $\widehat{\mathcal{R}}_{i}^{A,\,B}$.
\end{rem}

Now, we define the typical configurations. We recall Notation \ref{n_union}.
\begin{defn}[Typical configurations]
\label{d_typ} For a proper partition $(A,\,B)$ of $S$, we define
the typical configurations as follows.
\begin{itemize}
\item \textbf{Bulk typical configurations:} We define, for $a,\,b\in S$,
\[
\mathcal{B}^{a,\,b}=\mathcal{G}_{[\mathfrak{m}_{K},\,M-\mathfrak{m}_{K}-1]}^{a,\,b}\cup\widehat{\mathcal{R}}_{[\mathfrak{m}_{K},\,M-\mathfrak{m}_{K}]}^{a,\,b}\;,
\]
and then define
\begin{equation}
\mathcal{B}^{A,\,B}=\bigcup_{a\in A}\bigcup_{b\in B}\mathcal{B}^{a,\,b}=\mathcal{G}_{[\mathfrak{m}_{K},\,M-\mathfrak{m}_{K}-1]}^{A,\,B}\cup\widehat{\mathcal{R}}_{[\mathfrak{m}_{K},\,M-\mathfrak{m}_{K}]}^{A,\,B}\;,\label{e_BAB}
\end{equation}
where the second identity holds because of Remark \ref{r_uniR}. A
configuration belonging to $\mathcal{B}^{A,\,B}$ is called a \textit{bulk
typical configuration} between $\mathcal{S}(A)$ and $\mathcal{S}(B)$. 
\item \textbf{Edge typical configurations:} We define
\begin{equation}
\mathcal{E}^{A}=\mathcal{G}_{\mathfrak{m}_{K}-1}^{A,\,B}\cup\widehat{\mathcal{R}}_{[0,\,\mathfrak{m}_{K}]}^{A,\,B}\;\;\;\;\text{and}\;\;\;\;\mathcal{E}^{B}=\mathcal{G}_{M-\mathfrak{m}_{K}}^{A,\,B}\cup\widehat{\mathcal{R}}_{[M-\mathfrak{m}_{K},\,M]}^{A,\,B}\;.\label{e_EAB}
\end{equation}
Finally, we define $\mathcal{E}^{A,\,B}=\mathcal{E}^{A}\cup\mathcal{E}^{B}$.
A configuration belonging to $\mathcal{E}^{A,\,B}$ is called an \textit{edge
typical configuration} between $\mathcal{S}(A)$ and $\mathcal{S}(B)$.
\end{itemize}
Later in Proposition \ref{p_typ}, we shall show that $\mathcal{B}^{A,\,B}\cup\mathcal{E}^{A,\,B}=\widehat{\mathcal{N}}(\mathcal{S})$
and hence all relevant configurations in the analysis of metastable
behavior between $\mathcal{S}(A)$ and $\mathcal{S}(B)$ belong to
either $\mathcal{B}^{A,\,B}$ or $\mathcal{E}^{A,\,B}$.
\end{defn}

\begin{rem}
\label{r_typ}Since $\mathcal{R}_{0}^{A,\,B}=\mathcal{S}(A)$ and
$\mathcal{R}_{M}^{A,\,B}=\mathcal{S}(B)$ (cf. \eqref{e_reg3}), we
can readily observe that $\mathcal{S}(A)\subseteq\mathcal{E}^{A}$
and $\mathcal{S}(B)\subseteq\mathcal{E}^{B}$.
\end{rem}

\subsection{\label{sec9.2}Properties of typical configurations}

In this subsection, we analyze some properties of the edge and bulk
typical configurations. In fact, we have to take $K$ large enough
(i.e., $K\ge2829$) in order to get the structural properties of edge
and bulk typical configurations given in the current section. 

The first property asserts that $\mathcal{E}^{A}$ and $\mathcal{E}^{B}$
are disjoint. 
\begin{prop}
\label{p_disjE}The two sets $\mathcal{E}^{A}$ and $\mathcal{E}^{B}$
are disjoint.
\end{prop}

\begin{proof}
By part (2) of Lemma \ref{l_Hlb2} (cf. Remark \ref{r_disjR}), the
set $\widehat{\mathcal{R}}_{\mathfrak{m}_{K}}^{A,\,B}$ is disjoint
with $\mathcal{E}^{B}$; similarly, the set $\widehat{\mathcal{R}}_{M-\mathfrak{m}_{K}}^{A,\,B}$
is disjoint with $\mathcal{E}^{A}$. It is direct from the definition
that $\mathcal{G}_{\mathfrak{m}_{K}-1}^{A,\,B}$ and $\mathcal{G}_{M-\mathfrak{m}_{K}}^{A,\,B}$
are disjoint. By definition, $\mathcal{G}_{\mathfrak{m}_{K}-1}^{A,\,B},\,\mathcal{G}_{M-\mathfrak{m}_{K}}^{A,\,B}\subseteq\mathcal{G}^{A,\,B}$
are mutually disjoint with $\widehat{\mathcal{R}}_{[0,\,\mathfrak{m}_{K}]}^{A,\,B}$
and $\widehat{\mathcal{R}}_{[M-\mathfrak{m}_{K},\,M]}^{A,\,B}$. Hence,
it suffices to prove that $\widehat{\mathcal{R}}_{[0,\,\mathfrak{m}_{K}-1]}^{A,\,B}$
and $\widehat{\mathcal{R}}_{[M-\mathfrak{m}_{K}+1,\,M]}^{A,\,B}$
are disjoint. Otherwise, we can take a configuration $\sigma$ such
that 
\[
\sigma\in\widehat{\mathcal{R}}_{[0,\,\mathfrak{m}_{K}-1]}^{A,\,B}\cap\widehat{\mathcal{R}}_{[M-\mathfrak{m}_{K}+1,\,M]}^{A,\,B}\;.
\]
Since $\sigma\in\widehat{\mathcal{R}}_{[0,\,\mathfrak{m}_{K}-1]}^{A,\,B}$,
there exists a $\Gamma$-path in $\mathcal{X}\setminus\mathcal{G}^{A,\,B}$
(which is indeed a part of a canonical path) connecting $\sigma$
and $\mathcal{S}(A)$. Similarly, there exists a $\Gamma$-path in
$\mathcal{X}\setminus\mathcal{G}^{A,\,B}$ connecting $\sigma$ and
$\mathcal{S}(B)$. By concatenating them, we can find a $\Gamma$-path
$(\omega_{t})_{t=0}^{T}$ in $\mathcal{X}\setminus\mathcal{G}^{A,\,B}$
connecting $\mathcal{S}(A)$ and $\mathcal{S}(B)$. This contradicts
part (2) of Proposition \ref{p_Elb}.
\end{proof}
Now, we analyze the crucial features of the typical configurations.
Note that this is a 3D version of Proposition \ref{p_typ2prop}. 
\begin{prop}
\label{p_typ}For a proper partition $(A,\,B)$ of $S$, the following
properties hold for the typical configurations.
\begin{enumerate}
\item It holds that $\mathcal{E}^{A}\cap\mathcal{B}^{A,\,B}=\widehat{\mathcal{R}}_{\mathfrak{m}_{K}}^{A,\,B}$
and $\mathcal{E}^{B}\cap\mathcal{B}^{A,\,B}=\widehat{\mathcal{R}}_{M-\mathfrak{m}_{K}}^{A,\,B}$. 
\item We have $\mathcal{E}^{A,\,B}\cup\mathcal{B}^{A,\,B}=\widehat{\mathcal{N}}(\mathcal{S})$. 
\end{enumerate}
\end{prop}

\begin{proof}
(1) It suffices to prove the first identity, as the second one follows
similarly. One can observe that the set $\mathcal{G}_{\mathfrak{m}_{K}-1}^{A,\,B}\subseteq\mathcal{G}^{A,\,B}$
is disjoint with $\mathcal{B}^{A,\,B}$ from \eqref{e_BAB}, and the
set $\mathcal{G}_{[\mathfrak{m}_{K},\,M-\mathfrak{m}_{K}-1]}^{A,\,B}\subseteq\mathcal{G}^{A,\,B}$
is disjoint with $\mathcal{E}^{A}$ in view of the expression \eqref{e_EAB}.
Therefore, by \eqref{e_BAB} and \eqref{e_EAB}, we get 
\begin{equation}
\mathcal{E}^{A}\cap\mathcal{B}^{A,\,B}=\widehat{\mathcal{R}}_{[0,\,\mathfrak{m}_{K}]}^{A,\,B}\cap\widehat{\mathcal{R}}_{[\mathfrak{m}_{K},\,M-\mathfrak{m}_{K}]}^{A,\,B}\;.\label{e_typ2}
\end{equation}
By Lemma \ref{l_Hlb2}, the two sets $\mathcal{R}_{[0,\,\mathfrak{m}_{K}-1]}^{A,\,B}$
and $\mathcal{R}_{[\mathfrak{m}_{K},\,M-\mathfrak{m}_{K}]}^{A,\,B}$
cannot be connected by a $\Gamma$-path in $\mathcal{X}\setminus\mathcal{G}^{A,\,B}$,
and hence $\mathcal{\widehat{\mathcal{R}}}_{[0,\,\mathfrak{m}_{K}-1]}^{A,\,B}$
and $\mathcal{\widehat{\mathcal{R}}}_{[\mathfrak{m}_{K},\,M-\mathfrak{m}_{K}]}^{A,\,B}$
are disjoint. Therefore, we have 
\begin{equation}
\widehat{\mathcal{R}}_{[0,\,\mathfrak{m}_{K}]}^{A,\,B}\cap\widehat{\mathcal{R}}_{[\mathfrak{m}_{K},\,M-\mathfrak{m}_{K}]}^{A,\,B}=\widehat{\mathcal{R}}_{\mathfrak{m}_{K}}^{A,\,B}\cap\widehat{\mathcal{R}}_{[\mathfrak{m}_{K},\,M-\mathfrak{m}_{K}]}^{A,\,B}=\widehat{\mathcal{R}}_{\mathfrak{m}_{K}}^{A,\,B}\;.\label{e_typ3}
\end{equation}
The proof is completed by \eqref{e_typ2} and \eqref{e_typ3}.\medskip{}

\noindent (2) We will first prove that
\begin{equation}
\widehat{\mathcal{N}}(\mathcal{S})=\widehat{\mathcal{N}}(\mathcal{N}(\mathcal{R}_{[0,\,M]}^{A,\,B})\cup\mathcal{G}^{A,\,B})\;.\label{e_typ5}
\end{equation}
Since it is immediate that $\mathcal{S}$ is a subset of the right-hand
side, we have $\widehat{\mathcal{N}}(\mathcal{S})\subseteq\widehat{\mathcal{N}}(\mathcal{N}(\mathcal{R}_{[0,\,M]}^{A,\,B})\cup\mathcal{G}^{A,\,B})$.
On the other hand, since $\mathcal{N}(\mathcal{R}_{[0,\,M]}^{A,\,B})\cup\mathcal{G}^{A,\,B}\subseteq\widehat{\mathcal{N}}(\mathcal{S})$
clearly holds, we also have $\widehat{\mathcal{N}}(\mathcal{N}(\mathcal{R}_{[0,\,M]}^{A,\,B})\cup\mathcal{G}^{A,\,B})\subseteq\widehat{\mathcal{N}}(\mathcal{S})$.
This proves \eqref{e_typ5}. Since the sets $\mathcal{N}(\mathcal{R}_{[0,\,M]}^{A,\,B})$
and $\mathcal{G}^{A,\,B}$ are disjoint by Lemma \ref{l_gateR}, we
can apply Lemmas \ref{l_set} and \ref{l_gateR} to deduce 
\begin{align*}
\widehat{\mathcal{N}}(\mathcal{S}) & =\widehat{\mathcal{N}}(\mathcal{N}(\mathcal{R}_{[0,\,M]}^{A,\,B})\,;\,\mathcal{G}^{A,\,B})\cup\widehat{\mathcal{N}}(\mathcal{G}^{A,\,B}\,;\,\mathcal{N}(\mathcal{R}_{[0,\,M]}^{A,\,B}))\;.\\
 & =\widehat{\mathcal{N}}(\mathcal{R}_{[0,\,M]}^{A,\,B}\,;\,\mathcal{G}^{A,\,B})\cup\mathcal{G}^{A,\,B}=\widehat{\mathcal{R}}_{[0,\,M]}^{A,\,B}\cup\mathcal{G}^{A,\,B}\;.
\end{align*}
This completes the proof since by \eqref{e_gate3}, \eqref{e_BAB}
and \eqref{e_EAB}, we have $\mathcal{E}^{A,\,B}\cup\mathcal{B}^{A,\,B}=\mathcal{G}^{A,\,B}\cup\widehat{\mathcal{R}}_{[0,\,M]}^{A,\,B}$. 
\end{proof}

\subsection{\label{sec9.3}Structure of edge typical configurations}

As in the 2D case \cite[Sections 6.4 and 6.5]{KS 2D}, we analyze
the structure of edge typical configurations.

We remark that we fixed a proper partition $(A,\,B)$ of $S$. Decompose
\begin{equation}
\mathcal{E}^{A}=\mathcal{O}^{A}\cup\mathcal{I}^{A}\label{e_EA}
\end{equation}
where 
\[
\mathcal{O}^{A}=\{\sigma\in\mathcal{E}^{A}:H(\sigma)=\Gamma\}\;\;\;\;\text{and}\;\;\;\;\mathcal{I}^{A}=\{\sigma\in\mathcal{E}^{A}:H(\sigma)<\Gamma\}\;.
\]
We now take a subset $\overline{\mathcal{I}}^{A}$ of $\mathcal{I}^{A}$
so that we can decompose $\mathcal{I}^{A}$ into the following disjoint
union: 
\[
\mathcal{I}^{A}=\bigcup_{\sigma\in\overline{\mathcal{I}}^{A}}\mathcal{N}(\sigma)\;.
\]
Consequently, we get the following decomposition of $\mathcal{E}^{A}$:
\begin{equation}
\mathcal{E}^{A}=\mathcal{O}^{A}\cup\Big(\,\bigcup_{\sigma\in\overline{\mathcal{I}}^{A}}\mathcal{N}(\sigma)\,\Big)\;.\label{e_decEA}
\end{equation}

\begin{notation}
\label{n_barsigma}For $\sigma\in\mathcal{I}^{A}$, we denote by $\overline{\sigma}\in\overline{\mathcal{I}}^{A}$
the unique configuration satisfying $\sigma\in\mathcal{N}(\overline{\sigma})$.
\end{notation}

By part (1) of Lemma \ref{l_Hlb2}, for $\sigma,\,\sigma'\in\mathcal{R}_{\mathfrak{m}_{K}}^{A,\,B}$,
the two sets $\mathcal{N}(\sigma)$ and $\mathcal{N}(\sigma')$ are
disjoint. By a similar reasoning, we know that for any\textbf{ $\sigma\in\mathcal{R}_{\mathfrak{m}_{K}}^{A,\,B}$}
and $a\in A$, the sets $\mathcal{N}(\sigma)$ and $\mathcal{N}(\mathbf{s}_{a})$
are disjoint. Thus, we can assume that 
\begin{equation}
\mathcal{R}_{\mathfrak{m}_{K}}^{A,\,B}\cup\mathcal{S}(A)\subseteq\overline{\mathcal{I}}^{A}\;.\label{e_barIA}
\end{equation}
The following construction of an auxiliary Markov chain is an analogue
of \cite[Definition 6.20]{KS 2D}.
\begin{defn}
\label{d_EAMc}For a proper partition $(A,\,B)$ of $S$, we define
a Markov chain $Z^{A}(\cdot)$ on $\overline{\mathcal{I}}^{A}\cup\mathcal{O}^{A}$. 
\begin{itemize}
\item \textbf{(Graph) }We define the graph structure $\mathscr{G}^{A}=(\mathscr{V}^{A},\,\mathscr{E}^{A})$
for $\mathscr{V}^{A}=\mathcal{O}^{A}\cup\overline{\mathcal{I}}^{A}$.
The edge set $\mathscr{E}^{A}$ is defined by declaring that $\{\sigma,\,\sigma'\}\in\mathscr{E}^{A}$
for $\sigma,\,\sigma'\in\mathscr{V}^{A}$ if
\[
\begin{cases}
\sigma,\,\sigma'\in\mathcal{O}^{A}\text{ and }\sigma\sim\sigma'\text{ or}\\
\sigma\in\mathcal{O}^{A}\;,\;\sigma'\in\overline{\mathcal{I}}^{A}\;,\;\text{ and }\sigma\sim\zeta\text{ for some }\zeta\in\mathcal{N}(\sigma')\;.
\end{cases}
\]
\item \textbf{(Markov chain)} We first define a rate $r^{A}:\mathscr{V}^{A}\times\mathscr{V}^{A}\rightarrow[0,\,\infty)$.
If $\{\sigma,\,\sigma'\}\notin\mathscr{E}^{A}$, we set $r^{A}(\sigma,\,\sigma')=0$,
and if $\{\sigma,\,\sigma'\}\in\mathscr{E}^{A}$, we set 
\begin{equation}
r^{A}(\sigma,\,\sigma')=\begin{cases}
1 & \text{if }\sigma,\,\sigma'\in\mathcal{O}^{A}\;,\\
|\{\zeta\in\mathcal{N}(\sigma):\zeta\sim\sigma'\}| & \text{if }\sigma\in\overline{\mathcal{I}}^{A}\;,\;\sigma'\in\mathcal{O}^{A}\;,\\
|\{\zeta\in\mathcal{N}(\sigma'):\zeta\sim\sigma\}| & \text{if }\sigma\in\mathcal{O}^{A}\;,\;\sigma'\in\overline{\mathcal{I}}^{A}\;.
\end{cases}\label{e_EAMc}
\end{equation}
We now let $(Z^{A}(t))_{t\ge0}$ be the continuous-time Markov chain
on $\mathscr{V}^{A}$ with rate $r^{A}(\cdot,\,\cdot)$. Note that
the uniform distribution on $\mathscr{V}^{A}$ is the invariant measure
for the chain $Z^{A}(\cdot)$, and indeed this chain is reversible
with respect to this measure.
\item \textbf{(Potential-theoretic objects) }Denote by $L^{A}$, $h_{\cdot,\,\cdot}^{A}(\cdot)$,
and $\mathrm{cap}^{A}(\cdot,\,\cdot)$ the generator, equilibrium
potential, and capacity with respect to the Markov chain $Z^{A}(\cdot)$,
respectively. 
\end{itemize}
\end{defn}

We now give three important propositions regarding the objects constructed
above. These propositions play fundamental roles in the construction
of the test function on the edge typical configurations. 

We remark from \eqref{e_barIA} that $\mathcal{S}(A),\,\mathcal{R}_{\mathfrak{m}_{K}}^{A,\,B}\subseteq\overline{\mathcal{I}}^{A}\subseteq\mathscr{V}^{A}.$
Potential-theoretic objects between these two sets are crucially used
in our discussion. We define 
\begin{equation}
\mathfrak{e}(A)=\frac{1}{|\mathscr{V}^{A}|\,\mathrm{cap}^{A}(\mathcal{S}(A),\,\mathcal{R}_{\mathfrak{m}_{K}}^{A,\,B})}\;.\label{e_edef}
\end{equation}
For $n\in\llbracket1,\,q-1\rrbracket$ (with a slight abuse of notation)
we can write 
\begin{equation}
\mathfrak{e}(n)=\frac{1}{|\mathscr{V}^{A_{n}}|\,\mathrm{cap}^{A_{n}}(\mathcal{S}(A_{n}),\,\mathcal{R}_{\mathfrak{m}_{K}}^{A_{n},\,B_{n}})}\;,\label{e_e3def}
\end{equation}
where $A_{n}=\{1,\,\dots,\,n\}$ and $B_{n}=\{n+1,\,\dots,\,q\}$.
Since $\mathfrak{e}(A)$ depends on $A$ only through $|A|$, it holds
that $\mathfrak{e}(A)=\mathfrak{e}(|A|)$. We next derive a rough
bound of $\mathfrak{e}(n)$ via the Thomson principle. We refer to
e.g., \cite{BdenH meta} for the flow structure and the Thomson principle. 
\begin{prop}
\label{p_eest}For all $n\in\llbracket1,\,q-1\rrbracket$, we have
that $\mathfrak{e}(n)\le\frac{1}{K^{1/3}}$.
\end{prop}

\begin{proof}
We recall that an anti-symmetric function $\phi:\mathscr{V}^{A}\times\mathscr{V}^{A}\to\mathbb{R}$
is called a \emph{flow} associated with the Markov chain $Z^{A}(\cdot)$,
provided that $\phi(x,\,y)\ne0$ if and only if $\{x,\,y\}\in\mathscr{E}^{A}$.
For each flow $\phi$, the associated flow norm is defined by
\[
\|\phi\|^{2}:=\sum_{\{x,\,y\}\subseteq\mathscr{V}^{A}:\,\{x,\,y\}\in\mathscr{E}^{A}}\frac{\phi(x,\,y)^{2}}{|\mathscr{V}^{A}|^{-1}\,r^{A}(x,\,y)}\;.
\]
For each $x\in\mathscr{V}^{A}$, the divergence of a flow $\phi$
at $x$ is defined by
\[
(\mathrm{div}\,\phi)(x):=\sum_{y\in\mathscr{V}^{A}}\phi(x,\,y)\;.
\]
Finally, for two disjoint non-empty subsets $\mathcal{U},\,\mathcal{V}$
of $\mathscr{V}^{A}$, a flow $\phi$ is called a \textit{unit flow}
from $\mathcal{U}$ to $\mathcal{V}$ if 
\[
\sum_{x\in\mathcal{U}}(\mathrm{div}\,\phi)(x)=-\sum_{x\in\mathcal{V}}(\mathrm{div}\,\phi)(x)=1\;\;\;\;\text{and}\;\;\;\;(\mathrm{div}\,\phi)(z)=0\text{ for all }z\notin\mathcal{U}\cup\mathcal{V}\;.
\]
Then, by the Thomson principle (cf. \cite[Theorem 7.37]{BdenH meta}),
for any \textit{unit flow} $\psi$ from $\mathcal{S}(A)$ to $\mathcal{R}_{\mathfrak{m}_{K}}^{A,\,B}$,
we have 
\begin{equation}
\mathrm{cap}^{A}(\mathcal{S}(A),\,\mathcal{R}_{\mathfrak{m}_{K}}^{A,\,B})\ge\frac{1}{\|\psi\|^{2}}\;.\label{e_Thom}
\end{equation}
We shall construct below a \textit{\emph{unit flow}} $\psi$ from
$\mathcal{S}(A)$ to $\mathcal{R}_{\mathfrak{m}_{K}}^{A,\,B}$ that
satisfies 
\begin{equation}
\Vert\psi\Vert^{2}<\frac{\mathfrak{m}_{K}\,|\mathscr{V}^{A}|}{2M}\;.\label{e_bdpsi}
\end{equation}
Then, by combining \eqref{e_Thom} and \eqref{e_bdpsi}, we have (recalling
the definition \eqref{e_mK} of $\mathfrak{m}_{K}$)
\[
\mathrm{cap}^{A}(\mathcal{S}(A),\,\mathcal{R}_{\mathfrak{m}_{K}}^{A,\,B})\ge\frac{1}{\|\psi\|^{2}}>\frac{2M}{|\mathscr{V}^{A}|\,\mathfrak{m}_{K}}\ge\frac{K^{1/3}}{|\mathscr{V}^{A}|}\;.
\]
Recalling the definition \eqref{e_edef}, this completes the proof. 

Now, it remains to construct a \textit{\emph{unit flow}} $\psi$ from
$\mathcal{S}(A)$ to $\mathcal{R}_{\mathfrak{m}_{K}}^{A,\,B}$ satisfying
bound \eqref{e_bdpsi}. To this end, let us first fix $a\in A$ and
$b\in B$. Define
\begin{equation}
i_{K,\,L,\,M}=\max\,\{m\ge1:\Phi(\mathcal{S}(A),\,\mathcal{R}_{m}^{A,\,B})<\Gamma\}\;.\label{e_iKLM}
\end{equation}
By Corollary \ref{c_H3lb}, we know that $i_{K,\,L,\,M}<\mathfrak{m}_{K}$. 

Let us start by fixing $P,\,Q\in\mathfrak{S}_{M}$ such that $P\prec Q$,
$Q\setminus P=\{m\}$, $i_{K,\,L,\,M}\le|P|<|Q|\le\mathfrak{m}_{K}$,
$a\in A$, and $b\in B$. Then, we first define a flow $\psi_{P,\,Q}$
connecting $\overline{\sigma}_{P}^{a,\,b}=\overline{\sigma_{P}^{a,\,b}}$
and $\overline{\sigma}_{Q}^{a,\,b}=\overline{\sigma_{Q}^{a,\,b}}$
(cf. Notation \ref{n_barsigma}). First, we set 
\begin{equation}
\psi_{P,\,Q}(\sigma,\,\zeta)=-\psi_{P,\,Q}(\zeta,\,\sigma)=\frac{1}{2KLM}\label{e_eest}
\end{equation}
if $\sigma,\,\zeta\in\mathcal{C}_{P,\:Q}^{a,\,b}$ satisfy, for some
$\ell\in\mathbb{T}_{L}$, $k\in\mathbb{T}_{K}$, $v\in\llbracket1,\,L-2\rrbracket$,
and $h\in\llbracket1,\,K-2\rrbracket$, 
\begin{equation}
\begin{cases}
\sigma^{(m)}=\xi_{\ell,\,v}^{a,\,b}\;\;\text{and}\;\;\zeta^{(m)}=\xi_{\ell,\,v;\,k,\,1}^{a,\,b,\,+}\;\;\;\;\text{or}\;,\\
\sigma^{(m)}=\xi_{\ell,\,v;\,k,\,h}^{a,\,b,\,+}\;\;\text{and}\;\;\zeta^{(m)}=\xi_{\ell,\,v;\,k,\,h+1}^{a,\,b,\,+}\;\;\;\;\text{or\;,}\\
\sigma^{(m)}=\xi_{\ell,\,v;\,k,\,K-1}^{a,\,b,\,+}\;\;\text{and}\;\;\zeta^{(m)}=\xi_{\ell,\,v+1}^{a,\,b}\;.
\end{cases}\label{e_eest2}
\end{equation}
Now, we claim that all configurations that appear in \eqref{e_eest2}
except for the ones corresponding to $\xi_{\ell,\,1}^{a,\,b}$ and
$\xi_{\ell,\,L-1}^{a,\,b}$ belong to $\mathscr{V}^{A}$. To check
this, observe first that if the $m$-th floor of $\sigma\in\mathcal{C}_{P,\:Q}^{a,\,b}$
is of the form $\sigma^{(m)}=\xi_{\ell,\,v\,;\,k,\,h}^{a,\,b,\,+}$,
we have $H(\sigma)=\Gamma$ and hence $\sigma\in\mathcal{O}^{A}$.
On the other hand, if the $m$-th floor of $\sigma\in\mathcal{C}_{P,\:Q}^{a,\,b}$
is of the form $\sigma^{(m)}=\xi_{\ell,\,v}^{a,\,b}$ for $v\in\llbracket2,\,L-2\rrbracket$,
we have $H(\sigma)=\Gamma-2$ and moreover $\mathcal{N}(\sigma)=\{\sigma\}$.
This implies that $\sigma\in\overline{\mathcal{I}}^{A}$. This proves
the claim. On the other hand, since if $\sigma^{(m)}=\xi_{\ell,\,1}^{a,\,b}$
then $\sigma\in\mathcal{N}(\sigma_{P}^{a,\,b})$, and if $\sigma^{(m)}=\xi_{\ell,\,L-1}^{a,\,b}$
then $\sigma\in\mathcal{N}(\sigma_{Q}^{a,\,b})$ (cf. the canonical
paths provide $(\Gamma-1)$-paths), we can replace the configurations
corresponding to $\xi_{\ell,\,1}^{a,\,b}$ and $\xi_{\ell,\,L-1}^{a,\,b}$
that appear in \eqref{e_eest2} with $\overline{\sigma}_{P}^{a,\,b}$
and $\overline{\sigma}_{Q}^{a,\,b}$, respectively, to get a flow
connecting $\overline{\sigma}_{P}^{a,\,b}$ and $\overline{\sigma}_{Q}^{a,\,b}$.
We remark that we may have $\overline{\sigma}_{P}^{a,\,b}=\overline{\sigma}_{Q}^{a,\,b}$.

We deduce from the definition of the flow norm that 
\begin{equation}
\Vert\psi_{P,\,Q}\Vert^{2}=\frac{|\mathscr{V}^{A}|}{(2KLM)^{2}}\times K^{2}L(L-2)<\frac{|\mathscr{V}^{A}|}{4M^{2}}\;,\label{e_e3est3}
\end{equation}
where $K^{2}L(L-2)$ is the number of edges that appear in \eqref{e_eest2}.
Next, we define
\[
\psi=\sum_{r=i_{K,L,M}}^{\mathfrak{m}_{K}-1}\,\sum_{P,\,Q\in\mathfrak{S}_{M}:\,|P|=r,\,P\prec Q}\psi_{P,\,Q}\;.
\]
Notice from \eqref{e_iKLM} that a configuration of the form $\overline{\sigma}_{P,\,Q}^{a,\,b}$
with $|P|=i_{K,\,L,\,M}$ is indeed an element of $\mathcal{S}(A)$.
Then, from the definition \eqref{e_eest}, we can readily check that
$\psi(x)=0$ for all $x\notin\mathcal{S}(A)\cup\mathcal{R}_{\mathfrak{m}_{K}}^{A,\,B}$
(by using the fact that the flow on each edge has a constant magnitude
$\frac{1}{2KLM}$), and moreover it holds that (cf. \eqref{e_barIA})
\begin{equation}
\sum_{x\in\mathcal{S}(A)}\sum_{y\in\mathscr{V}^{A}}\psi(x,\,y)=1\;.\label{e_e3est4}
\end{equation}
Indeed, to prove the last assertion, it suffices to observe that 
\begin{align*}
\sum_{x\in\mathcal{S}(A)}\sum_{y\in\mathscr{V}^{A}}\psi(x,\,y) & =\sum_{P,\,Q\in\mathfrak{S}_{M}:\,|P|=i_{K,L,M},\,P\prec Q}\,\sum_{\zeta\in\mathscr{V}^{A}:\,\overline{\sigma}_{P}^{a,b}\sim\zeta}\psi_{P,\,Q}(\overline{\sigma}_{P}^{a,\,b},\,\zeta)\\
 & =\frac{1}{2KLM}\times KL\times2M=1\;,
\end{align*}
where $KL$ is the number of configurations in $\mathcal{C}_{P,\,Q}^{a,\,b}$
connected to $\overline{\sigma}_{P}^{a,\,b}$, and $2M$ is the number
of possible choices of $P$ and $Q$. \textit{\emph{Consequently,
the flow $\psi$ is a unit flow from $\mathcal{S}(A)$ to $\mathcal{R}_{\mathfrak{m}_{K}}^{A,\,B}$. }}

\textit{\emph{Thus, it suffices to verify \eqref{e_bdpsi}. }}Since
the support of $\psi_{P,\,Q}$ (which is the collection of edges on
which $\psi_{P,\,Q}$ is non-zero) for different pairs $(P,\,Q)$
are disjoint, we deduce from \eqref{e_e3est3} that 
\[
\Vert\psi\Vert^{2}=\sum_{r=i_{K,L,M}}^{\mathfrak{m}_{K}-1}\,\sum_{P,\,Q\in\mathfrak{S}_{M}:\,|P|=r,\,P\prec Q}\Vert\psi_{P,\,Q}\Vert^{2}<\mathfrak{m}_{K}\times2M\times\frac{|\mathscr{V}^{A}|}{4M^{2}}=\frac{\mathfrak{m}_{K}\,|\mathscr{V}^{A}|}{2M}\;,
\]
and therefore $\psi$ satisfies \textit{\emph{\eqref{e_bdpsi}.}}
\end{proof}
For simplicity, we write (cf. \eqref{e_barIA})
\begin{equation}
\mathfrak{h}^{A}(\cdot)=h_{\mathcal{S}(A),\,\mathcal{R}_{\mathfrak{m}_{K}}^{A,B}}^{A}(\cdot)\label{e_hA}
\end{equation}
where $h^{A}$ is the equilibrium potential defined in Definition
\ref{d_EAMc}. This function is a fundamental object in the construction
of the test function in Section \ref{sec10}. 
\begin{prop}
\label{p_hAest}For $\sigma\in\widehat{\mathcal{R}}_{\mathfrak{m}_{K}}^{A,\,B}\cap\mathcal{O}^{A}\subseteq\mathscr{V}^{A}$,
we have $\mathfrak{h}^{A}(\sigma)=0$.
\end{prop}

\begin{proof}
We fix $\sigma\in\widehat{\mathcal{R}}_{\mathfrak{m}_{K}}^{A,\,B}\cap\mathcal{O}^{A}$.
It suffices to prove that any $\Gamma$-path $(\omega_{t})_{t=0}^{T}$
from $\sigma$ to $\mathcal{S}(A)$ must visit $\mathcal{N}(\mathcal{R}_{\mathfrak{m}_{K}}^{A,\,B})$.
Suppose first that the path $(\omega_{t})_{t=0}^{T}$ does not visit
$\mathcal{G}^{A,\,B}$. Since $\sigma\in\mathcal{\widehat{R}}_{\mathfrak{m}_{K}}^{A,\,B}$,
there exists a $\Gamma$-path in $\mathcal{X}\setminus\mathcal{G}^{A,\,B}$
connecting $\mathcal{R}_{\mathfrak{m}_{K}}^{A,\,B}$ and $\sigma$,
and therefore by concatenating this path with $(\omega_{t})_{t=0}^{T}$,
we get a $\Gamma$-path in $\mathcal{X}\setminus\mathcal{G}^{A,\,B}$
connecting $\mathcal{R}_{\mathfrak{m}_{K}}^{A,\,B}$ and $\mathcal{S}(A)$.
This contradicts part (2) of Lemma \ref{l_Hlb2}. Thus, the path $(\omega_{t})_{t=0}^{T}$
must visit $\mathcal{G}^{A,\,B}$ and we let 
\[
t_{0}=\min\,\{t:\omega_{t}\in\mathcal{G}^{A,\,B}\}\;.
\]
By part (2) of Lemma \ref{l_gate}, we have $\omega_{t_{0}-1}\in\mathcal{N}(\mathcal{R}_{[\mathfrak{m}_{K}-1,\,M-\mathfrak{m}_{K}+1]}^{A,\,B})$.
If $\omega_{t_{0}-1}\in\mathcal{N}(\mathcal{R}_{i}^{A,\,B})$ for
some $i\in\llbracket\mathfrak{m}_{K}-1,\,M-\mathfrak{m}_{K}+1\rrbracket\setminus\{\mathfrak{m}_{K}\}$,
then $(\omega_{t})_{t=0}^{t_{0}-1}$ induces a $\Gamma$-path from
$\mathcal{R}_{\mathfrak{m}_{K}}^{A,\,B}$ from $\mathcal{R}_{i}^{A,\,B}$
avoiding $\mathcal{G}^{A,\,B}$, which contradicts part (2) of Lemma
\ref{l_Hlb2}. Hence, we can conclude that $\omega_{t_{0}-1}\in\mathcal{N}(\mathcal{R}_{\mathfrak{m}_{K}}^{A,\,B})$,
as desired.
\end{proof}
\begin{rem}
The previous proposition implies that configurations $\sigma$ that
belong to $\widehat{\mathcal{R}}_{\mathfrak{m}_{K}}^{A,\,B}\cap\mathcal{O}^{A}$
are dead-ends attached to $\mathcal{N}(\mathcal{R}_{\mathfrak{m}_{K}}^{A,\,B})$
(cf. grey protuberances attached to green boxes in Figures \ref{fig7.2}
and \ref{fig9.1}).
\end{rem}

The next proposition highlights the fact that the auxiliary process
$Z^{A}(\cdot)$ defined in Definition \ref{d_EAMc} approximates the
behavior of the Metropolis--Hastings dynamics at the edge typical
configurations.
\begin{prop}
\label{p_ZA}Define a projection map $\Pi^{A}:\mathcal{E}^{A}\rightarrow\mathscr{V}^{A}$
by (cf. Notation \ref{n_barsigma})
\[
\Pi^{A}(\sigma)=\begin{cases}
\overline{\sigma} & \text{if }\sigma\in\mathcal{I}^{A}\;,\\
\sigma & \text{if }\sigma\in\mathcal{O}^{A}\;.
\end{cases}
\]
Then, there exists $C=C(K,\,L,\,M)>0$ such that 
\begin{enumerate}
\item for $\sigma_{1},\,\sigma_{2}\in\mathcal{O}^{A}$, we have
\begin{equation}
\Big|\,\frac{1}{q}e^{-\Gamma\beta}\,r^{A}(\Pi^{A}(\sigma_{1}),\,\Pi^{A}(\sigma_{2}))-\mu_{\beta}(\sigma_{1})\,r_{\beta}(\sigma_{1},\,\sigma_{2})\,\Big|\le Ce^{-(\Gamma+1)\beta}\;,\label{e_ZA1}
\end{equation}
\item for $\sigma_{1}\in\mathcal{O}^{A}$ and $\sigma_{2}\in\overline{\mathcal{I}}^{A}$,
we have
\begin{equation}
\Big|\,\frac{1}{q}e^{-\Gamma\beta}\,r^{A}(\Pi^{A}(\sigma_{1}),\,\Pi^{A}(\sigma_{2}))-\sum_{\zeta\in\mathcal{N}(\sigma_{2})}\mu_{\beta}(\sigma_{1})\,r_{\beta}(\sigma_{1},\,\zeta)\,\Big|\le Ce^{-(\Gamma+1)\beta}\;.\label{e_ZA2}
\end{equation}
\end{enumerate}
\end{prop}

\begin{proof}
(1) Suppose that $\sigma_{1},\,\sigma_{2}\in\mathcal{O}^{A}$. Since
if $\sigma_{1}\not\sim\sigma_{2}$ then the left-hand side of \eqref{e_ZA1}
is $0$, we may assume that $\sigma_{1}\sim\sigma_{2}$. In this case,
$\{\sigma_{1},\,\sigma_{2}\}\in\mathscr{E}^{A}$, and thus
\[
\Big|\,\frac{1}{q}e^{-\Gamma\beta}\,r^{A}(\Pi^{A}(\sigma_{1}),\,\Pi^{A}(\sigma_{2}))-\mu_{\beta}(\sigma_{1})\,r_{\beta}(\sigma_{1},\,\sigma_{2})\,\Big|=\Big|\,\frac{1}{q}e^{-\Gamma\beta}-\frac{1}{Z_{\beta}}e^{-\Gamma\beta}\,\Big|
\]
since $\mu_{\beta}(\sigma_{1})=\mu_{\beta}(\sigma_{2})=\frac{1}{Z_{\beta}}e^{-\Gamma\beta}$
by the definition of $\mathcal{O}^{A}$. By \eqref{e_Zbest}, the
right-hand side of the previous display is $O_{\beta}(e^{-(\Gamma+1)\beta})$.\medskip{}

\noindent (2) Let $\sigma_{1}\in\mathcal{O}^{A}$ and $\sigma_{2}\in\overline{\mathcal{I}}^{A}$.
Similarly, we may assume that $\sigma_{1}\sim\sigma_{2}$. Then, we
can write 
\begin{align*}
 & \Big|\,\frac{1}{q}e^{-\Gamma\beta}\,r^{A}(\Pi^{A}(\sigma_{1}),\,\Pi^{A}(\sigma_{2}))-\sum_{\zeta\in\mathcal{N}(\sigma_{2})}\mu_{\beta}(\sigma_{1})\,r_{\beta}(\sigma_{1},\,\zeta)\,\Big|\\
 & =\Big|\,\frac{1}{q}e^{-\Gamma\beta}\,|\{\zeta\in\mathcal{N}(\sigma_{2}):\zeta\sim\sigma_{1}\}|-\sum_{\zeta\in\mathcal{N}(\sigma_{2}):\,\zeta\sim\sigma_{1}}\min\,\{\mu_{\beta}(\sigma_{1}),\,\mu_{\beta}(\zeta)\}\,\Big|\\
 & =|\{\zeta\in\mathcal{N}(\sigma_{2}):\zeta\sim\sigma_{1}\}|\times\Big|\,\frac{1}{q}e^{-\Gamma\beta}-\frac{1}{Z_{\beta}}e^{-\Gamma\beta}\,\Big|\;,
\end{align*}
since $\min\,\{\mu_{\beta}(\sigma_{1}),\,\mu_{\beta}(\zeta)\}=\mu_{\beta}(\sigma_{1})$
for all $\zeta\in\mathcal{N}(\sigma_{2})$. Again by \eqref{e_Zbest},
the last line is clearly bounded from above by $KL\times O_{\beta}(e^{-\Gamma\beta}\,e^{-\beta})=O_{\beta}(e^{-(\Gamma+1)\beta})$.
This concludes the proof.
\end{proof}

\subsection{\label{sec9.4}Analysis of 3D transition paths}

In this section, we finally define the collection of transition paths
between ground states that appear in Theorem \ref{t_transpath}. 
\begin{defn}[Transition paths]
\label{d_transpath} Write 
\begin{equation}
\mathcal{H}^{A,\,B}=\mathcal{G}^{A,\,B}\cup\widehat{\mathcal{R}}_{[\mathfrak{m}_{K},\,M-\mathfrak{m}_{K}]}^{A,\,B}\;.\label{e_HAB}
\end{equation}
A path $(\omega_{t})_{t=0}^{T}$ is called a\textit{ transition path
between} \textit{$\mathcal{S}(A)$ and $\mathcal{S}(B)$} if 
\begin{align*}
 & \omega_{0}\in\widehat{\mathcal{N}}(\mathcal{S}(A)\,;\,\mathcal{G}^{A,\,B})\;,\;\omega_{T}\in\widehat{\mathcal{N}}(\mathcal{S}(B)\,;\,\mathcal{G}^{A,\,B})\;,\;\text{and}\\
 & \omega_{t}\in\mathcal{H}^{A,\,B}\text{ for all }t\in\llbracket1,\,T-1\rrbracket\;.
\end{align*}
In particular, we have $\omega_{0}\in\mathcal{N}(\mathcal{R}_{\mathfrak{m}_{K}-1}^{A,\,B})$,
$\omega_{1}\in\mathcal{G}_{\mathfrak{m}_{K}-1}^{A,\,B}$, $\omega_{T-1}\in\mathcal{G}_{M-\mathfrak{m}_{K}}^{A,\,B}$,
and $\omega_{T}\in\mathcal{N}(\mathcal{R}_{M-\mathfrak{m}_{K}+1}^{A,\,B})$
by part (1) of Lemma \ref{l_gate}.
\end{defn}

\begin{rem}
The two sets $\widehat{\mathcal{N}}(\mathcal{S}(A)\,;\,\mathcal{G}^{A,\,B})$
and $\widehat{\mathcal{N}}(\mathcal{S}(B)\,;\,\mathcal{G}^{A,\,B})$
are disjoint thanks to part (2) of Proposition \ref{p_Elb}. 
\end{rem}

Now, we characterize all the optimal paths between ground states in
terms of the transition paths.
\begin{thm}
\label{t_trans}Let $(\omega_{t})_{t=0}^{T}$ be a $\Gamma$-path
connecting $\mathcal{S}(A)$ and $\mathcal{S}(B)$. Then, $(\omega_{t})_{t=0}^{T}$
has a transition path between $\mathcal{S}(A)$ and $\mathcal{S}(B)$
as a sub-path.
\end{thm}

\begin{proof}
Let $(\omega_{t})_{t=0}^{T}$ be a $\Gamma$-path connecting $\mathcal{S}(A)$
and $\mathcal{S}(B)$, and define 
\[
T'=\min\,\{t:\omega_{t}\in\widehat{\mathcal{N}}(\mathcal{S}(B)\,;\,\mathcal{G}^{A,\,B})\}\;.
\]
Then, define
\[
t'=\max\,\{t<T':\omega_{t}\in\widehat{\mathcal{N}}(\mathcal{S}(A)\,;\,\mathcal{G}^{A,\,B})\}\;.
\]
We claim that the sub-path $(\omega_{t})_{t=t'}^{T'}$ is a transition
path between $\mathcal{S}(A)$ and $\mathcal{S}(B)$. By part (1)
of Lemma \ref{l_gate}, we have 
\[
\omega_{t'}\in\mathcal{N}(\mathcal{R}_{\mathfrak{m}_{K}-1}^{A,\,B})\;,\;\omega_{t'+1}\in\mathcal{G}_{\mathfrak{m}_{K}-1}^{A,\,B}\;,\;\omega_{T'-1}\in\mathcal{G}_{M-\mathfrak{m}_{K}}^{A,\,B}\;,\;\text{and}\;\omega_{T'}\in\mathcal{N}(\mathcal{R}_{M-\mathfrak{m}_{K}+1}^{A,\,B})\;.
\]
In particular, we get $\omega_{t'+1},\,\omega_{T'-1}\in\mathcal{H}^{A,\,B}$.
To complete the proof of the claim, it suffices to check that, if
$\sigma\in\mathcal{H}^{A,\,B}$ and $\zeta\notin\mathcal{H}^{A,\,B}$
satisfy $\sigma\sim\zeta$ and $H(\zeta)\le\Gamma$, then $\zeta\in\mathcal{N}(\mathcal{R}_{\mathfrak{m}_{K}-1}^{A,\,B})\cup\mathcal{N}(\mathcal{R}_{M-\mathfrak{m}_{K}+1}^{A,\,B})$.
To prove this, let us first assume that $\sigma\in\widehat{\mathcal{R}}_{i}^{a,\,b}$
for some $a\in A$, $b\in B$, and $i\in\llbracket\mathfrak{m}_{K},\,M-\mathfrak{m}_{K}\rrbracket$
(cf. \eqref{e_HAB}). Then, since $\zeta\notin\widehat{\mathcal{R}}_{i}^{a,\,b}$
and $H(\zeta)\le\Gamma$, by the definition of $\widehat{\mathcal{R}}_{i}^{a,\,b}$
we must have $\zeta\in\mathcal{G}^{A,\,B}$, and hence we get a contradiction
to the fact that $\zeta\notin\mathcal{H}^{A,\,B}$. Next, we assume
that $\sigma\in\mathcal{G}_{i}^{A,\,B}$ for some $i\in\llbracket\mathfrak{m}_{K}-1,\,M-\mathfrak{m}_{K}\rrbracket$.
Since $\zeta\in\mathcal{X}\setminus\mathcal{H}^{A,\,B}$, by Lemma
\ref{l_gate}, we have $\zeta\in\mathcal{N}(\mathcal{R}_{\mathfrak{m}_{K}-1}^{A,\,B})\cup\mathcal{N}(\mathcal{R}_{M-\mathfrak{m}_{K}+1}^{A,\,B})$.
This completes the proof.
\end{proof}
Therefore, we can now say that the set $\mathcal{G}^{A,\,B}\cup\widehat{\mathcal{R}}_{[\mathfrak{m}_{K},\,M-\mathfrak{m}_{K}]}^{A,\,B}$
consists of a saddle plateau between $\mathcal{S}(A)$ and $\mathcal{S}(B)$,
which is a huge set of saddle configurations.

Now, we can prove Theorem \ref{t_transpath}.
\begin{proof}[Proof of Theorem \ref{t_transpath}]
 Denote by $\widehat{\tau}$ the hitting time of the set $\{\sigma\in\mathcal{X}:H(\sigma)\ge\Gamma+1\}$.
Then, by the large deviation principle (e.g. \cite[Theorem 3.2]{NZB}),
we have that 
\[
\mathbb{P}_{\mathbf{s}}^{\beta}\,[\widehat{\tau}<e^{\beta(\Gamma+1/2)}]=o_{\beta}(1)\;.
\]
Hence, by part (1) of Theorem \ref{t_LDT results}, we have that $\mathbb{P}_{\mathbf{s}}^{\beta}\,[\tau_{\breve{\mathbf{s}}}<\widehat{\tau}]=1-o_{\beta}(1)$.
Thus, the conclusion of the theorem follows immediately from Theorem
\ref{t_trans}. 
\end{proof}

\section{\label{sec10}Construction of Test Function}

We fix in this section a proper partition $(A,\,B)$ of $S$. The
main purpose of the current section is to construct a test function
$\widetilde{h}=\widetilde{h}_{A,\,B}^{\beta}:\mathcal{X}\rightarrow\mathbb{R}$
that satisfies the two requirements of Proposition \ref{p_H1approx}. 
\begin{notation}
\label{n_consts}Since the partition $(A,\,B)$ is fixed, we simply
write $\mathfrak{b}=\mathfrak{b}(|A|)$, $\mathfrak{e}_{A}=\mathfrak{e}(|A|)$,
$\mathfrak{e}_{B}=\mathfrak{e}(|B|)$, and $\mathfrak{c}=\mathfrak{c}(|A|)$
so that $\mathfrak{c}=\mathfrak{b}+\mathfrak{e}_{A}+\mathfrak{e}_{B}$
throughout the current section.
\end{notation}

\subsection{\label{sec10.1}Construction of test function}

We now define a function $\widetilde{h}:\mathcal{X}\rightarrow\mathbb{R}$
which indeed fulfills all requirements in Proposition \ref{p_H1approx},
as we shall verify later.
\begin{defn}[Test function]
\label{d_testf} We construct the test function $\widetilde{h}$
on $\mathcal{E}^{A,\,B}$, $\mathcal{B}^{A,\,B}$, and $(\mathcal{E}^{A,\,B}\cup\mathcal{B}^{A,\,B})^{c}$
separately. Recall Notation \ref{n_barsigma}.
\begin{enumerate}
\item \textbf{Construction of $\widetilde{h}$ on edge typical configurations
$\mathcal{E}^{A,\,B}=\mathcal{E}^{A}\cup\mathcal{E}^{B}$. }
\begin{itemize}
\item For $\sigma\in\mathcal{E}^{A}$, we recall the decomposition \eqref{e_decEA}
of $\mathcal{E}^{A}$ and define 
\begin{equation}
\widetilde{h}(\sigma)=\begin{cases}
1-\frac{\mathfrak{e}_{A}}{\mathfrak{c}}(1-\mathfrak{h}^{A}(\sigma)) & \text{if }\sigma\in\mathcal{O}^{A}\;,\\
1-\frac{\mathfrak{e}_{A}}{\mathfrak{c}}(1-\mathfrak{h}^{A}(\overline{\sigma})) & \text{if }\sigma\in\mathcal{I}^{A}\;.
\end{cases}\label{e_testf1}
\end{equation}
\item For $\sigma\in\mathcal{E}^{B}$, we similarly define 
\begin{equation}
\widetilde{h}(\sigma)=\begin{cases}
\frac{\mathfrak{e}_{B}}{\mathfrak{c}}(1-\mathfrak{h}^{B}(\sigma)) & \text{if }\sigma\in\mathcal{O}^{B}\;,\\
\frac{\mathfrak{e}_{B}}{\mathfrak{c}}(1-\mathfrak{h}^{B}(\overline{\sigma})) & \text{if }\sigma\in\mathcal{I}^{B}\;.
\end{cases}\label{e_testf2}
\end{equation}
\end{itemize}
\item \textbf{Construction of $\widetilde{h}$ on bulk typical configurations
$\mathcal{B}^{A,\,B}$. }Recall the 2D test function $\widetilde{h}^{\mathrm{2D}}$
explained in Proposition \ref{p_testfcn2}. We define the test function
on each component of the decomposition \eqref{e_BAB} of $\mathcal{B}^{A,\,B}$. 
\begin{itemize}
\item Construction on $\mathcal{G}_{[\mathfrak{m}_{K},\,M-\mathfrak{m}_{K}-1]}^{A,\,B}$:
Let us first fix $P,\,Q\in\mathfrak{S}_{M}$ such that $P\prec Q$
and $|P|\in\llbracket\mathfrak{m}_{K},\,M-\mathfrak{m}_{K}-1\rrbracket$.
Write
\[
\mathcal{G}_{P,\,Q}^{A,\,B}=\bigcup_{a\in A,\,b\in B}\mathcal{G}_{P,\,Q}^{a,\,b}\;.
\]
The test function $\widetilde{h}$ is defined on $\mathcal{G}_{P,\,Q}^{A,\,B}$
by 
\begin{equation}
\widetilde{h}(\sigma)=\frac{1}{\mathfrak{c}}\,\Big[\,\frac{M-\mathfrak{m}_{K}-|P|-(1-\widetilde{h}^{\mathrm{2D}}(\sigma^{(m)}))}{M-2\mathfrak{m}_{K}}\mathfrak{b}+\mathfrak{e}_{B}\,\Big]\;\;\;\;;\;\sigma\in\mathcal{G}_{P,\,Q}^{A,\,B}\;,\label{e_testf3}
\end{equation}
where $\{m\}=Q\setminus P$ so that $\sigma^{(m)}$ is a 2D gateway
configuration between $\mathbf{s}_{a}^{\mathrm{2D}}$ and $\mathbf{s}_{b}^{\mathrm{2D}}$
for some $(a,\,b)\in A\times B$. Since $\mathcal{G}_{[\mathfrak{m}_{K},\,M-\mathfrak{m}_{K}-1]}^{A,\,B}$
can be decomposed into
\[
\mathcal{G}_{[\mathfrak{m}_{K},\,M-\mathfrak{m}_{K}-1]}^{A,\,B}=\bigcup_{i=\mathfrak{m}_{K}}^{M-\mathfrak{m}_{K}-1}\,\bigcup_{P,\,Q\in\mathfrak{S}_{M}:\,P\prec Q\,\text{and}\,|P|=i}\mathcal{G}_{P,\,Q}^{A,\,B}\;,
\]
we can combine the constructions \eqref{e_testf3} to define the test
function on $\mathcal{G}_{[\mathfrak{m}_{K},\,M-\mathfrak{m}_{K}-1]}^{A,\,B}$.
\item Construction on $\widehat{\mathcal{R}}_{i}^{A,\,B}$ for $i\in\llbracket\mathfrak{m}_{K},\,M-\mathfrak{m}_{K}\rrbracket$:
We set
\begin{equation}
\widetilde{h}(\sigma)=\frac{1}{\mathfrak{c}}\,\Big[\,\frac{M-\mathfrak{m}_{K}-i}{M-2\mathfrak{m}_{K}}\mathfrak{b}+\mathfrak{e}_{B}\,\Big]\;\;\;\;;\;\sigma\in\widehat{\mathcal{R}}_{i}^{A,\,B}\;,\label{e_testf4}
\end{equation}
so that the function $\widetilde{h}$ is constant on each $\widehat{\mathcal{R}}_{i}^{A,\,B}$,
$i\in\llbracket\mathfrak{m}_{K},\,M-\mathfrak{m}_{K}\rrbracket$.
\end{itemize}
\item Construction of $\widetilde{h}$ on the remainder set $\mathcal{X}\setminus(\mathcal{E}^{A,\,B}\cup\mathcal{B}^{A,\,B})$:
We define $\widetilde{h}(\sigma)=1$ for all $\sigma\in\mathcal{X}\setminus(\mathcal{E}^{A,\,B}\cup\mathcal{B}^{A,\,B})$.
\end{enumerate}
\end{defn}

\begin{rem}
\label{r_testf}From the definition above, we can readily observe
the following properties of the test function $\widetilde{h}$.
\begin{enumerate}
\item In view of part (1) of Proposition \ref{p_typ}, we should check that
the definitions of $\widetilde{h}$ on $\mathcal{E}^{A,\,B}$ and
$\mathcal{B}^{A,\,B}$ agree on $\widehat{\mathcal{R}}_{\mathfrak{m}_{K}}^{A,\,B}$
and $\widehat{\mathcal{R}}_{M-\mathfrak{m}_{K}}^{A,\,B}$. This can
be verified from \eqref{e_testf1}, \eqref{e_testf2}, \eqref{e_testf4},
and Proposition \ref{p_hAest}. In particular, both definitions imply
that the value of $\widetilde{h}$ on the former set is \textit{constantly}
$\frac{\mathfrak{b}+\mathfrak{e}_{B}}{\mathfrak{c}}$, while the value
of $\widetilde{h}$ on the latter set is \textit{constantly} $\frac{\mathfrak{e}_{B}}{\mathfrak{c}}$.
\item It is obvious that $\widetilde{h}\equiv1$ on $\mathcal{S}(A)$ and
$\widetilde{h}\equiv0$ on $\mathcal{S}(B)$, and moreover we can
readily verify from the definition that $0\le\widetilde{h}\le1$.
\end{enumerate}
The remainder of this section is devoted to proving parts (1) and
(2) of Proposition \ref{p_H1approx}. In the remainder of the current
section, \emph{we assume for simplicity that $K<L<M$.} The other
cases, $K=L<M$, $K<L=M$, or $K=L=M$, can be handled in the exact
same manner.
\end{rem}

\subsection{\label{sec10.2}Dirichlet form of test function}

We first prove that the test function $\widetilde{h}$ satisfies property
(2) of Proposition \ref{p_H1approx}.
\begin{proof}[Proof of part (2) of Proposition \ref{p_H1approx}]
 We divide the Dirichlet form into three parts as 
\[
\Big[\,\sum_{\{\sigma,\,\zeta\}\subseteq\mathcal{E}^{A,B}\cup\mathcal{B}^{A,B}}+\sum_{\substack{\sigma\in\mathcal{E}^{A,B}\cup\mathcal{B}^{A,B}\\
\zeta\in(\mathcal{E}^{A,B}\cup\mathcal{B}^{A,B})^{c}
}
}+\sum_{\{\sigma,\,\zeta\}\subseteq(\mathcal{E}^{A,B}\cup\mathcal{B}^{A,B})^{c}}\,\Big]\,\mu_{\beta}(\sigma)\,r_{\beta}(\sigma,\,\zeta)\,\{\widetilde{h}(\zeta)-\widetilde{h}(\sigma)\}^{2}\;.
\]
We first consider the second summation. Observe first that, by part
(2) of Proposition \ref{p_typ}, we have $\mathcal{E}^{A,\,B}\cup\mathcal{B}^{A,\,B}=\widehat{\mathcal{N}}(\mathcal{S})$
and thus we get $H(\zeta)\ge\Gamma+1$ if $\sigma\sim\zeta$. Hence,
by \eqref{e_muprop} and Theorem \ref{t_Zbest}, we get 
\[
\mu_{\beta}(\sigma)\,r_{\beta}(\sigma,\,\zeta)=\min\,\{\mu_{\beta}(\sigma),\,\mu_{\beta}(\zeta)\}=\mu_{\beta}(\zeta)\le Ce^{-(\Gamma+1)\beta}\;.
\]
From the fact that $0\le\widetilde{h}\le1$ (cf. part (2) of Remark
\ref{r_testf}), we can conclude that the second summation is $o_{\beta}(1)\,e^{-\Gamma\beta}$.
The third summation is trivially $0$ by the definition of the test
function on $(\mathcal{E}^{A,\,B}\cup\mathcal{B}^{A,\,B})^{c}$. Therefore,
it remains to show that 
\begin{equation}
\sum_{\{\sigma,\,\zeta\}\subseteq\mathcal{E}^{A,B}\cup\mathcal{B}^{A,B}}\mu_{\beta}(\sigma)\,r_{\beta}(\sigma,\,\zeta)\,\{\widetilde{h}(\zeta)-\widetilde{h}(\sigma)\}^{2}=\frac{1+o_{\beta}(1)}{\mathfrak{c}q}\,e^{-\Gamma\beta}\;.\label{e_firstsum}
\end{equation}
By part (1) of Proposition \ref{p_typ} and the fact that \textbf{$\widetilde{h}$}
is constant on each $\widehat{\mathcal{R}}_{i}^{A,\,B}$, $i\in\llbracket\mathfrak{m}_{K},\,M-\mathfrak{m}_{K}\rrbracket$
(cf. \eqref{e_testf4}), we can decompose the left-hand side into
\begin{equation}
\Big[\,\sum_{\{\sigma,\,\zeta\}\subseteq\mathcal{B}^{A,B}}+\sum_{\{\sigma,\,\zeta\}\subseteq\mathcal{E}^{A}}+\sum_{\{\sigma,\,\zeta\}\subseteq\mathcal{E}^{B}}\,\Big]\,\mu_{\beta}(\sigma)\,r_{\beta}(\sigma,\,\zeta)\,\{\widetilde{h}(\zeta)-\widetilde{h}(\sigma)\}^{2}\;.\label{e_mceq1}
\end{equation}
Again by the fact that \textbf{$\widetilde{h}$} is constant on each
$\widehat{\mathcal{R}}_{i}^{A,\,B}$, we can express the first summation
as 
\begin{equation}
\sum_{i=\mathfrak{m}_{K}}^{M-\mathfrak{m}_{K}-1}\,\sum_{a\in A,\,b\in B}\,\sum_{\substack{P,\,Q\in\mathfrak{S}_{M}:\\
P\prec Q,\,|P|=i
}
}\,\sum_{\{\sigma,\,\zeta\}\subseteq\mathcal{B}^{a,b}:\,\{\sigma,\,\zeta\}\cap\mathcal{G}_{P,Q}^{a,b}\ne\emptyset}\mu_{\beta}(\sigma)\,r_{\beta}(\sigma,\,\zeta)\,\{\widetilde{h}(\zeta)-\widetilde{h}(\sigma)\}^{2}\;.\label{e_mceq2}
\end{equation}
By \eqref{e_testf3} and \eqref{e_testf4}, Theorem \ref{t_Zbest}
and \eqref{e_Zbest2D}, we can write the last summation as $(1+o_{\beta}(1))$
times
\begin{equation}
\frac{2\mathfrak{b}^{2}\,e^{-2KL\beta}}{q\mathfrak{c}^{2}(M-2\mathfrak{m}_{K})^{2}}\,\sum_{\{\sigma,\,\zeta\}\subseteq\mathcal{B}^{a,b}:\,\{\sigma,\,\zeta\}\cap\mathcal{G}_{P,Q}^{a,b}\ne\emptyset}\mu_{\beta}^{\mathrm{2D}}(\sigma^{(m)})\,r_{\beta}^{\mathrm{2D}}(\sigma^{(m)},\,\zeta^{(m)})\,\{\widetilde{h}^{\mathrm{2D}}(\zeta^{(m)})-\widetilde{h}^{\mathrm{2D}}(\sigma^{(m)})\}^{2}\;,\label{e_mceq3}
\end{equation}
where $\{m\}=Q\setminus P$ and $\sigma^{(m)}$ and $\zeta^{(m)}$
are regarded as 2D Ising configurations. By Proposition \ref{p_testfcn2},
the last summation is $\frac{1+o_{\beta}(1)}{2\kappa^{\mathrm{2D}}}\,e^{-\Gamma^{\mathrm{2D}}\beta}$.
Therefore, display \eqref{e_mceq3} equals
\begin{equation}
\frac{\mathfrak{b}^{2}\,e^{-2KL\beta}}{\mathfrak{c}^{2}(M-2\mathfrak{m}_{K})^{2}}\times\frac{(1+o_{\beta}(1))\,e^{-\Gamma^{\mathrm{2D}}\beta}}{q\kappa^{\mathrm{2D}}}\;.\label{e_mceq4}
\end{equation}
Inserting this to \eqref{e_mceq2} (and recalling \eqref{e_b}), we
get 
\begin{equation}
\sum_{\{\sigma,\,\zeta\}\subseteq\mathcal{B}^{A,B}}\mu_{\beta}(\sigma)\,r_{\beta}(\sigma,\,\zeta)\,\{\widetilde{h}(\zeta)-\widetilde{h}(\sigma)\}^{2}=\frac{\mathfrak{b}+o_{\beta}(1)}{\mathfrak{c}^{2}q}\,e^{-\Gamma\beta}\;.\label{e_mceq5}
\end{equation}

Next, we deal with the second and third summations of \eqref{e_mceq1}.
By \eqref{e_testf1} and Proposition \ref{p_ZA}, the second summation
equals 
\begin{equation}
\frac{e^{-\Gamma\beta}}{q}\,\sum_{\{\sigma,\,\zeta\}\subseteq\mathscr{V}^{A}}\frac{\mathfrak{e}_{A}^{2}\,r^{A}(\sigma,\,\zeta)\,\{\mathfrak{h}^{A}(\zeta)-\mathfrak{h}^{A}(\sigma)\}^{2}}{\mathfrak{c}^{2}}+o_{\beta}(1)\,e^{-\Gamma\beta}=\frac{\mathfrak{e}_{A}+o_{\beta}(1)}{\mathfrak{c}^{2}q}\,e^{-\Gamma\beta}\;.\label{e_mceq6}
\end{equation}
Similarly, we get 
\begin{equation}
\sum_{\{\sigma,\,\zeta\}\subseteq\mathcal{E}^{B}}\mu_{\beta}(\sigma)\,r_{\beta}(\sigma,\,\zeta)\,\{\widetilde{h}(\zeta)-\widetilde{h}(\sigma)\}^{2}=\frac{\mathfrak{e}_{B}+o_{\beta}(1)}{\mathfrak{c}^{2}q}\,e^{-\Gamma\beta}\;.\label{e_mceq7}
\end{equation}
Therefore, by \eqref{e_mceq5}, \eqref{e_mceq6}, and \eqref{e_mceq7},
we can conclude that the left-hand side of \eqref{e_firstsum} is
equal to
\[
(1+o_{\beta}(1))\times\frac{\mathfrak{b}+\mathfrak{e}_{A}+\mathfrak{e}_{B}}{\mathfrak{c}^{2}q}\,e^{-\Gamma\beta}=\frac{1+o_{\beta}(1)}{q\mathfrak{c}}\,e^{-\Gamma\beta}\;.
\]
This concludes the proof.
\end{proof}

\subsection{$H^{1}$-approximation}

Now it remains to prove that the test function $\widetilde{h}$ satisfies
part (1) of Proposition \ref{p_H1approx}. We shall carry this out
in the current section to conclude the proof of Proposition \ref{p_H1approx}. 

\emph{We abbreviate $h=h_{\mathcal{S}(A),\,\mathcal{S}(B)}^{\beta}$
in the remainder of the section.} Then, the next lemma asserts that
the equilibrium potential is nearly constant on each $\mathcal{N}$-neighborhood.
Since this lemma can be proved in the exact same manner as \cite[Lemma 7.8]{KS 2D},
we omit the proof.
\begin{lem}
\label{l_eqpot}For any $\sigma\in\mathcal{X}$ such that $H(\sigma)<\Gamma$,
it holds that
\[
\max_{\zeta\in\mathcal{N}(\sigma)}\,|h(\zeta)-h(\sigma)|=o_{\beta}(1)\;.
\]
\end{lem}

Now we proceed to the proof of \eqref{e_H1approx1}. By \eqref{e_Diriinnprod},
we can write
\begin{align*}
D_{\beta}(h-\widetilde{h}) & =\langle h-\widetilde{h},\,-\mathcal{L}_{\beta}h+\mathcal{L}_{\beta}\widetilde{h}\rangle_{\mu_{\beta}}\\
 & =D_{\beta}(h)+D_{\beta}(\widetilde{h})-\langle h,\,-\mathcal{L}_{\beta}\widetilde{h}\rangle_{\mu_{\beta}}-\langle\widetilde{h},\,-\mathcal{L}_{\beta}h\rangle_{\mu_{\beta}}\;.
\end{align*}
Since $\widetilde{h}\equiv h\equiv1$ on $\mathcal{S}(A)$, $\widetilde{h}\equiv h\equiv0$
on $\mathcal{S}(B)$ (cf. Remark \ref{r_testf}-(2)), and $\mathcal{L}_{\beta}h\equiv0$
on $\mathcal{X}\setminus\mathcal{S}$ (cf. \eqref{e_eqpotsol}), we
have
\[
\langle\widetilde{h},\,-\mathcal{L}_{\beta}h\rangle_{\mu_{\beta}}=\sum_{\mathbf{s}\in\mathcal{S}(A)}\widetilde{h}(\mathbf{s})\,(-\mathcal{L}_{\beta}h)(\mathbf{s})\,\mu_{\beta}(\mathbf{s})=\sum_{\mathbf{s}\in\mathcal{S}(A)}h(\mathbf{s})\,(-\mathcal{L}_{\beta}h)(\mathbf{s})\,\mu_{\beta}(\mathbf{s})=D_{\beta}(h)\;.
\]
By the last two displayed equations, we obtain that
\begin{equation}
D_{\beta}(h-\widetilde{h})=D_{\beta}(\widetilde{h})-\sum_{\sigma\in\mathcal{X}}h(\sigma)\,(-\mathcal{L}_{\beta}\widetilde{h})(\sigma)\,\mu_{\beta}(\sigma)\;.\label{e_Dbhh}
\end{equation}
Therefore, by part (2) of Proposition \ref{p_H1approx} proved in
the previous subsection and the definition of $\mathcal{L}_{\beta}$
(cf. \eqref{e_gen}), we are left to prove that
\begin{equation}
\sum_{\sigma\in\mathcal{X}}h(\sigma)\sum_{\zeta\in\mathcal{X}}\mu_{\beta}(\sigma)\,r_{\beta}(\sigma,\,\zeta)\,[\widetilde{h}(\sigma)-\widetilde{h}(\zeta)]=\frac{1+o_{\beta}(1)}{q\mathfrak{c}}\,e^{-\Gamma\beta}\;.\label{e_WTS}
\end{equation}
For simplicity, we define
\begin{equation}
\psi(\sigma)=\sum_{\zeta\in\mathcal{X}}\mu_{\beta}(\sigma)\,r_{\beta}(\sigma,\,\zeta)\,[\widetilde{h}(\sigma)-\widetilde{h}(\zeta)]\;,\label{e_psidef}
\end{equation}
so that we can rewrite our objective \eqref{e_WTS} as 
\begin{equation}
\sum_{\sigma\in\mathcal{X}}h(\sigma)\,\psi(\sigma)=\frac{1+o_{\beta}(1)}{q\mathfrak{c}}\,e^{-\Gamma\beta}\;.\label{e_WTS2}
\end{equation}
In summary, it suffices to prove \eqref{e_WTS2} to prove that $\widetilde{h}$
satisfies part (1) of Proposition \ref{p_H1approx}. The proof of
\eqref{e_WTS2} is divided into several lemmas. First, we demonstrate
that $\psi(\sigma)$ is negligible if $\sigma$ is not a typical configuration. 
\begin{lem}
\label{l_div1}For every $\sigma\in\mathcal{X}\setminus(\mathcal{E}^{A,\,B}\cup\mathcal{B}^{A,\,B})$
(i.e., $\sigma\notin\widehat{\mathcal{N}}(\mathcal{S})$ by Proposition
\ref{p_typ}), it holds that $\psi(\sigma)=o_{\beta}(e^{-\Gamma\beta})$.
\end{lem}

\begin{proof}
Since $\widetilde{h}\equiv1$ on $\mathcal{X}\setminus(\mathcal{E}^{A,\,B}\cup\mathcal{B}^{A,\,B})$
by part (3) of Definition \ref{d_testf}, it readily holds that
\[
\psi(\sigma)=\sum_{\zeta\in\mathcal{E}^{A,B}\cup\mathcal{B}^{A,B}}\mu_{\beta}(\sigma)\,r_{\beta}(\sigma,\,\zeta)\,[\widetilde{h}(\sigma)-\widetilde{h}(\zeta)]\;.
\]
Then, by \eqref{e_muprop} and part (2) of Proposition \ref{p_typ},
if $\zeta\in\mathcal{E}^{A,\,B}\cup\mathcal{B}^{A,\,B}$ with $\sigma\sim\zeta$
then $H(\sigma)\ge\Gamma+1$, and thus
\[
\mu_{\beta}(\sigma)\,r_{\beta}(\sigma,\,\zeta)=\mu_{\beta}(\sigma)=O_{\beta}(e^{-(\Gamma+1)\beta})\;.
\]
Along with the fact that $0\le\widetilde{h}\le1$, we conclude that
$\psi(\sigma)=O_{\beta}(e^{-(\Gamma+1)\beta})=o_{\beta}(e^{-\Gamma\beta})$.
\end{proof}
We are left to consider $\psi(\sigma)$ for $\sigma\in\mathcal{E}^{A,\,B}\cup\mathcal{B}^{A,\,B}=\widehat{\mathcal{N}}(\mathcal{S})$.
To this end, we decompose as $\psi=\psi_{1}+\psi_{2}$ where
\begin{align}
\psi_{1}(\sigma) & =\sum_{\zeta\in\mathcal{E}^{A,B}\cup\mathcal{B}^{A,B}}\mu_{\beta}(\sigma)\,r_{\beta}(\sigma,\,\zeta)\,[\widetilde{h}(\sigma)-\widetilde{h}(\zeta)]\;,\label{eq:psi1}\\
\psi_{2}(\sigma) & =\sum_{\zeta\notin\mathcal{E}^{A,B}\cup\mathcal{B}^{A,B}}\mu_{\beta}(\sigma)\,r_{\beta}(\sigma,\,\zeta)\,[\widetilde{h}(\sigma)-\widetilde{h}(\zeta)]\;.\nonumber 
\end{align}
In fact, we can show that $\psi_{2}(\sigma)$ is negligible.
\begin{lem}
\label{l_div2}For $\sigma\in\mathcal{E}^{A,\,B}\cup\mathcal{B}^{A,\,B}$,
we have $\psi_{2}(\sigma)=o_{\beta}(e^{-\Gamma\beta})$.
\end{lem}

\begin{proof}
This follows directly by the same argument presented in the proof
of Lemma \ref{l_div1}.
\end{proof}
Now, to estimate $\psi_{1}(\sigma)$, let us first look at the bulk
typical configurations that are not the edge typical configurations.
\begin{lem}
\label{l_div3}We have that $\psi_{1}(\sigma)=o_{\beta}(e^{-\Gamma\beta})$
for all $\sigma\in\mathcal{G}_{[\mathfrak{m}_{K},\,M-\mathfrak{m}_{K}-1]}^{A,\,B}$.
\end{lem}

\begin{proof}
For $\sigma\in\mathcal{G}_{[\mathfrak{m}_{K},\,M-\mathfrak{m}_{K}-1]}^{A,\,B}$,
by definition we can write
\begin{align*}
\psi_{1}(\sigma) & =\sum_{\zeta\in\mathcal{E}^{A,B}\cup\mathcal{B}^{A,B}}\mu_{\beta}(\sigma)\,r_{\beta}(\sigma,\,\zeta)\,[\widetilde{h}(\sigma)-\widetilde{h}(\zeta)]\\
 & =\frac{\mathfrak{b}}{\mathfrak{c}(M-2\mathfrak{m}_{K})}\,\sum_{\zeta\in\mathcal{E}^{A,B}\cup\mathcal{B}^{A,B}}\mu_{\beta}(\sigma)\,r_{\beta}(\sigma,\,\zeta)\,[\widetilde{h}^{\mathrm{2D}}(\sigma^{(m)})-\widetilde{h}^{\mathrm{2D}}(\zeta^{(m)})]
\end{align*}
for some $m\in Q\setminus P$ with $P\prec Q$ (cf. Definition \ref{d_testf}),
where $\sigma^{(m)}$ and $\zeta^{(m)}$ are considered as 2D Ising
configurations. Then, by Theorem \ref{t_Zbest} and \eqref{e_Zbest2D},
the last display equals
\[
\frac{2\mathfrak{b}(1+o_{\beta}(1))}{q\mathfrak{c}(M-2\mathfrak{m}_{K})}\,e^{-2KL\beta}\times\sum_{\zeta\in\mathcal{E}^{A,B}\cup\mathcal{B}^{A,B}}\mu_{\beta}^{\mathrm{2D}}(\sigma)\,r_{\beta}^{\mathrm{2D}}(\sigma,\,\zeta)\,[\widetilde{h}^{\mathrm{2D}}(\sigma^{(m)})-\widetilde{h}^{\mathrm{2D}}(\zeta^{(m)})]\;.
\]
Since $\sigma^{(m)}$ is a 2D gateway configuration, by part (1) of
Proposition \ref{p_testfcn2.2}, the last summation equals $o_{\beta}(e^{-\Gamma^{\mathrm{2D}}\beta})$.
Therefore, we conclude that
\[
\psi_{1}(\sigma)=\frac{2\mathfrak{b}(1+o_{\beta}(1))}{q\mathfrak{c}(M-2\mathfrak{m}_{K})}\,e^{-2KL\beta}\times o_{\beta}(e^{-\Gamma^{\mathrm{2D}}\beta})=o_{\beta}(e^{-\Gamma\beta})\;.
\]
\end{proof}
\begin{lem}
\label{l_div4}For all $i\in\llbracket\mathfrak{m}_{K}+1,\,M-\mathfrak{m}_{K}-1\rrbracket$,
we have that 
\[
\sum_{\sigma\in\widehat{\mathcal{R}}_{i}^{A,B}}\psi_{1}(\sigma)=0\;.
\]
Moreover, $|\psi_{1}(\sigma)|\le Ce^{-\beta\Gamma}$ for all $\sigma\in\widehat{\mathcal{R}}_{i}^{A,\,B}$,
for some fixed constant $C>0$.
\end{lem}

\begin{proof}
Recall from the definition that $\widetilde{h}$ is defined as constant
on each $\widehat{\mathcal{R}}_{i}^{A,\,B}$. Thus, $\psi_{1}(\sigma)=0$
for all $\sigma\in\widehat{\mathcal{R}}_{i}^{A,\,B}\setminus\mathcal{N}(\mathcal{R}_{i}^{A,\,B})$
and it suffices to show that
\[
\sum_{\sigma\in\mathcal{N}(\mathcal{R}_{i}^{A,B})}\psi_{1}(\sigma)=o_{\beta}(e^{-\beta\Gamma})\;.
\]
It remains to prove that for all $a\in A$, $b\in B$, and $P\in\mathfrak{S}_{M}$
such that $|P|\in\llbracket\mathfrak{m}_{K}+1,\,M-\mathfrak{m}_{K}-1\rrbracket$,
\begin{equation}
\sum_{\sigma\in\mathcal{N}(\sigma_{P}^{a,b})}\,\sum_{\zeta\in\mathcal{E}^{A,B}\cup\mathcal{B}^{A,B}}\mu_{\beta}(\sigma)\,r_{\beta}(\sigma,\,\zeta)\,[\widetilde{h}(\sigma)-\widetilde{h}(\zeta)]=o_{\beta}(e^{-\beta\Gamma})\;.\label{e_div4.1}
\end{equation}
Indeed, the left-hand side can be written as 
\begin{align*}
 & \sum_{Q\in\mathfrak{S}_{M}:\,P\prec Q}\,\sum_{\sigma\in\mathcal{C}_{P,Q}^{a,b}\cap\mathcal{N}(\sigma_{P}^{a,b}),\,\zeta\in\mathcal{C}_{P,Q}^{a,b}:\,\zeta\sim\sigma}\mu_{\beta}(\sigma)\,r_{\beta}(\sigma,\,\zeta)\,[\widetilde{h}(\sigma)-\widetilde{h}(\zeta)]\\
 & +\sum_{Q'\in\mathfrak{S}_{M}:\,Q'\prec P}\,\sum_{\sigma\in\mathcal{C}_{Q',P}^{a,b}\cap\mathcal{N}(\sigma_{P}^{a,b}),\,\zeta\in\mathcal{C}_{Q',P}^{a,b}:\,\zeta\sim\sigma}\mu_{\beta}(\sigma)\,r_{\beta}(\sigma,\,\zeta)\,[\widetilde{h}(\sigma)-\widetilde{h}(\zeta)]\;.
\end{align*}
Since we constructed the test function $\widetilde{h}$ between $\sigma_{P}^{a,\,b}$
and $\sigma_{Q}^{a,\,b}$ ($P\prec Q$) and between $\sigma_{Q'}^{a,\,b}$
and $\sigma_{P}^{a,\,b}$ ($Q'\prec P$) in the same manner, the two
summations above cancel out with each other, and thus we obtain \eqref{e_div4.1}.

Finally, for the last statement of the lemma, it suffices to see that
if $\sigma\in\mathcal{N}(\mathcal{R}_{i}^{A,\,B})$ and $\zeta\notin\mathcal{N}(\mathcal{R}_{i}^{A,\,B})$
with $\sigma\sim\zeta$ then $H(\zeta)\ge\Gamma$, and thus
\[
\mu_{\beta}(\sigma)\,r_{\beta}(\sigma,\,\zeta)=\mu_{\beta}(\zeta)\le\frac{1}{Z_{\beta}}e^{-\beta\Gamma}=O_{\beta}(e^{-\beta\Gamma})\;,
\]
where the last equality holds by Theorem \ref{t_Zbest}. This proves
the last statement of the lemma since the number of summands in \eqref{eq:psi1}
does not depend on $\beta$, and since we have $0\le\widetilde{h}\le1$
(cf. Remark \ref{r_testf}). 
\end{proof}
Next, we turn to the edge typical configurations.
\begin{lem}
\label{l_div5}The following statements hold.
\begin{enumerate}
\item If $\sigma\in\mathcal{O}^{A}\setminus(\widehat{\mathcal{R}}_{\mathfrak{m}_{K}}^{A,\,B}\cup\mathcal{N}(\mathcal{S}(A)))$,
we have $\psi_{1}(\sigma)=o_{\beta}(e^{-\Gamma\beta})$.
\item If $\sigma\in\overline{\mathcal{I}}^{A}\setminus(\widehat{\mathcal{R}}_{\mathfrak{m}_{K}}^{A,\,B}\cup\mathcal{N}(\mathcal{S}(A)))$,
it holds that
\begin{equation}
\sum_{\zeta\in\mathcal{N}(\sigma)}\psi_{1}(\zeta)=0\;,\label{e_div5.1}
\end{equation}
and $|\psi_{1}(\zeta)|\le Ce^{-\Gamma\beta}$ for all $\zeta\in\mathcal{N}(\sigma)$
where $C$ is a constant independent of $\beta$.
\end{enumerate}
\end{lem}

\begin{proof}
(1) By part (1) of Proposition \ref{p_ZA} and the definition of \textbf{$\widetilde{h}$},
we calculate
\begin{align*}
\psi_{1}(\sigma) & =\sum_{\zeta\in\mathcal{E}^{A}}\frac{\mathfrak{e}_{A}}{q\mathfrak{c}}\,e^{-\Gamma\beta}\,r^{A}(\sigma,\,\Pi^{A}(\zeta))[\mathfrak{h}^{A}(\sigma)-\mathfrak{h}^{A}(\Pi^{A}(\zeta))]+O_{\beta}(e^{-(\Gamma+1)\beta})\\
 & =\frac{\mathfrak{e}_{A}}{q\mathfrak{c}}\,e^{-\Gamma\beta}\times|\mathscr{V}^{A}|\cdot(-L^{A}\mathfrak{h}^{A})(\sigma)+O_{\beta}(e^{-(\Gamma+1)\beta})\;.
\end{align*}
Since $L^{A}\mathfrak{h}^{A}=0$ on $\mathcal{O}^{A}\setminus(\widehat{\mathcal{R}}_{\mathfrak{m}_{K}}^{A,\,B}\cup\mathcal{N}(\mathcal{S}(A)))$
by the elementary property of equilibrium potentials (cf. \eqref{e_eqpotsol}),
we may conclude that $\psi_{1}(\sigma)=O_{\beta}(e^{-(\Gamma+1)\beta})=o_{\beta}(e^{-\Gamma\beta})$.\medskip{}

\noindent (2) First, we prove \eqref{e_div5.1}. Note that $\widetilde{h}$
is constant on $\mathcal{N}(\sigma)$. Thus,
\[
\sum_{\zeta\in\mathcal{N}(\sigma)}\psi_{1}(\zeta)=\sum_{\zeta\in\mathcal{N}(\sigma)}\sum_{\zeta'\in\mathcal{O}^{A}}\mu_{\beta}(\zeta)\,r_{\beta}(\zeta,\,\zeta')\,[\widetilde{h}(\zeta)-\widetilde{h}(\zeta')]\;.
\]
By part (2) of Proposition \ref{p_ZA} and the definition of $\widetilde{h}$,
this is equal to
\begin{align*}
 & \sum_{\zeta'\in\mathcal{O}^{A}}\frac{\mathfrak{e}_{A}}{q\mathfrak{c}}\,e^{-\Gamma\beta}\times r^{A}(\sigma,\,\zeta')[\mathfrak{h}^{A}(\sigma)-\mathfrak{h}^{A}(\zeta')]+O_{\beta}(e^{-(\Gamma+1)\beta})\\
 & =\frac{\mathfrak{e}_{A}}{q\mathfrak{c}}\,e^{-\Gamma\beta}\times|\mathscr{V}^{A}|\cdot(-L^{A}\mathfrak{h}^{A})(\sigma)+o_{\beta}(e^{-\Gamma\beta})\;.
\end{align*}
Since $L^{A}\mathfrak{h}^{A}=0$ on $\mathcal{\overline{\mathcal{I}}}^{A}\setminus(\widehat{\mathcal{R}}_{\mathfrak{m}_{K}}^{A,\,B}\cup\mathcal{N}(\mathcal{S}(A)))$,
we conclude that $\sum_{\zeta\in\mathcal{N}(\sigma)}\psi_{1}(\zeta)=o_{\beta}(e^{-\Gamma\beta})$
and \eqref{e_div5.1} is now proved.

Finally, for the last statement, the last display implies that for
all $\zeta\in\mathcal{N}(\sigma)$,
\begin{align*}
|\psi_{1}(\zeta)| & =\Big|\sum_{\zeta'\in\mathcal{O}^{A}}\frac{\mathfrak{e}_{A}}{q\mathfrak{c}}\,e^{-\Gamma\beta}\times r^{A}(\sigma,\,\zeta')[\mathfrak{h}^{A}(\sigma)-\mathfrak{h}^{A}(\zeta')]\Big|+O_{\beta}(e^{-(\Gamma+1)\beta})\\
 & \le\sum_{\zeta'\in\mathcal{O}^{A}}\frac{\mathfrak{e}_{A}}{q\mathfrak{c}}\,r^{A}(\sigma,\,\zeta')e^{-\Gamma\beta}+O_{\beta}(e^{-(\Gamma+1)\beta})\le Ce^{-\Gamma\beta}\;,
\end{align*}
where the first inequality holds since $0\le\mathfrak{h}^{A}\le1$
and the second inequality holds by Proposition \ref{p_eest}, \eqref{e_EAMc},
and the fact that the number of such $\zeta'\in\mathcal{O}^{A}$ with
$\sigma\sim\zeta'$ does not depend on $\beta$. This concludes the
proof.
\end{proof}
\begin{lem}
\label{l_div6}It holds that 
\[
\sum_{\sigma\in\widehat{\mathcal{R}}_{\mathfrak{m}_{K}}^{A,B}}\psi_{1}(\sigma)=o_{\beta}(e^{-\Gamma\beta})\;,
\]
and that $|\psi_{1}(\sigma)|\le Ce^{-\Gamma\beta}$ for all $\sigma\in\widehat{\mathcal{R}}_{\mathfrak{m}_{K}}^{A,\,B}$
where $C$ is a constant independent of $\beta$.
\end{lem}

\begin{proof}
First, we consider the first statement. Proposition \ref{p_hAest}
and the definition of \textbf{$\widetilde{h}$} on $\widehat{\mathcal{R}}_{\mathfrak{m}_{K}}^{A,\,B}$
imply that $\psi_{1}(\sigma)=0$ for all $\sigma\in\widehat{\mathcal{R}}_{\mathfrak{m}_{K}}^{A,\,B}\setminus\mathcal{N}(\mathcal{R}_{\mathfrak{m}_{K}}^{A,\,B})$.
Hence, it suffices to prove that 
\begin{equation}
\sum_{\sigma\in\mathcal{N}(\mathcal{R}_{\mathfrak{m}_{K}}^{A,B})}\psi_{1}(\sigma)=o_{\beta}(e^{-\Gamma\beta})\;.\label{e_div6.1}
\end{equation}
Since $\widetilde{h}$ is constant on $\mathcal{N}(\mathcal{R}_{\mathfrak{m}_{K}}^{A,\,B})$,
the left-hand side can be decomposed into 
\begin{equation}
\Big[\,\sum_{\sigma\in\mathcal{N}(\mathcal{R}_{\mathfrak{m}_{K}}^{A,B}),\,\zeta\in\mathcal{E}^{A}}+\sum_{\sigma\in\mathcal{N}(\mathcal{R}_{\mathfrak{m}_{K}}^{A,B}),\,\zeta\in\mathcal{B}^{A,B}}\,\Big]\,\mu_{\beta}(\sigma)\,r_{\beta}(\sigma,\,\zeta)\,[\widetilde{h}(\sigma)-\widetilde{h}(\zeta)]\;.\label{e_div6.2}
\end{equation}

Let us analyze the first summation of \eqref{e_div6.2}. By part (2)
of Proposition \ref{p_ZA}, this equals
\[
\sum_{\sigma\in\mathcal{R}_{\mathfrak{m}_{K}}^{A,B}}\sum_{\zeta\in\mathcal{O}^{A}}\frac{\mathfrak{e}_{A}}{q\mathfrak{c}}\,e^{-\Gamma\beta}\times r^{A}(\sigma,\,\zeta)\,[\mathfrak{h}^{A}(\sigma)-\mathfrak{h}^{A}(\zeta)]+O_{\beta}(e^{-(\Gamma+1)\beta})\;.
\]
By the property of capacities (e.g., \cite[(7.1.39)]{BdenH meta})
and Proposition \ref{p_hAest}, we have
\begin{equation}
\mathfrak{e}_{A}^{-1}=|\mathscr{V}^{A}|\,\mathrm{cap}^{A}(\mathcal{S}(A),\,\mathcal{R}_{\mathfrak{m}_{K}}^{A,\,B})=-\sum_{\sigma\in\mathcal{R}_{\mathfrak{m}_{K}}^{A,B}}\,\sum_{\zeta:\,\{\sigma,\,\zeta\}\in\mathscr{E}^{A}}r^{A}(\sigma,\,\zeta)\,\{\mathfrak{h}^{A}(\sigma)-\mathfrak{h}^{A}(\zeta)\}\;.\label{e_div6.3}
\end{equation}
Summing up, we obtain
\begin{equation}
\sum_{\sigma\in\mathcal{N}(\mathcal{R}_{\mathfrak{m}_{K}}^{A,B}),\,\zeta\in\mathcal{E}^{A}}\psi_{1}(\sigma)=-\frac{1}{q\mathfrak{c}}\,e^{-\Gamma\beta}+o_{\beta}(e^{-\Gamma\beta})\;.\label{e_div6.4}
\end{equation}

Next, we analyze the second summation of \eqref{e_div6.2}.
\begin{equation}
\sum_{a\in A,\,b\in B}\,\sum_{P,\,Q\in\mathfrak{S}_{M}:\,P\prec Q,\,|P|=\mathfrak{m}_{K}}\,\sum_{\sigma\in\mathcal{N}(\sigma_{P,Q}^{a,b}),\,\zeta\in\mathcal{B}^{a,b}}\mu_{\beta}(\sigma)\,r_{\beta}(\sigma,\,\zeta)\,[\widetilde{h}(\sigma)-\widetilde{h}(\zeta)]\;.\label{e_div6.5}
\end{equation}
By Theorem \ref{t_Zbest}, \eqref{e_Zbest2D}, and part (2) of Proposition
\ref{p_testfcn2.2}, this becomes (recall the 2D constant $\kappa^{\mathrm{2D}}$
from \eqref{e_kappa2Ddef}) 
\begin{equation}
|A||B|\times2M\times\frac{2\mathfrak{b}(1+o_{\beta}(1))}{q\mathfrak{c}(M-2\mathfrak{m}_{K})}\,e^{-2KL\beta}\times\frac{1}{2\kappa^{\mathrm{2D}}}\,e^{-\Gamma^{\mathrm{2D}}\beta}=\frac{1+o_{\beta}(1)}{q\mathfrak{c}}\,e^{-\Gamma\beta}\;,\label{e_div6.6}
\end{equation}
where the identity follows from the definition of $\mathfrak{b}$
in \eqref{e_b}. Combining this with \eqref{e_div6.2} and \eqref{e_div6.4},
we can prove the first statement of the lemma.

For the second statement, from the discussion before \eqref{e_div6.1}
it is inferred that we only need to prove for $\sigma\in\mathcal{N}(\mathcal{R}_{\mathfrak{m}_{K}}^{A,\,B})$.
For such $\sigma\in\mathcal{N}(\mathcal{R}_{\mathfrak{m}_{K}}^{A,\,B})$,
the previous proof implies that
\[
\psi_{1}(\sigma)=\Big[\,\sum_{\zeta\in\mathcal{E}^{A}}+\sum_{\zeta\in\mathcal{B}^{A,B}}\,\Big]\,\mu_{\beta}(\sigma)\,r_{\beta}(\sigma,\,\zeta)\,[\mathfrak{h}^{A}(\sigma)-\mathfrak{h}^{A}(\zeta)]+O_{\beta}(e^{-(\Gamma+1)\beta})\;,
\]
where we used the fact that $0\le\widetilde{h}\le1$. By \eqref{e_div6.3}
and Proposition \ref{p_eest}, the first summation in the right-hand
side is bounded by $Ce^{-\Gamma\beta}$. By \eqref{e_div6.5} and
\eqref{e_div6.6}, the second summation in the right-hand side is
also bounded by $Ce^{-\Gamma\beta}$. Therefore, we conclude the proof
of the second statement.
\end{proof}
\begin{lem}
\label{l_div7}It holds that 
\[
\sum_{\sigma\in\mathcal{N}(\mathcal{S}(A))}\psi_{1}(\sigma)=\frac{1+o_{\beta}(1)}{q\mathfrak{c}}\,e^{-\Gamma\beta}\;\;\;\;\text{and}\;\;\;\;\sum_{\sigma\in\mathcal{N}(\mathcal{S}(B))}\psi_{1}(\sigma)=o_{\beta}(e^{-\Gamma\beta})\;.
\]
Moreover, it holds that $|\psi_{1}(\sigma)|\le Ce^{-\Gamma\beta}$
for all $\sigma\in\mathcal{N}(\mathcal{S}(A))\cup\mathcal{N}(\mathcal{S}(B))$
where $C$ is a constant independent of $\beta$.
\end{lem}

\begin{proof}
We concentrate on the claim for $\mathcal{N}(\mathcal{S}(A))$, since
the corresponding claim for $\mathcal{N}(\mathcal{S}(B))$ can be
proved in the exact same way.

By the property of capacities (e.g., \cite[(7.1.39)]{BdenH meta})
as above, we can write that 
\[
\mathfrak{e}_{A}^{-1}=|\mathscr{V}^{A}|\,\mathrm{cap}^{A}(\mathcal{S}(A),\,\mathcal{R}_{\mathfrak{m}_{K}}^{A,\,B})=\sum_{\sigma\in\mathcal{S}(A)}\,\sum_{\zeta:\,\{\sigma,\,\zeta\}\in\mathscr{E}^{A}}r^{A}(\sigma,\,\zeta)\,\{\mathfrak{h}^{A}(\sigma)-\mathfrak{h}^{A}(\zeta)\}\;.
\]
Therefore, by the definition of $\widetilde{h}$ and part (2) of Proposition
\ref{p_ZA}, $\sum_{\sigma\in\mathcal{N}(\mathcal{S}(A))}\psi_{1}(\sigma)$
equals
\begin{align*}
 & \sum_{\sigma\in\mathcal{S}(A)}\,\sum_{\zeta:\,\{\sigma,\,\zeta\}\in\mathscr{E}^{A}}\frac{\mathfrak{e}_{A}}{q\mathfrak{c}}\,e^{-\Gamma\beta}\cdot r^{A}(\sigma,\,\zeta)[\mathfrak{h}^{A}(\sigma)-\mathfrak{h}^{A}(\zeta)]+O_{\beta}(e^{-(\Gamma+1)\beta})=\frac{1}{q\mathfrak{c}}\,e^{-\Gamma\beta}+o_{\beta}(e^{-\Gamma\beta})\;.
\end{align*}
This proves the first statement. As before, the fact that $|\psi_{1}|\le Ce^{-\Gamma\beta}$
on $\mathcal{N}(\mathcal{S}(A))$ is straightforward from the observations
made in the proof.
\end{proof}
Finally, we present a proof of Proposition \ref{p_H1approx} by combining
all computations above. 
\begin{proof}[Proof of Proposition \ref{p_H1approx}]
 It remains to prove that $\widetilde{h}$ satisfies part (1) since
we already verified in the previous subsection that it satisfies part
(2). By the discussion at the beginning of the subsection, it suffices
to prove \eqref{e_WTS2}. By the definition of $\psi$ given in \eqref{e_psidef}
and the series of Lemmas \ref{l_div1}-\ref{l_div7}, and the fact
that $0\le h\le1$, we have
\[
\sum_{\sigma\in\mathcal{X}}h(\sigma)\,\psi(\sigma)=\sum_{\sigma\in\mathcal{N}(\mathcal{S}(A))}h(\sigma)\,\psi_{1}(\sigma)+o_{\beta}(e^{-\Gamma\beta})=\sum_{\sigma\in\mathcal{N}(\mathcal{S}(A))}\psi_{1}(\sigma)+o_{\beta}(e^{-\Gamma\beta})\;,
\]
where the second identity follows from Lemmas \ref{l_eqpot}. Thus,
by applying Lemma \ref{l_div7}, we can complete the proof of \eqref{e_WTS2}.
\end{proof}

\section{\label{sec11}Remarks on Open boundary Condition}

Thus far, we have only considered the models under periodic boundary
conditions. In this section, we consider the same models under open
boundary conditions. The proofs for the open boundary case differ
slightly to those of the periodic case; however, the fundamentals
of the proofs are essentially identical. Hence, we do not repeat the
detail but focus solely on the technical points producing the different
forms of the main results.

\subsubsection*{Energy Barrier}

We start by explaining that for the open boundary case, the energy
barrier is given by 
\begin{equation}
\Gamma=KL+K+1\;.\label{e_Ebopen}
\end{equation}
One can observe that the canonical path explained in Figure \ref{fig6.3}
becomes an optimal path (note that we should start from a corner of
box in this case) with height $KL+K+1$ between ground states. This
proves that the energy barrier $\Gamma$ is at most $KL+K+1$. Hence,
it remains to prove the corresponding lower bound, i.e., of the fact
that $\Gamma\ge KL+K+1$. Rigorous proof of this has been developed
in \cite{NZ} for the 2D model, and the same argument also applies
to the 3D model as well using the arguments given in Section \ref{sec8}.

\subsubsection*{Sub-exponential prefactor}

As mentioned earlier, the large deviation-type results (Theorems \ref{t_LDT results}
and \ref{t_transpath}) hold under open boundary conditions without
modification, except for the value of $\Gamma$. On the other hand,
for the precise estimates (Theorems \ref{t_EK} and \ref{t_MC}),
the prefactor $\kappa$ must be appropriately modified.

For simplicity, we assume that $q=2$ and analyze the transition from
$\mathbf{s}_{1}$ to $\mathbf{s}_{2}$. To heuristically investigate
the speed of this transition in the open boundary case via a comparison
to the periodic one, it suffices to check the bulk part of the transition,
because the edge part is negligible (as $K\rightarrow\infty$) as
in the periodic boundary case. The bulk transition must start from
a configuration filled with $\mathfrak{m}_{K}$ lines of spin $2$
at either the bottom or top of the lattice box $\Lambda$. In the
periodic case, there are $M$ choices for these starting clusters
(of spins $2$) of size $KL\times\mathfrak{m}_{K}$; thus, we can
observe that the speed of the transition is slowed by a factor of
$M/2$ under this restriction. Now, let us suppose that we are at
a configuration such that several floors of spin $2$ are located
at the bottom of the lattice, as in Figure \ref{fig6.3}. When we
expand this cluster of spin $2$ in the periodic case, there are $2$
(namely, up and down) possible choices for the next floor to be filled;
on the other hand, there is only one (namely, up) possible choice
in the open boundary case. This further slows down the transition
by a factor of $2$. Next, when we expand the floor at the top of
the cluster of spin $2$, we may again look at the bulk part of the
spin updates (cf. Definition \ref{d_typ2}). Thus, we suppose that
there are two lines filled with spin $2$ on that floor. There are
$L$ possible choices of the location in the periodic case, but just
two possible choices in the open case. Thus, this gives us a factor
of $L/2$. Moreover, we may choose one of two directions of growth
of lines in the periodic case, which gives us additional factor of
$2$. Finally, there are $K$ possible ways to form a protuberance
in the periodic case; however, we now have only two (at the corners)
possible choices. This further slows down the transition by a factor
of $K/2$. Once the protuberance has been formed, we have only one
direction in which to expand it, whereas we have two directions in
the periodic case. This slows down the transition by a factor of $2$.
Summing up, the transition on the bulk is slowed by a factor of 
\[
\frac{M}{2}\times2\times\frac{L}{2}\times2\times\frac{K}{2}\times2=KLM\;.
\]
Turning this into a rigorous argument (via the same logic applied
to the periodic case), we obtain the following Eyring--Kramers law
with a modified (compared to the periodic case) prefactor. Recall
that we assumed $K\le L\le M$.
\begin{thm}
\label{t_EK-1}Suppose that we impose open boundary conditions on
the model. Then, there exists a constant $\kappa'=\kappa'(K,\,L,\,M)>0$
such that, for all $\mathbf{s},\,\mathbf{s}'\in\mathcal{S},$
\[
\mathbb{E}_{\mathbf{s}}^{\beta}\,[\tau_{\breve{\mathbf{s}}}]=(1+o_{\beta}(1))\,\frac{\kappa'}{q-1}\,e^{\Gamma\beta}\;\;\;\;\text{and}\;\;\;\;\mathbb{E}_{\mathbf{s}}^{\beta}\,[\tau_{\mathbf{s}'}]=(1+o_{\beta}(1))\,\kappa'\,e^{\Gamma\beta}\;.
\]
Moreover, the constant $\kappa'$ satisfies
\begin{align}
\lim_{K\rightarrow\infty}KLM\cdot\kappa'(K,\,L,\,M) & =\begin{cases}
1/8 & \text{if }K<L<M\;,\\
1/16 & \text{if }K=L<M\text{ or }K<L=M\;,\\
1/48 & \text{if }K=L=M\;.
\end{cases}\label{e_kappaest-1}
\end{align}
The constant $\kappa'$ can be defined in terms of new bulk and edge
constants $\mathfrak{b}'(n)$ and $\mathfrak{e}'(n)$, in the exact
same manner as done in Section \ref{sec3.1}.
\end{thm}

Then, Theorem \ref{t_MC} also holds for open boundary conditions
with modified limiting Markov chain $X'(\cdot)$ with rate $r_{X'}(\mathbf{s},\,\mathbf{s}')=(\kappa')^{-1}$
for all $\mathbf{s},\,\mathbf{s}'\in\mathcal{S}$. 
\begin{acknowledgement*}
SK was supported by NRF-2019-Fostering Core Leaders of the Future
Basic Science Program/Global Ph.D. Fellowship Program and the National
Research Foundation of Korea (NRF) grant funded by the Korean government
(MSIT) (No. 2022R1F1A106366811, No. 2022R1A5A6000840). IS was supported
by the National Research Foundation of Korea (NRF) grant funded by
the Korean government (MSIT) (No. 2022R1F1A106366811, 2022R1A5A6000840).
\end{acknowledgement*}

\end{document}